\documentclass[a4paper,reqno]{amsart}
\usepackage[utf8]{inputenc}
\usepackage{amsmath, amssymb, amsthm, amsfonts, mathtools, array, xfrac, float, multirow, indentfirst, url, enumerate, comment, afterpage, pifont, makecell, xcolor, hyperref, bm}
\usepackage{mathrsfs}
\usepackage{yhmath}
\usepackage{booktabs,url}
\usepackage{color}

\allowdisplaybreaks

\usepackage[left=2.5cm, right=2.5cm, bottom=3cm]{geometry}
\usepackage[font=footnotesize,labelfont=bf]{caption}

\usepackage{tikz}
\usetikzlibrary{patterns, patterns.meta}\usepackage{comment}
\usetikzlibrary{positioning, arrows, decorations.markings, calc}

\newtheorem{theorem}{Theorem}
\newtheorem{lemma}[theorem]{Lemma}

\newtheorem{corollary}{Corollary}

\theoremstyle{definition}
\newtheorem{remark}{Remark}

\newcommand{\NN}{\mathbb{N}}

\newcommand{\R}{\mathbb{R}}
\newcommand{\C}{\mathbb{C}}

\newcommand{\de}[0]{\mathrel{\mathop:}=}

\newcommand{\ie}[0]{\mathrm{i}}
\newcommand{\dif}[1]{\mathrm{d}#1}

\DeclareMathOperator*{\Res}{Res}

\begin{document}

\title{An explicit form of Ingham's zero density estimate}
\author[Shashi Chourasiya]{Shashi Chourasiya}
\address{
    SC: School of Science, University of New South Wales (Canberra), Northcott Drive, Campbell, ACT 2600, Australia
}
\email{s.chourasiya@unsw.edu.au}
\author[Aleksander Simoni\v{c} ]{Aleksander Simoni\v{c} }
\address{
    AS: Faculty of Mathematics, Natural Sciences and Information Technologies, University of Primorska, Glagolja\v{s}ka 8, 6000 Koper, Slovenia
}
\email{aleksander.simonic@famnit.upr.si}

\subjclass[2020]{11M06, 11M26}
\keywords{Riemann zeta-function, Zero density estimates, Moments of the Riemann zeta-function}
\begin{abstract}
Ingham (1940) proved that $N(\sigma,T)\ll T^{3(1-\sigma)/(2-\sigma)}\log^{5}{T}$, where $N(\sigma,T)$ counts the number of the non-trivial zeros $\rho$ of the Riemann zeta-function with $\Re\{\rho\}\geq\sigma\geq 1/2$ and $0<\Im\{\rho\}\leq T$. We provide an explicit version of this result with the exponent $(7-5\sigma)/(2-\sigma)$ of the logarithmic factor. In addition, we also provide an explicit estimate with asymptotically correct main term for the fourth power moment of the Riemann zeta-function on the critical line.
\end{abstract}

\maketitle
\thispagestyle{empty}

\section{Introduction}
Let $\zeta(s)$ be the Riemann zeta-function and let $\rho=\beta+\ie\gamma$ be its non-trivial zero, that is, $\zeta(\rho)=0$ with $0\leq\beta\leq 1$ and $\gamma\in\R$.  Let $N(T)$ be the number (counted with multiplicities) of all non-trivial zeros with $0<\gamma\leq T$, while for $\sigma\geq 0$ and $T>0$ define 
$N(\sigma,T)\de \# \left\{ \rho \colon \zeta(\rho)=0, \beta\geq\sigma, 0<\gamma\leq T\right\}$, where the zeros are again counted with multiplicities. Trivially, the function $N(\sigma,T)$ is decreasing in $\sigma$ and increasing in $T$, $N(\sigma,T)\leq N(T)\sim\frac{1}{2\pi}T\log{T}$ for $0\leq\sigma\leq 1/2$, $N(\sigma,T)\leq \frac{1}{2}N(T)$ for $1/2<\sigma<1$ and $N(\sigma,T)=0$ for $\sigma\geq 1$, while the Riemann Hypothesis (RH) asserts that $N(\sigma,T)=N(T)$ for $0\leq\sigma\leq1/2$ and $N(\sigma,T)=0$ for $\sigma>1/2$. Since RH has been verified up to $H_{\mathrm{RH}}\de 3\cdot 10^{12}$ by Platt and Trudgian~\cite{PT21}, we have $N(\sigma,T)=0$ for $\sigma>1/2$ and $T \leq H_{\mathrm{RH}}$. Consequently, we can restrict ourselves to the range of $1/2\leq\sigma\leq1$ and $T\geq H_{\mathrm{RH}}$ throughout this paper.

Loosely speaking, a non-trivial estimate on $N(\sigma,T)$ is known as \emph{zero density estimate}. Estimates of the form\footnote{As usual, $A(\sigma)$ is defined as ``infimum of all such exponents'', see~\cite[Definition~37]{TaoTrudgianYang} for a precise definition.} 
\begin{equation}
\label{eq:GeneralForm}
N(\sigma,T)\ll T^{A(\sigma)(1-\sigma)}\log^{B}{T}\quad \textrm{or} \quad N(\sigma,T)\ll_{\varepsilon}T^{A(\sigma)(1-\sigma)+\varepsilon} 
\end{equation}
with $A(\sigma)\leq1/(1-\sigma)$, $B\geq0$, $\varepsilon>0$ and the implied constants being uniform in $\sigma$ and $T$, are of great interest in the distribution theory of prime numbers~\cite{Ivic,StarichkovaZD}. The Lindel\"{o}f Hypothesis (LH), which is believed to be a weaker statement than RH, implies $A(\sigma)\leq 2$ in the expression on the right-hand side of~\eqref{eq:GeneralForm} for all\footnote{It is known that LH guarantees even $A(\sigma)=0$ for $\sigma \geq \frac{3}{4}+\delta$ and $\delta>0$ fixed.} $\sigma\geq 1/2$, and this is known as the \emph{density hypothesis}~\cite[p.~47]{Ivic}. Bourgain~\cite{jean_bourgain_2000} showed that the density hypothesis holds for $\sigma\geq 25/32=0.78125$, and substantial progress has been made in obtaining much better upper bounds for $A(\sigma)$ in this region, see~\cite[Table 2 on p.~33]{TaoTrudgianYang} for the current best results. On the other hand, region $1/2\leq \sigma<19/25=0.76$ is dominated by two results, namely that
\begin{equation*}
A(\sigma)\leq \min\left\{\frac{3}{2-\sigma},\frac{15}{5\sigma+3}\right\},
\end{equation*}
which is due to Ingham~\cite{In40} and Guth--Maynard~\cite{guth2024new}, respectively. Note that $3/(2-\sigma)\leq 15/(5\sigma+3)$ for $\sigma\in[1/2,0.7]$. This means that, apart from Selberg's zero density estimate\footnote{His result is sharper than~\cite{In40} for $\sigma$ close to the critical line, e.g., when $0\leq\sigma-1/2\ll \log{\log{T}}/\log{T}$.}~\cite{Sel46} (see also~\cite{BaluyotPhD,SimonicEZDE}), Huxley's refinement~\cite[Chapter~23]{HuxleyBook}
\begin{equation}
\label{eq:Ingham2ndZD}
N(\sigma,T) \ll \left(\sigma-\frac{1}{2}\right)^{-\frac{1}{3}}T^{\frac{3(1-\sigma)}{2-\sigma}}\left(\log{T}\right)^{\frac{6}{2-\sigma}}
\end{equation}
of Ingham's \emph{later}\footnote{This is to distinguish between $N(\sigma,T)\ll T^{3(1-\sigma)/(2-\sigma)}\log^{5}{T}$ and Ingham's earlier zero density estimates, see~\cite{ingham1937difference}, such as $N(\sigma,T)\ll T^{8(1-\sigma)/3}\log^{5}{T}$.} zero density estimate is currently the best known result on $N(\sigma,T)$ for $1/2<\sigma\leq 0.7$.

In order to obtain effective and unconditional results that are related to the distribution of prime numbers, explicit versions of~\eqref{eq:GeneralForm} are often required, see~\cite{fiori2023sharper,johnston2023some,CullyJohnstonRvM,CullyJohnstonRvMII}. Recently, several relevant results have appeared in the literature, see~\cite{Kadiri2013,KADIRI201822,bellotti2024explicit,chourasiya2024explicit} for example. The aim of this paper is to provide the first explicit version of~\eqref{eq:Ingham2ndZD}, see Corollary~\ref{cor:mainZD}, where we also made an improvement over the logarithmic factor.

\begin{corollary}
\label{cor:mainZD}
Let $T\geq H_{\mathrm{RH}}$. Then
\begin{equation}
\label{eq:mainZD}
N(\sigma,T) \leq \mathfrak{B}_1 T^{\frac{3(1-\sigma)}{2-\sigma}}\left(\log{T}\right)^{\frac{7-5\sigma}{2-\sigma}} + \mathfrak{B}_2 \log^{2}{T} + \mathfrak{B}_3 \log{T}
\end{equation}
for $\sigma_1\leq\sigma\leq\sigma_2$, where the values for $\mathfrak{B}_1$, $\mathfrak{B}_2$ and $\mathfrak{B}_3$ are given by Table~\ref{tab:CorMainZD}.
\end{corollary}

Note that $(7-5\sigma)/(2-\sigma)$ is a decreasing function in $\sigma\in[1/2,1]$ that has the range $[2,3]$ there. The intervals $[\sigma_1,\sigma_2]$ in Table~\ref{tab:CorMainZD} are elements of the set $\left\{[n/32,(n+1)/32]\colon n\in\{16,\ldots,31\}\right\}$. This choice is justified by the fact that~\eqref{eq:mainZD} is asymptotically better than $N(\sigma,T)\ll T^{8(1-\sigma)/3}\log^{5}{T}$ for $\sigma\in[1/2,0.875]$ and the latter zero density estimate\footnote{We are comparing our result to an improved version (see Section~\ref{sec:appendix}) of~\cite[Theorem~1.1 and Table~1]{KADIRI201822}, see also~\cite[Lemma~7.1]{CullyJohnstonRvM}, where a variant of such a zero density estimate was proved.} is asymptotically better than the trivial one for $\sigma\in(0.625,1)$. Corollary~\ref{cor:mainZD} implies that 
\begin{gather*}
N(\sigma,T) \leq 8.185\cdot T^{\frac{3(1-\sigma)}{2-\sigma}}\left(\log{T}\right)^{\frac{7-5\sigma}{2-\sigma}} + 9.461\log^{2}{T} + 167.8\log{T}, \quad \sigma\in[0.500,0.625],\; T\geq 3\cdot10^{12}, \\
N(\sigma,T) \leq 21.77\cdot T^{\frac{3(1-\sigma)}{2-\sigma}}\left(\log{T}\right)^{\frac{7-5\sigma}{2-\sigma}} + 8.290\log^{2}{T} + 147.0\log{T}, \quad \sigma\in[0.625,0.875],\; T\geq 3\cdot10^{12}.
\end{gather*}
The above estimates improve~\cite[Theorem~1.1 and Table~1]{chourasiya2024explicit} for $\sigma\in[0.5, 0.667]$, and the latter improves~\eqref{eq:KLN} for $\sigma\in[0.625,0.8575]$ while it is also better for $\sigma\in(0.8575,0.875]$ and $T\geq 10^{61.3}$.

\begin{table}[h]
\centering
\footnotesize{
\begin{tabular}{lllllllll} 
\toprule
$\sigma_1$ & $\sigma_2$ & $\mathfrak{B}_1$ & $\mathfrak{B}_2$ & $\mathfrak{B}_3$ & $d$ & $\kappa$ & $\delta$ & $\log_{10}A_0$ \\ 
\midrule
$0.50000$ & $0.53125$ & $5.069$ & $9.461$ & $167.8$ & $0.28$ & $0.34$ & $0.012$ & $11.25$ \\
$0.53125$ & $0.56250$ & $6.046$ & $9.168$ & $162.6$ & $0.28$ & $0.32$ & $0.025$ & $11.26$ \\
$0.56250$ & $0.59375$ & $7.077$ & $8.875$ & $157.4$ & $0.28$ & $0.30$ & $0.038$ & $11.27$ \\
$0.59375$ & $0.62500$ & $8.185$ & $8.582$ & $152.2$ & $0.28$ & $0.29$ & $0.052$ & $11.27$ \\
$0.62500$ & $0.65625$ & $9.387$ & $8.290$ & $147.0$ & $0.28$ & $0.27$ & $0.066$ & $11.28$ \\
$0.65625$ & $0.68750$ & $10.70$ & $7.997$ & $141.9$ & $0.28$ & $0.25$ & $0.081$ & $11.29$ \\
$0.68750$ & $0.71875$ & $12.11$ & $7.704$ & $136.7$ & $0.28$ & $0.23$ & $0.097$ & $11.30$ \\
$0.71875$ & $0.75000$ & $13.66$ & $7.411$ & $131.5$ & $0.28$ & $0.21$ & $0.114$ & $11.31$ \\
$0.75000$ & $0.78125$ & $15.35$ & $7.118$ & $126.3$ & $0.28$ & $0.20$ & $0.132$ & $11.32$ \\
$0.78125$ & $0.81250$ & $17.21$ & $6.826$ & $121.1$ & $0.28$ & $0.18$ & $0.150$ & $11.33$ \\
$0.81250$ & $0.84375$ & $19.31$ & $6.533$ & $115.9$ & $0.28$ & $0.16$ & $0.170$ & $11.35$ \\
$0.84375$ & $0.87500$ & $21.77$ & $6.240$ & $110.7$ & $0.28$ & $0.14$ & $0.192$ & $11.36$ \\
$0.87500$ & $0.90625$ & $24.83$ & $5.745$ & $101.9$ & $0.29$ & $0.13$ & $0.215$ & $11.37$ \\
$0.90625$ & $0.93750$ & $29.01$ & $5.463$ & $96.87$ & $0.29$ & $0.11$ & $0.241$ & $11.40$ \\
$0.93750$ & $0.96875$ & $35.35$ & $5.180$ & $91.86$ & $0.29$ & $0.09$ & $0.270$ & $11.42$ \\
$0.96875$ & $1.00000$ & $46.06$ & $4.897$ & $86.84$ & $0.29$ & $0.07$ & $0.303$ & $11.46$ \\
\bottomrule
\end{tabular}}
\caption{The values for $\mathfrak{B}_1$, $\mathfrak{B}_2$ and $\mathfrak{B}_3$ from Corollary~\ref{cor:mainZD} for the corresponding $\sigma\in[\sigma_1,\sigma_2]$, as well as the optimized values of several parameters from Theorem~\ref{thm:mainZD}.} 
\label{tab:CorMainZD}
\end{table}

The precise result from which Corollary~\ref{cor:mainZD} follows is stated as Theorem~\ref{thm:mainZD} in Section~\ref{subsec:ZD}. Its proof is based on Ingham's original approach~\cite{In40}, and relies\footnote{Two other important ingredients are Gabriel's convexity theorem (Theorem~\ref{thm:FXbyGabriel}) and inequality~\eqref{eq:Littlewood} that is implied by Littlewood's lemma. One can avoid them, usually for the cost of producing worse log-factor in~\eqref{eq:Ingham2ndZD}, see~\cite[Chapter~11]{Ivic} and also~\cite[Chapter~12]{MontgomeryTopics} for generalisations to Dirichlet $L$-functions. In fact, Huxley established~\eqref{eq:Ingham2ndZD} for this class of $L$-functions.} on an upper bound for the fourth power moment of $\zeta(s)$ on the critical line. Let $k$ be a positive real number and $0\leq A\leq B$, and define
\begin{equation}
\label{eq:Mk}
\mathcal{M}_{k}(A,B) \de \int_{A}^{B}\left|\zeta\left(\frac{1}{2}+\ie t\right)\right|^{2k}\dif{t}.
\end{equation}
In particular, we set $\mathcal{M}_{k}(T)\de \mathcal{M}_{k}(0,T)$ for $T\geq 0$. The LH trivially implies $\mathcal{M}_{k}(T)\ll_{\varepsilon} T^{1+\varepsilon}$, and if the latter is true for some strictly increasing and unbounded sequence of $k$'s, then LH holds~\cite[Theorem~13.2]{titchmarsh1986theory}. However, much stronger statement
\begin{equation}
\label{eq:MkBound}
\mathcal{M}_{k}(T) \asymp_{k} T\left(\log{T}\right)^{k^2}
\end{equation}
is believed to be true. Unconditionally, the upper bound in~\eqref{eq:MkBound} is known to hold for $0<k\leq 2$, see~\cite{HeapRadziwillSound}, while the lower bound is true for all $k\geq 1$, see~\cite{RadziwillSoundLower}. On RH, the upper bound in~\eqref{eq:MkBound} also holds for $k>2$, see~\cite{Harper2013}, while the lower bound is also true for $0<k<1$, see~\cite{HBFractionalMoments}. Random matrix theory predicts the asymptotic form of~\eqref{eq:MkBound} for $k\in\NN$, that is,
\begin{equation}
\label{eq:MkAsympBound}
\mathcal{M}_{k}(T) = Tp_{k}(\log{T}) + O_{\varepsilon}\left(T^{1/2+\varepsilon}\right)
\end{equation}
for $\varepsilon>0$, where $p_{k}(x)$ is a polynomial of degree $k^2$ whose coefficients can be effectively computed, see~\cite{ConreyMoments}.
Polynomials $p_{k}(x)$ from~\eqref{eq:MkAsympBound} are known only for $k\in\{1,2\}$: $p_1(x)=x+2\gamma-1-\log{(2\pi)}$ where $\gamma$ is the Euler--Mascheroni constant, is due to Ingham~\cite{Ingham1926MV}, and $p_2(x)=\frac{1}{2\pi^2}x^4+\sum_{k=1}^{4}a_{4-k}x^{4-k}$ is due to Heath-Brown~\cite{HeathBrown1979TheFP}, where his\footnote{His error term is of order $T^{7/8+\varepsilon}$. This was improved in~\cite{IvicMotohashi95}, see also~\cite{PalTrud}, to $T^{2/3+\varepsilon}$ by using methods from the spectral theory of the Riemann zeta-function.} values for $a_k$, $k\in\{0,1,2,3\}$, coincide with those predicted by the random matrix theory. It should be mentioned that Ingham proved in the same paper~\cite{Ingham1926MV} that $p_2(x)=\frac{1}{2\pi^2}x^4+O(x^3)$, i.e.,
\begin{equation}
\label{eq:4thPMGeneral}
\mathcal{M}_{2}(T) = \frac{1}{2\pi^2}T\log^{4}{T} + O\left(T\log^{3}{T}\right),
\end{equation}
thus providing an asymptotic formula for $\mathcal{M}_{2}(T)$. Recently, explicit estimates for the second power moment $\mathcal{M}_{1}(T)$ appeared in~\cite{DonaZunigaAlterman} and~\cite{SimonicStarichkova} that are of the form~\eqref{eq:MkAsympBound} but with better error terms. To the best of the authors' knowledge, no explicit estimates for the fourth power moment $\mathcal{M}_{2}(T)$ have appeared in the literature. In this paper, we are providing them in Corollary~\ref{cor:main4thPM}.

\begin{corollary}
\label{cor:main4thPM}
We have
\begin{equation}
\label{eq:main4thPM1}
\mathcal{M}_{2}(T) \leq \frac{1}{\pi^2}T\log^{4}{\frac{T}{2}} + 20.7225\cdot T\log^{\frac{7}{2}}{\frac{T}{2}} + 1.9532\cdot 10^6
\end{equation}
for $T\geq 3\cdot 10^3$, and
\begin{equation}
\label{eq:main4thPM2}
-\left(\mathfrak{M}_{1}(T_0)+\frac{2\log{2}}{\pi^2}+\frac{T_0\log{\frac{T}{2}}}{\pi^2 T}\right)T\log^{3}{\frac{T}{2}}\leq \mathcal{M}_{2}(T)-\frac{1}{2\pi^2}T\log^{4}{\frac{T}{2}} \leq \mathfrak{M}_{1}(T_0) T\log^{3}{\frac{T}{2}} + \mathfrak{M}_{2}(T_0)
\end{equation}
for $T\geq T_0$, where the values for $\mathfrak{M}_1(T_0)$ and $\mathfrak{M}_2(T_0)$ are given by Table~\ref{tab:CorMain4thPM}.
\end{corollary}

The precise result from which Corollary~\ref{cor:main4thPM} follows is stated as Corollary~\ref{cor:4thPMExplicitGeneral} in Section~\ref{subsec:4thPM}, which in turn is a further consequence of Theorem~\ref{thm:4thPMExplicitGeneral}. Our proof of Theorem~\ref{thm:4thPMExplicitGeneral} follows Ramachandra's proof~\cite{Ramachandra75}, see Section~\ref{subsec:4thPMmethod} for details. In addition to zero density estimates, results on the fourth power moment, as well as other higher power moments, are used in the Dirichlet divisor problem~\cite[Section~12.3]{titchmarsh1986theory} and also to count integral ideals in abelian number fields~\cite{languasco2025counting}.

\begin{table}[h]
\centering
\footnotesize{
\begin{tabular}{lllll} 
\toprule
$T_0$ & $\mathfrak{M}_{1}(T_0)$ & $\mathfrak{M}_{2}(T_0)$ & $\mathcal{M}_{2}(T_0)\leq$ & $\mathcal{M}_{2}(T_0,2T_0)\leq$ \\ 
\midrule
$10^4$ & $59.300$ & $1.3796\cdot10^{9}$ & $3.8095\cdot10^{8}$ & $9.9857\cdot10^{8}$ \\
$10^5$ & $48.801$ & $2.8051\cdot10^{10}$ & $8.7748\cdot10^{9}$ & $1.9276\cdot10^{10}$ \\
$10^6$ & $43.230$ & $4.1722\cdot10^{11}$ & $1.3688\cdot10^{11}$ & $2.8034\cdot10^{11}$ \\
$10^7$ & $39.720$ & $5.8170\cdot10^{12}$ & $1.8478\cdot10^{12}$ & $3.9692\cdot10^{12}$ \\
$10^8$ & $37.550$ & $7.9793\cdot10^{13}$ & $2.5002\cdot10^{13}$ & $5.4791\cdot10^{13}$ \\
$10^9$ & $36.211$ & $1.0629\cdot10^{15}$ & $3.3317\cdot10^{14}$ & $7.2965\cdot10^{14}$ \\
$10^{10}$ & $34.842$ & $1.3897\cdot10^{16}$ & $4.3883\cdot10^{15}$ & $9.5079\cdot10^{15}$ \\
$10^{15}$ & $31.170$ & $4.4328\cdot10^{21}$ & $1.4174\cdot10^{21}$ & $3.0154\cdot10^{21}$ \\
$10^{20}$ & $29.364$ & $9.6675\cdot10^{26}$ & $3.1233\cdot10^{26}$ & $6.5442\cdot10^{26}$ \\
\bottomrule
\end{tabular}}
\caption{The values for $\mathfrak{M}_1(T_0)$ and $\mathfrak{M}_2(T_0)$ from Corollaries~\ref{cor:main4thPM} and~\ref{cor:main3rdPM}, as well as upper bounds for $\mathcal{M}_{2}(T_0)$ and $\mathcal{M}_{2}(T_0,2T_0)$ that are used in the proof of Corollary~\ref{cor:main4thPM}.} 
\label{tab:CorMain4thPM}
\end{table}

As an easy consequence of~\cite{DonaZunigaAlterman} and Corollary~\ref{cor:main4thPM} we also derive an explicit upper bound for the third power moment, see Corollary~\ref{cor:main3rdPM}.

\begin{corollary}
\label{cor:main3rdPM}
We have
\[
\mathcal{M}_{\frac{3}{2}}(T) \leq \left(\frac{1}{2\pi^2}+\frac{\mathfrak{M}_{1}(T_0)}{\log{T}}+\frac{\mathfrak{M}_{2}(T_0)}{T\log^{4}{T}}\right)^{\frac{1}{2}}T\log^{\frac{5}{2}}{T}
\]
for $T\geq T_0\geq 10^{6}$, where the values for $\mathfrak{M}_{1}(T_0)$ and $\mathfrak{M}_{2}(T_0)$ are given by Table~\ref{tab:CorMain4thPM}.
\end{corollary}

The exponent of the logarithm is for $1/4$ larger than the conjectured value, see~\eqref{eq:MkBound}. It is interesting that this is still the best known unconditional result on $\mathcal{M}_{\frac{3}{2}}(T)$.

Although not pursued here, Corollary~\ref{cor:mainZD} can be generalized to Dirichlet $L$-functions by following the approach presented here; however, substantial work should be done on generalizing Lemmas~\ref{lem:Mxone_half}--\ref{lem:arg} and~\ref{lem:Sumd2}. These and possible extensions to other $L$-functions are currently under investigation.

The outline of this paper is as follows. In Section~\ref{sec:MainResults} we state the main results, namely Theorem~\ref{thm:mainZD} for the zero density estimate and Theorem~\ref{thm:4thPMExplicitGeneral} for the fourth power moment estimate, and prove the above corollaries. In Section~\ref{sec:Gabriel} we prove Gabriel's convexity theorem, together with two applications (Corollaries~\ref{cor:FXbyGabriel} and~\ref{cor:convexity}), and in Section~\ref{sec:ProofThm1} we prove Theorem~\ref{thm:mainZD}. Section~\ref{sec:4thPM} is dedicated to the proof of Theorem~\ref{thm:4thPMExplicitGeneral}, while in Section~\ref{sec:appendix} we state an improved explicit version of Ingham's earlier zero density estimate.

\subsection*{Acknowledgements} The authors thank Andrew Yang, Tim Trudgian, Olivier Ramar\'{e} and Daniel Johnston for their interest in the work.

\section{The main results}
\label{sec:MainResults}

In this section, we are stating the precise results that are related to a zero density estimate~\eqref{eq:mainZD}, see Theorem~\ref{thm:mainZD}, and to the fourth power moment estimates~\eqref{eq:main4thPM1} and~\eqref{eq:main4thPM2}, see Theorem~\ref{thm:4thPMExplicitGeneral}. In addition, proofs of Corollaries~\ref{cor:mainZD},~\ref{cor:main4thPM} and~\ref{cor:main3rdPM} are also provided.

\subsection{The zero density estimate}
\label{subsec:ZD}

The proof of Theorem~\ref{thm:mainZD} will be presented in Section~\ref{sec:ProofThm1}.

\begin{theorem}
\label{thm:mainZD}
Let $T\geq H\geq H_0\geq 1002$, $T\geq T_0\geq\max\left\{e^{2d},e^e\right\}$ for $d>0$, $0<\delta\leq\frac{1}{20}\log{(hT_0)}$ for $h\geq 3\cdot10^{10}/T_0$, and $\kappa>0$. Then
\begin{flalign*}
N(\sigma,T)-N(\sigma,H) &\leq T^{\frac{3(1-\sigma)}{2-\sigma}}\left(\frac{\mathfrak{A}_1}{2\pi d}\left(\log{(hT)}\right)^{\frac{1}{2-\sigma}}\left(\log{T}\right)^{\frac{6-5\sigma}{2-\sigma}} + \frac{\mathfrak{A}_2}{2\pi d} \left(\log{(hT)}\right)^{\frac{2\sigma-1}{2-\sigma}}\log{T}\right) \\ 
&+ \left(\mathfrak{A}_3\cdot\log{(HT)}+\mathfrak{A}_4\cdot\log{(hT)}+\mathfrak{A}_5\right)\log{T}
\end{flalign*}
for $1/2+d/\log{T}<\sigma_1\leq\sigma\leq\sigma_2\leq 1$. Functions $\mathfrak{A}_1$ and $\mathfrak{A}_2$ are defined by 
\begin{flalign*}
\mathfrak{A}_1&=\mathfrak{A}_1(\sigma_1,\sigma_2,T_0,d,L_0,L_1,K_1,K_3) \de L_1\exp{\left(d\left(3+\frac{4\log{\log{T_0}}}{\log{T_0}}\right)\right)}\max\left\{1,L_0^{-\frac{2d}{\log{T_0}}}\right\}\times\\
&\times\max\left\{1,K_3^{-\frac{3d}{\log{T_0}}}\right\}\max\left\{1,K_1^{\frac{3d}{\log{T_0}}}\right\}
\max\left\{L_0^{\frac{2}{2-\sigma_1}},L_0^{\frac{2}{2-\sigma_2}}\right\}\times \\
&\times\max\left\{K_3^{\frac{2\sigma_1-1}{2-\sigma_1}},K_3^{\frac{2\sigma_2-1}{2-\sigma_2}}\right\}\max\left\{K_1^{\frac{3(1-\sigma_1)}{2-\sigma_1}},K_1^{\frac{3(1-\sigma_2)}{2-\sigma_2}}\right\}
\end{flalign*}
and
\begin{flalign*}
\mathfrak{A}_2&=\mathfrak{A}_2(\sigma_1,\sigma_2,T_0,d,L_0,L_2,K_2,K_3) \de L_2e^{3d}\max\left\{1,K_2^{\frac{3d}{\log{T_0}}}\right\}\max\left\{1,L_0^{-\frac{2d}{\log{T_0}}}\right\}\times \\
&\times\max\left\{1,K_3^{-\frac{3d}{\log{T_0}}}\right\}\max\left\{L_0^{\frac{2}{2-\sigma_1}},L_0^{\frac{2}{2-\sigma_2}}\right\}\max\left\{K_3^{\frac{2\sigma_1-1}{2-\sigma_1}},K_3^{\frac{2\sigma_2-1}{2-\sigma_2}}\right\}\max\left\{K_2^{\frac{3(1-\sigma_1)}{2-\sigma_1}},K_2^{\frac{3(1-\sigma_2)}{2-\sigma_2}}\right\},
\end{flalign*}
where $K_1=K_1(T_0,A_0,\kappa,h)$, $K_2=K_2(T_0,A_0,\kappa)$ and $K_3=K_3(T_0,\kappa,h,\delta)$ are defined by~\eqref{eq:K1},~\eqref{eq:K2} and~\eqref{eq:K3}, respectively, while $L_0 = L_0(T_0,H_0,\kappa,h,\delta)$, $L_1 = L_1(T_0,K_1,K_3,h,\delta)$ and $L_2 = L_2(T_0,K_2,K_3,h,\delta)$ are defined by~\eqref{eq:L0},~\eqref{eq:L1} and~\eqref{eq:L2}, respectively. Moreover,
\begin{gather*}
\mathfrak{A}_3(\sigma_1,T_0,d) \de \frac{1}{\pi\log{2}}\left(\frac{3-2\sigma_1}{2d}+\frac{1}{\log{T_0}}\right), \quad \mathfrak{A}_4(\sigma_1,T_0,d) \de \frac{3}{2\log{2}}\left(\frac{3-2\sigma_1}{2d}+\frac{1}{\log{T_0}}\right), \\
\mathfrak{A}_5(\sigma_1,T_0,H_0,d,h) \de \frac{C_6}{2\pi}\left(\frac{3-2\sigma_1}{2d}+\frac{1}{\log{T_0}}\right) + \frac{C_{5}}{2\pi d},
\end{gather*}
where $C_5=C_{5}(hT_0,h)$ and $C_6=C_6(hT_0,H_0)$ are defined by~\eqref{eq:C5} and~\eqref{eq:C6}, respectively.
\end{theorem}

\begin{proof}[Proof of Corollary~\ref{cor:mainZD}]
We are using Theorem~\ref{thm:mainZD} with $T_0=H_0=H=H_{\mathrm{RH}}$, $h=1$ and $0<d\leq 0.89$. Then
\begin{flalign}
N(\sigma,T) &\leq \frac{1}{2\pi d}\left(\mathfrak{A}_1+\mathfrak{A}_2\cdot\left(\log{H_{\mathrm{RH}}}\right)^{-\frac{6(1-\sigma_2)}{2-\sigma_2}}\right)T^{\frac{3(1-\sigma)}{2-\sigma}}\left(\log{T}\right)^{\frac{7-5\sigma}{2-\sigma}} \nonumber \\ 
&+ \left(\mathfrak{A}_3+\mathfrak{A}_4\right)\log^{2}{T} + \left(\mathfrak{A}_3\log{H_{\mathrm{RH}}}+\mathfrak{A}_5\right)\log{T} \label{eq:cor1}
\end{flalign}
under the conditions of Theorem~\ref{thm:mainZD}. Note that the functions $\mathfrak{A}_n$, $n\in\{1,\ldots,5\}$, are decreasing in $\sigma_1$. This implies that~\eqref{eq:cor1} holds for $\sigma\in(1/2+d/\log{T},0.53125]$ with $\sigma_1=1/2$ and $\sigma_2=0.53125$. The values for the parameters $d$, $\kappa$, $\delta$ and $A_0$ were obtained (see Table~\ref{tab:CorMainZD}) in a way that the main term in~\eqref{eq:cor1} is the smallest possible. Because the conditions of Theorem~\ref{thm:mainZD} are then satisfied, this implies
\[
N(\sigma,T) \leq 5.069\cdot T^{\frac{3(1-\sigma)}{2-\sigma}}\left(\log{T}\right)^{\frac{7-5\sigma}{2-\sigma}} + 9.461\log^{2}{T} + 167.8\log{T}
\]
for $\sigma\in(1/2+0.28/\log{T},0.53125]$ and $T\geq H_{\mathrm{RH}}$. The above estimate is true also for $\sigma\in[1/2,1/2+0.28/\log{T}]$ since 
\begin{equation}
\label{eq:CloseCL}
N(T) \leq \frac{T}{2\pi}\log{\frac{T}{2\pi e}} + 0.43\log{T} \leq 5.069\exp{\left(\frac{-0.56}{\frac{3}{2}-\frac{0.28}{\log{H_{\mathrm{RH}}}}}\right)}T\log^{2}{T}+9.461\log^{2}{T} + 167.8\log{T} 
\end{equation}
holds for all $T\geq H_{\mathrm{RH}}$ and 
\[
T\exp{\left(\frac{-0.56}{\frac{3}{2}-\frac{0.28}{\log{T_0}}}\right)} \leq T^{\frac{3(1-\sigma)}{2-\sigma}}
\]
holds for $\sigma\in[1/2,1/2+0.28/\log{T}]$ and $T\geq H_{\mathrm{RH}}$, while the first inequality in~\eqref{eq:CloseCL} follows by~\cite{BellottiWong2024}. This proves~\eqref{eq:mainZD} for $\sigma\in[1/2,0.53125]$. For the other ranges, the procedure is similar: the values for the parameters $d$, $\kappa$, $\delta$ and $A_0$ are obtained for each interval so that the main term in~\eqref{eq:cor1} is the smallest possible; see Table~\ref{tab:CorMainZD}. In each case, the conditions of Theorem~\ref{thm:mainZD} are satisfied, so~\eqref{eq:cor1} holds. The proof is thus complete.
\end{proof}

\subsection{The fourth power moment estimate}
\label{subsec:4thPM}

The proof of Theorem~\ref{thm:4thPMExplicitGeneral} will be presented in Section~\ref{sec:4thPM}. The $2k$-th power moment $\mathcal{M}_{k}(A,B)$ is defined by~\eqref{eq:Mk}.

\begin{theorem}
\label{thm:4thPMExplicitGeneral}
Assume the following notation and conditions: 
\begin{enumerate}
    \item $T\geq T_0\geq 55$;
    \item $\bm{c}=(c_1,c_2)$ with $c_1\in(-1,-1/2)$ and $c_2\in(0,1/2)$;
    \item $\bm{\sigma}=\left(\sigma_1,\bm{\sigma'},\bm{\sigma''}\right)$ with $\sigma_1\in[0,1/2)$, $\bm{\sigma'}=(\sigma_2,\sigma_3)$ with $\sigma_2\in(1/2,1]$ and $\sigma_3\in[0,1/2)$, and $\bm{\sigma''}=(\sigma_4,\sigma_5)$ with $\sigma_4\in(1/2,|c_1|)$ and $\sigma_5\in(1/2-c_2,1/2)$;
    \item $\bm{\mathfrak{a}}=\left(\mathfrak{a}_1,\mathfrak{a}_2\right)$ with $\mathfrak{a}_1>0$ and $\mathfrak{a}_2>1$, and $\bm{\mathfrak{b}}=\left(\mathfrak{b}_1,\mathfrak{b}_2\right)$ with $\mathfrak{b}_1>0$ and $\mathfrak{b}_2>1$;
    \item $T-\mathfrak{a}_1\left(\log{T}\right)^{\mathfrak{a}_2}\geq T_0-\mathfrak{a}_1\left(\log{T_0}\right)^{\mathfrak{a}_2}\geq 2$, $\left(\log{T}\right)^{\mathfrak{a}_2}/T\leq \left(\log{T_0}\right)^{\mathfrak{a}_2}/T_0$ and $\mathfrak{a}_1\left(\log{T_0}\right)^{\mathfrak{a}_2-1}\geq 1$;
    \item $T-\mathfrak{b}_1\left(\log{T}\right)^{\mathfrak{b}_2}\geq T_0-\mathfrak{b}_1\left(\log{T_0}\right)^{\mathfrak{b}_2}\geq 2$, $\left(\log{T}\right)^{\mathfrak{b}_2}/T\leq \left(\log{T_0}\right)^{\mathfrak{b}_2}/T_0$ and $\mathfrak{b}_1\left(\log{T_0}\right)^{\mathfrak{b}_2-1}\geq 2$.
\end{enumerate}
Then
\begin{equation}
\label{eq:4thPMExplicitGeneral}
\left|\mathcal{M}_{2}(T,2T) - \frac{1}{2\pi^2}T\log^{4}{T}\right| \leq  \mathfrak{F}_{1}\left(T_0,\bm{c},\bm{\sigma},\bm{\mathfrak{a}},\bm{\mathfrak{b}}\right)T\log^{3}{T}
\end{equation}
and 
\begin{equation}
\label{eq:4thPMExplicitGeneral2}
\mathcal{M}_{2}(T,2T) \leq \frac{1}{\pi^2}T\log^{4}{T} + \left(2\sum_{n=3}^{6}\left(\frac{\mathcal{J}_{n}}{\pi^2}\log{T}+2\mathcal{J}_{1}\mathcal{J}_{n}\right)^{\frac{1}{2}} + \mathfrak{F}_{2}\left(T_0,\bm{c},\bm{\mathfrak{a}},\bm{\mathfrak{b}}\right)\right)T\log^{3}{T},
\end{equation}
where
\[
\mathfrak{F}_{1}\left(T_0,\bm{c},\bm{\sigma},\bm{\mathfrak{a}},\bm{\mathfrak{b}}\right)\de \mathcal{J}_0 + \mathcal{J}_1 + \mathcal{J}_{2,1} + 2\mathcal{J}_{2,2} + \sum_{n=3}^{6}\mathcal{J}_{n} + 2\mathop{\sum\sum}_{3\leq n<m\leq 6}\left(\mathcal{J}_{n}\mathcal{J}_m\right)^{\frac{1}{2}}
\]
and
\[
\mathfrak{F}_{2}\left(T_0,\bm{c},\bm{\mathfrak{a}},\bm{\mathfrak{b}}\right) = \mathcal{J}_0 + 2\mathcal{J}_{1} + \sum_{n=3}^{6}\mathcal{J}_{n} + 2\mathop{\sum\sum}_{3\leq n<m\leq 6}\left(\mathcal{J}_{n}\mathcal{J}_m\right)^{\frac{1}{2}}.
\]
Here, $\mathcal{J}_{n}=\mathcal{J}_{n}(T_0)$ for $n\in\{1,3,4\}$ are defined by~\eqref{eq:cJ1},~\eqref{eq:cJ3} and~\eqref{eq:cJ4}, respectively, while $\mathcal{J}_{0}=\mathcal{J}_0\left(T_0,\bm{c},\bm{\mathfrak{a}},\bm{\mathfrak{b}}\right)$, $\mathcal{J}_{2,1}=\mathcal{J}_{2,1}(T_0,\sigma_1)$, $\mathcal{J}_{2,2}=\mathcal{J}_{2,2}\left(T_0,\bm{\sigma'},\bm{\sigma''},\bm{c},\bm{\mathfrak{a}},\bm{\mathfrak{b}}\right)$, $\mathcal{J}_5=\mathcal{J}_5\left(T_0,c_1,\bm{\mathfrak{a}}\right)$ and $\mathcal{J}_6=\mathcal{J}_6\left(T_0,c_2,\bm{\mathfrak{b}}\right)$ are defined by~\eqref{eq:cJ0},~\eqref{eq:cJ21},~\eqref{eq:cJ22},~\eqref{eq:cJ5} and~\eqref{eq:cJ6}, respectively.
\end{theorem}

\begin{corollary}
\label{cor:4thPMExplicitGeneral}
Assume the notation and conditions of Theorem~\ref{thm:4thPMExplicitGeneral}. Then
\begin{equation}
\label{4thMlargeV}
\mathcal{M}_{2}(T_0,T) \leq \frac{1}{2\pi^2}T\log^{4}{\frac{T}{2}} + \mathfrak{F}_{1}\left(T_0,\bm{c},\bm{\sigma},\bm{\mathfrak{a}},\bm{\mathfrak{b}}\right)T\log^{3}{\frac{T}{2}} + \mathcal{M}_{2}(T_0,2T_0)
\end{equation}
and
\begin{equation}
\label{eq:4thMLower}
\mathcal{M}_{2}(T_0,T) \geq \frac{1}{2\pi^2}T\log^{4}{\frac{T}{2}} - \left(\frac{2\log{2}}{\pi^2}+\mathfrak{F}_{1}\left(T_0,\bm{c},\bm{\sigma},\bm{\mathfrak{a}},\bm{\mathfrak{b}}\right)\right)T\log^{3}{\frac{T}{2}} - \frac{T_0}{\pi^2}\log^{4}{\frac{T}{2}}.
\end{equation}
Furthermore,
\begin{equation}
\label{4thMsmallV}
\mathcal{M}_{2}(T_0,T) \leq \frac{1}{\pi^2}T\log^{4}{\frac{T}{2}} + \left(2\sum_{n=3}^{6}\left(\frac{\mathcal{J}_{n}}{\pi^2}\log{\frac{T}{2}}+2\mathcal{J}_{1}\mathcal{J}_{n}\right)^{\frac{1}{2}} + \mathfrak{F}_{2}\left(T_0,\bm{c},\bm{\mathfrak{a}},\bm{\mathfrak{b}}\right)\right)T\log^{3}{\frac{T}{2}} + \mathcal{M}_{2}(T_0,2T_0).
\end{equation}
\end{corollary}

\begin{proof}
It is clear that all three estimates hold for $T\in[T_0,2T_0]$, so we can assume $T\geq 2T_0$. Let $n_0\geq 1$ be an integer such that $2^{-n_0}T\geq T_0$ but $2^{-n_0-1}T<T_0$. We then have
\[
\mathcal{M}_{2}(T_0,T) = \mathcal{M}_{2}\left(T_0,\frac{T}{2^{n_0}}\right) + \sum_{n=0}^{n_0-1}\mathcal{M}_{2}\left(\frac{T}{2^{n+1}},\frac{T}{2^{n}}\right)
\]
by a dyadic partition. We can use the estimates from Theorem~\ref{thm:4thPMExplicitGeneral} to bound the above sum due to our conditions. Then the upper bounds from Corollary~\ref{cor:4thPMExplicitGeneral} immediately follow. The lower bound is derived as follows. Using $(1-x)^{4}-1\geq -4x$, which is valid for $x\geq 0$, we obtain
\begin{flalign*}
\frac{1}{2\pi^2}\sum_{n=0}^{n_0-1}\frac{T}{2^{n+1}}\left(\log{\frac{T}{2^{n+1}}}\right)^{4} &\geq \frac{1}{2\pi^2}\left(\sum_{n=0}^{n_0-1}\frac{1}{2^{n+1}}\right)T\log^{4}{\frac{T}{2}} - \frac{2\log{2}}{\pi^2}\left(\sum_{n=0}^{n_0-1}\frac{n}{2^{n+1}}\right)T\log^{3}\frac{T}{2} \\
&\geq \frac{1}{2\pi^2}\left(1-\frac{2T_0}{T}\right)T\log^{4}{\frac{T}{2}} - \frac{2\log{2}}{\pi^2}T\log^{3}\frac{T}{2}.
\end{flalign*}
The result now follows from~\eqref{eq:4thPMExplicitGeneral}.
\end{proof}

\begin{corollary}
\label{cor:4thPM}
Let $A_{0}\geq 10^{10}$. Then we have $\mathcal{M}_{2}(A_0) \leq a_1(A_0)$ and $\mathcal{M}_{2}(A_0,T) \leq a_2(A_0)T\log^{4}{T}$ for every $T\geq A_0$, where
\[
a_1(A_0) \de \frac{1}{2\pi^2}A_0\log^{4}\frac{A_0}{2} + 34.842 A_0\log^{3}\frac{A_0}{2} + 1.3897\cdot 10^{16}
\]
and
\[
a_2(A_0) \de \frac{1}{2\pi^2}\left(1-\frac{\log{2}}{\log{A_0}}\right)^{4} + \frac{34.842}{\log{A_0}}\left(1-\frac{\log{2}}{\log{A_0}}\right)^{3} + \frac{9.5079\cdot 10^{15}}{A_0\log^{4}{A_0}}.
\]
\end{corollary}

\begin{proof}
The first estimate follows by the second inequality in~\eqref{eq:main4thPM2} of Corollary~\ref{cor:main4thPM}. Because $\mathcal{M}_{2}\left(A_0,T\right)\leq \mathcal{M}_2\left(10^{10},T\right)$ and $\mathcal{M}_{2}\left(10^{10},2\cdot 10^{10}\right)\leq 9.5079\cdot10^{15}$ by Table~\ref{tab:CorMain4thPM}, we have $\mathcal{M}_{2}(A_0,T) \leq a_2(T)T\log^{4}{T}$ by~\eqref{eq:Cor2M2}, see the proof of Corollary~\ref{cor:main4thPM}. The derivative analysis shows that $a_2(T)$ is a decreasing function for $T\geq 10^5$, implying that $a_2(T)\leq a_2(A_0)$. This gives the second estimate.
\end{proof}

\begin{proof}[Proof of Corollary~\ref{cor:main4thPM}]
Assume $T\geq T_0'=T_0\geq 55$, and also that the conditions from Theorem~\ref{thm:4thPMExplicitGeneral} are satisfied. By~\eqref{4thMsmallV} we have
\begin{equation*}
\mathcal{M}_{2}(T_0',T) \leq \frac{1}{\pi^2}T\log^{4}{\frac{T}{2}} + \mathcal{C}_{1}\left(T_0',\bm{c},\bm{\mathfrak{a}},\bm{\mathfrak{b}}\right)T\log^{\frac{7}{2}}{\frac{T}{2}} + \mathcal{M}_{2}(T_0',2T_0')
\end{equation*}
for $T\geq T_0'$, where
\[
\mathcal{C}_{1}\left(T_0',\bm{c},\bm{\mathfrak{a}},\bm{\mathfrak{b}}\right) \de 2\sum_{n=3}^{6}\left(\frac{\mathcal{J}_{n}}{\pi^2}+\frac{2\mathcal{J}_{1}\mathcal{J}_{n}}{\log{(T_0'/2)}}\right)^{\frac{1}{2}} + \frac{\mathfrak{F}_{2}\left(T_0',\bm{c},\bm{\mathfrak{a}},\bm{\mathfrak{b}}\right)}{\log^{1/2}{(T_0'/2)}}.
\]
We are taking $T_0'=3\cdot 10^{3}$ because for this value the calculations via numerical integration are still reasonably fast. We obtain $7.031\cdot 10^5<\mathcal{M}_2(T_0') < 7.032\cdot 10^5$ and $1.249\cdot 10^6<\mathcal{M}_2(T_0',2T_0')<1.25\cdot 10^6$. For the choices $\bm{c}=(-0.923364,0.161176)$, $\bm{\mathfrak{a}}=(0.531241,1.52906)$ and $\bm{\mathfrak{b}}=(0.754804,1.68921)$, the value of $\mathcal{C}_{1}(T_0',\bm{c},\bm{\mathfrak{a}},\bm{\mathfrak{b}})$ is $20.7225$, which is as small\footnote{According to the method \texttt{FindMinimum} from \emph{Mathematica}. It should also be noted that all computations in this paper were performed with this software system.} as possible. We can also see that the conditions~(5) and~(6) from Theorem~\ref{thm:4thPMExplicitGeneral} are satisfied by this choice of parameters. Therefore, Corollary~\ref{cor:4thPMExplicitGeneral} implies that
\begin{equation}
\label{eq:Cor2MExpl}
\mathcal{M}_{2}\left(T_0',T\right) \leq \frac{1}{\pi^2}T\log^{4}{\frac{T}{2}} + 20.7225\cdot T\log^{\frac{7}{2}}{\frac{T}{2}} + 1.25\cdot 10^6
\end{equation}
for $T\geq T_0'=3\cdot 10^3$. Inequality~\eqref{eq:main4thPM1} now easily follows from~\eqref{eq:Cor2MExpl} since $\mathcal{M}_{2}(T)=\mathcal{M}_{2}\left(T_0',T\right)+\mathcal{M}_{2}\left(T_0'\right)$.

Assume $T\geq T_0''=T_0\geq T_0'=3\cdot 10^3$, and also that the conditions from Theorem~\ref{thm:4thPMExplicitGeneral} are satisfied. Then 
\begin{equation}
\label{eq:Cor2M2}
\mathcal{M}_{2}\left(T_0'',T\right) \leq \frac{1}{2\pi^2}T\log^{4}{\frac{T}{2}} + \mathfrak{F}_{1}\left(T_0'',\bm{c},\bm{\sigma},\bm{\mathfrak{a}},\bm{\mathfrak{b}}\right)T\log^{3}{\frac{T}{2}} + \mathcal{M}_{2}\left( T_{0}'',2T_{0}''\right)
\end{equation}
by~\eqref{4thMlargeV}. Also, $\mathcal{M}_{2}\left(T_0'',2T_0''\right)\leq \mathcal{M}_{2}\left(T_0',2T_0''\right)\leq\mathcal{C}_{2}\left( T_{0}''\right)$ and $\mathcal{M}_{2}\left(T_0''\right)\leq\mathcal{C}_{2}\left( T_{0}''/2\right)+7.032\cdot10^5$ by~\eqref{eq:Cor2MExpl} and~\eqref{eq:main4thPM1}, respectively, where
\[
\mathcal{C}_{2}\left( T_{0}''\right)\de \frac{2}{\pi^2}T_0''\log^{4}{T_0''} + 20.7225\left(2T_0''\log^{\frac{7}{2}}{T_0''} \right)+ 1.25\cdot 10^6.
\] 
Trivially, $\mathcal{M}_{2}(T)=\mathcal{M}_{2}\left(T_0''\right)+\mathcal{M}_{2}\left(T_0'',T\right)$. We take the values for $T_0''$ as indicated in Table~\ref{tab:ProofCorMain4thPM}. Finding a global minimum for $\mathfrak{F}_1\left(T_0'',\bm{c},\bm{\sigma},\bm{\mathfrak{a}},\bm{\mathfrak{b}}\right)$ for each $T_0''$ is a computationally very complex problem, since we need to optimize 11 parameters that satisfy the conditions~(2)--(4) from Theorem~\ref{thm:4thPMExplicitGeneral}. To simplify the process, we used randomization; we generate $50$ trials where for each trial we randomly choose $100$ sets of values for these parameters, then we select an element from this set such that $\mathfrak{F}_1\left(T_0'',\bm{c},\bm{\sigma},\bm{\mathfrak{a}},\bm{\mathfrak{b}}\right)$ is minimal, and finally the ``minimal'' set from all trials. The values we obtained with such a method are listed in Table~\ref{tab:ProofCorMain4thPM}.

\begin{table}[h]
\centering
\footnotesize{
\begin{tabular}{lllll} 
\toprule
$T_0''$ & $\bm{c}$ & $\bm{\sigma}$ & $\bm{\mathfrak{a}}$ & $\bm{\mathfrak{b}}$ \\ 
\midrule
$10^4$ & $(-0.879,0.1809)$ & $(0.3545,(0.7898,0.4008),(0.5456,0.3933))$ & $(0.339,1.7553)$ & $(1.7215,1.303)$ \\
$10^5$ & $(-0.938,0.194)$ & $(0.368,(0.718,0.435),(0.629,0.3654))$ & $(1.99,1.44)$ & $(1.33,1.45)$ \\
$10^6$ & $(-0.9295,0.1721)$ & $(0.4077,(0.6957,0.3413),(0.5672,0.4288))$ & $(1.6769,1.5632)$ & $(1.6891,1.6748)$ \\
$10^7$ & $(-0.9273,0.2061)$ & $(0.4104,(0.6263,0.3536),(0.5743,0.41))$ & $(1.8473,1.8288)$ & $(1.649,1.9094)$ \\
$10^8$ & $(-0.9266,0.1637)$ & $(0.3561,(0.7103,0.3381),(0.5645,0.4285))$ & $(0.6164,1.9123)$ & $(1.0018,1.9929)$ \\
$10^9$ & $(-0.834,0.22985)$ & $(0.3964,(0.9479,0.3927),(0.5846,0.4235))$ & $(1.3609,1.1367)$ & $(0.7112,1.8455)$ \\
$10^{10}$ & $(-0.8303,0.1676)$ & $(0.3985,(0.668,0.4088),(0.551,0.386))$ & $(1.4356,1.28)$ & $(1.5563,1.6332)$ \\
$10^{15}$ & $(-0.813,0.221)$ & $(0.3464,(0.8268,0.4263),(0.5698,0.3511))$ & $(1.0851,1.0245)$ & $(0.8137,1.8276)$ \\
$10^{20}$ & $(-0.875,0.1936)$ & $(0.3785,(0.8615,0.45),(0.619,0.393))$ & $(0.1917,1.525)$ & $(1.1358,1.717)$ \\
\bottomrule
\end{tabular}}
\caption{Our choice of the values for the parameters $\bm{c}$, $\bm{\sigma}$, $\bm{\mathfrak{a}}$ and $\bm{\mathfrak{b}}$.} 
\label{tab:ProofCorMain4thPM}
\end{table}

We obtain $\mathfrak{F}_1\left(T_0'',\bm{c},\bm{\sigma},\bm{\mathfrak{a}},\bm{\mathfrak{b}}\right)\leq \mathfrak{M}_{1}\left(T_0''\right)$, where the values for $\mathfrak{M}_{1}\left(T_0''\right)$ are listed in Table~\ref{tab:CorMain4thPM}. The conditions~(5) and~(6) from Theorem~\ref{thm:4thPMExplicitGeneral} are clearly satisfied for each $T_0''$ by the above choice of parameters. The first inequality in~\eqref{eq:main4thPM2} thus follows from~\eqref{eq:4thMLower} since $\mathcal{M}_{2}(T)\geq \mathcal{M}_{2}\left(T_0'',T\right)$. Now, $\mathcal{M}_{2}\left(10^4,2\cdot 10^4\right)\leq 9.9857\cdot10^8$ and $\mathcal{M}_{2}\left(10^4\right)\leq 3.8095\cdot10^8$, confirming the values listed in Table~\ref{tab:CorMain4thPM}. This gives $\mathcal{M}_{2}\left(10^4\right)+\mathcal{M}_{2}\left(10^4,2\cdot 10^4\right)\leq\mathfrak{M}_{2}\left(10^4\right)$, where the value for $\mathfrak{M}_{2}\left(10^4\right)$ is as in Table~\ref{tab:CorMain4thPM}. The proof of~\eqref{eq:main4thPM2} for $T_0=10^4$ is thus complete. For other values of $T_0$ we proceed as follows. For $n\in\{5,\ldots,10\}\cup\{15\}\cup\{20\}$, let $\mathscr{A}_n\de\{1,\ldots,n-4\}$ if $n\in\{5,\ldots,10\}$, $\mathscr{A}_n\de\{5,\ldots,11\}$ if $n=15$, and $\mathscr{A}_n\de\{5\}\cup\{10,\ldots,16\}$ if $n=20$. Then
\begin{multline*}
\mathcal{M}_{2}\left(10^n,2\cdot 10^n\right) \leq \min\biggl\{\mathcal{C}_{2}\left(10^n\right), \\
\min_{i\in\mathscr{A}_n}\left\{\frac{10^n}{\pi^2}\log^{4}{10^n}+\left(2\cdot10^n\right)\mathfrak{M}_{1}\left(10^{n-i}\right)\log^{3}{10^n}+\mathcal{M}_{2}\left(10^{n-i},2\cdot10^{n-i}\right)\right\}\biggr\}
\end{multline*}
and
\begin{multline*}
\mathcal{M}_{2}\left(10^n\right) \leq \min\left\{\mathcal{C}_{2}\left(\frac{10^n}{2}\right)+7.032\cdot10^{5},\right. \\
\left.\min_{i\in\mathscr{A}_n}\left\{\frac{10^n}{2\pi^2}\log^{4}{\frac{10^n}{2}}+\mathfrak{M}_{1}\left(10^{n-i}\right)10^{n}\log^{3}{\frac{10^n}{2}}+\mathfrak{M}_{2}\left(10^{n-i}\right)\right\}\right\}.
\end{multline*}
With this we recursively verify that $\mathcal{M}_{2}\left(10^n\right)+\mathcal{M}_{2}\left(10^n,2\cdot 10^n\right)\leq\mathfrak{M}_{2}\left(10^n\right)$, where the values for $\mathfrak{M}_{2}\left(10^n\right)$ are listed in Table~\ref{tab:CorMain4thPM}. The proof of Corollary~\ref{cor:main4thPM} is thus complete.
\end{proof}

\begin{proof}[Proof of Corollary~\ref{cor:main3rdPM}]
We have
\[
\mathcal{M}_{1}(T)\leq T\log{T}, \quad \mathcal{M}_{2}(T)\leq \left(\frac{1}{2\pi^2}+\frac{\mathfrak{M}_{1}(T_0)}{\log{T}}+\frac{\mathfrak{M}_{2}(T_0)}{T\log^{4}{T}}\right)T\log^{4}{T}
\]
for $T\geq T_0\geq 10^6$ by~\cite[Theorem~3.4]{DonaZunigaAlterman} for the first inequality, and by Corollary~\ref{cor:main4thPM} for the second inequality. The result now follows by the Cauchy-Bunyakovsky-Schwarz (CBS) inequality since then $\mathcal{M}_{2/3}(T)\leq \left(\mathcal{M}_{1}(T)\mathcal{M}_{2}(T)\right)^{1/2}$.
\end{proof}

\section{Gabriel's convexity theorem and applications}
\label{sec:Gabriel}

Gabriel~\cite[Theorem 2]{Gabriel27} found an interesting two-variable convexity estimate for integrals of holomorphic functions along vertical lines in the strip, see Lemma~\ref{lem:Gabriel}.

\begin{lemma}
\label{lem:Gabriel}
Let $f(z)$ be a holomorphic function in the strip $\alpha<\Re\{z\}<\beta$ for $\alpha<\beta$ with continuous extension to the boundary $\Re\{z\}=\alpha$, $\Re\{z\}=\beta$. Assume that 
\[
\lim_{\substack{|y|\to\infty \\ y\in\R}}|f(x+\ie y)|=0
\]
uniformly in $x\in[\alpha,\beta]$. For $c>0$ define
\begin{equation}
\label{eq:integral1}
G_{f}(x,c)\de \left(\int_{-\infty}^{\infty} \left|f(x+\ie y)\right|^{\frac{1}{c}} \dif{y}\right)^{c}.
\end{equation}
Then
\begin{equation}
\label{eq:Gabriel}
G_{f}(x, w_{1}\lambda + w_{2}\mu) \leq G_{f}^{w_{1}} ( \alpha, \lambda) G_{f}^{w_{2}}( \beta, \mu)
\end{equation}
for $x\in[\alpha,\beta]$, where
\begin{equation}
\label{eq:w1w2}
w_{1} = w_1(x,\alpha,\beta) \de \frac{\beta-x}{\beta- \alpha}, \quad  w_{2} = w_2(x,\alpha,\beta) \de \frac{x-\alpha}{\beta-\alpha},
\end{equation}
and $\lambda$ and $\mu$ are arbitrary positive numbers.
\end{lemma}

\begin{proof}
See, e.g.,~\cite[pp.~381--384]{KaratsubaVoronin}.
\end{proof}

This result should be compared with~\cite[Theorem 7]{HardyInghamPolya}, where~\eqref{eq:Gabriel} with $\lambda=\mu=1/p$ was proved under different growth conditions on $f(z)$. 

Let $y\mapsto f(x+\ie y)$ be a continuous function for some fixed $x\in\R$ and $A\leq y\leq B$. Then
\begin{equation}
\label{eq:integral2}
I_{f}(x,A,B,c) \de \int_{A}^{B}\left|f(x+\ie y)\right|^{\frac{1}{c}}\dif{y}
\end{equation}
exists for every $c>0$. The next theorem shows how we can replace, under certain mild conditions on $f(z)$, the infinite integral~\eqref{eq:integral1} with the finite integral~\eqref{eq:integral2}, and still more or less preserve the form of Gabriel's inequality~\eqref{eq:Gabriel}. 

\begin{theorem}
\label{thm:FXbyGabriel}
Let $1/2\leq\alpha<\beta$ and consider a holomorphic function $f(z)$ in $\{z\in\C\colon \alpha<\Re\{z\}<\beta\}\setminus\{1\}$ with continuous extension to the vertical lines $\Re\{z\}=\alpha$ and $\Re\{z\}=\beta$ such that:
\begin{enumerate}
    \item If $\beta>1$, then $f(z)$ has a pole of order $m\geq0$ at $z=1$. If $\beta\leq 1$, we set $m=0$.
    \item For any fixed $d>0$, $\lim_{|y|\to\infty,y\in\R}e^{-dy^2}|f(x+\ie y)|=0$ uniformly in $x\in[\alpha,\beta]$.
    \item For $x\in\{\alpha,\beta\}$ and for any fixed $a\geq0$, $b>0$ and $c>0$, $\lim_{|u|\to\infty,u\in\R}e^{-u^2}I_{f}(x,a,bu,c)=0$.
\end{enumerate}
Also, let $T>H\geq 0$, $\kappa>0$, $A_1\geq 0$, and $A_2\geq 0$, and let $\lambda$ and $\mu$ be arbitrary positive numbers. Then
\begin{multline}
\label{eq:FXbyGabriel}
\left(I_{f}\left(x,-T,-H,w_{1}\lambda+w_{2}\mu\right)+I_{f}\left(x,H,T,w_{1}\lambda+w_{2}\mu\right)\right)^{w_{1}\lambda+w_{2}\mu} \leq \left(\frac{(x-1)^2+H^2}{x^2+H^2}\right)^{-\frac{m}{2}}\times \\
\times\exp{\left(\kappa\left(1+\left(\frac{\beta-\alpha}{2T}\right)^{2}\right)\right)}\left(\mathcal{J}(\alpha,A_1,\lambda)\right)^{w_1\lambda}\left(\mathcal{J}(\beta,A_2,\mu)\right)^{w_2\mu}
\end{multline}
for $x\in[\alpha,\beta]$, where
\begin{equation}
\label{eq:FXbyGabrielJ}
\mathcal{J}(a,b,c) \de I_{f}(a,-b,b,c)+\int_{\frac{\kappa}{c}\left(\frac{b}{T}\right)^2}^{\infty}e^{-u}\left(I_{f}\left(a,-T\sqrt{\frac{cu}{\kappa}},-b,c\right)+I_{f}\left(a,b,T\sqrt{\frac{cu}{\kappa}},c\right)\right)\dif{u},
\end{equation}
$w_1=w_1(x,\alpha,\beta)$ and $w_2=w_2(x,\alpha,\beta)$ are given by~\eqref{eq:w1w2}, and $H>0$ if $x=1$ and $m\neq0$.
\end{theorem}


\begin{proof}
Write $z=x+\ie y$ for $(x,y)\in\R^{2}$ and define
\[
g(z) \de \left(\frac{z-1}{z}\right)^{m}\exp{\left(\kappa\left(\frac{z}{T}\right)^{2}\right)}, \quad \phi(z)\de g(z)f(z),
\]
where $m$ is from the condition~(1) of Theorem~\ref{thm:FXbyGabriel}. It is clear that $\phi(z)$ is a holomorphic function in the strip $\alpha<x<\beta$ with continuous extension to the vertical lines $x=\alpha$ and $y=\beta$. By a slight modification of~\cite[Lemma 3.7]{KADIRI201822} we have $|g(z)|\leq\omega_1(x)e^{-\kappa(|y|/T)^{2}}$ for $x\geq1/2$, $y\in\R$, and $|g(z)|\geq\omega_2(x)$ for $x\geq1/2$, $|y|\in[H,T]$, where
\[
\omega_1(x)\de \exp{\left(\kappa\left(\frac{x}{T}\right)^{2}\right)}, \quad \omega_2(x)\de \left(\frac{(x-1)^2+H^2}{x^2+H^2}\right)^{\frac{m}{2}}\exp{\left(\kappa\left(\frac{x}{T}\right)^{2}-\kappa\right)}.
\]
This also shows that $\lim_{|y|\to\infty}\left|\phi(x+\ie y)\right|=0$, uniformly in $x\in[\alpha,\beta]$, due to the condition~(2). Therefore, Lemma~\ref{lem:Gabriel} can be applied for $\phi(z)$. We are going to provide upper and lower bounds for $G_{\phi}(x,c)$, where this function is defined by~\eqref{eq:integral1}. For $A\geq0$, $x\in\{\alpha,\beta\}$ and $c>0$ we have
\begin{flalign}
G_{\phi}(x,c) &\leq \omega_{1}(x)\left(\int_{-\infty}^{\infty}\exp{\left(-\frac{\kappa}{c}\left(\frac{y}{T}\right)^2\right)}\left|f(x+\ie y)\right|^{\frac{1}{c}}\dif{y}\right)^{c} \nonumber \\ 
&= \omega_1(x)\left(\int_{-\infty}^{\infty}e^{-v^2}\left(\frac{\dif{}}{\dif{v}}I_{f}\left(x,A,T\sqrt{\frac{c}{\kappa}}v,c\right)\right)\dif{v}\right)^{c} \label{eq:gabrielU1} \\ 
&= \omega_1(x)\left(\int_{-\infty}^{\infty}2ve^{-v^2}I_{f}\left(x,A,T\sqrt{\frac{c}{\kappa}}v,c\right)\dif{v}\right)^{c}\label{eq:gabrielU2} \\
&\leq\omega_1(x)\left(\mathcal{J}(x,A,c)\right)^{c}, \label{eq:gabrielU3}
\end{flalign}
where $\mathcal{J}(a,b,c)$ is defined by~\eqref{eq:FXbyGabrielJ}. Equality~\eqref{eq:gabrielU1} follows by taking $y=Tv\sqrt{c/\kappa}$ and noticing that
\[
\frac{\dif{}}{\dif{v}}I_{f}\left(x,A,T\sqrt{\frac{c}{\kappa}}v,c\right) = T\sqrt{\frac{c}{\kappa}}\left.\frac{\partial}{\partial{B}}I_{f}(x,A,B,c)\right|_{B=T\sqrt{\frac{c}{\kappa}}v} = T\sqrt{\frac{c}{\kappa}}\left|f\left(x+\ie T\sqrt{\frac{c}{\kappa}}v\right)\right|^{\frac{1}{c}},
\]
while~\eqref{eq:gabrielU2} follows by integration by parts and the condition~(3). Inequality~\eqref{eq:gabrielU3} follows by splitting the integral, taking $v^2=u$ and observing that 
\[
I_{f}\left(x,-T\sqrt{\frac{cu}{\kappa}},A,c\right) = I_{f}\left(x,-T\sqrt{\frac{cu}{\kappa}},-A,c\right) + I_{f}\left(x,-A,A,c\right), 
\]
and also that
$I_{f}(x,A,T\sqrt{cu/\kappa},c)\leq0$ and $I_{f}(x,-T\sqrt{cu/\kappa},-A,c)\leq0$ for $u\in[0,(\kappa/c)(A/T)^2]$. On the other hand, 
\begin{equation}
\label{eq:gabrielL}    
G_{\phi}(x,c_1) \geq  \omega_{2}(x)\left(\left(\int_{-T}^{-H}+\int_{H}^{T}\right)\left|f(x+\ie y)\right|^{\frac{1}{c_1}}\dif{y}\right)^{c_1} = \omega_{2}(x)\left(I_{f}(x,-T,-H,c_1)+I_{f}(x,H,T,c_1)\right)^{c_1}
\end{equation}
for $x\in[\alpha,\beta]$, $c_1>0$ and $T>H>0$ if $x=1$ and $m\neq 0$. To obtain an upper bounds for $G_{\phi}(\alpha,\lambda)$ and $G_{\phi}(\beta,\mu)$ from~\eqref{eq:Gabriel}, we are using~\eqref{eq:gabrielU3} with $x=\alpha$, $c=\lambda$ and $A=A_1$ for the former function, and with $x=\beta$, $c=\mu$ and $A=A_2$ for the latter function. Lower bound for $G_{\phi}(x,w_{1}\lambda+w_{2}\mu)$ from~\eqref{eq:Gabriel} is given by~\eqref{eq:gabrielL} for $c_1=w_{1}\lambda+w_{2}\mu$. The claim of Theorem~\ref{thm:FXbyGabriel} now follows by Lemma~\ref{lem:Gabriel} and
\[
\alpha^{2}(\beta-x) + \beta^{2}(x-\alpha) - x^{2}(\beta-\alpha) \leq \frac{1}{4}(\beta-\alpha)^{3}.
\]
The latter is true since the function on the left-hand side of the above inequality in the variable $x\in[\alpha,\beta]$ attains the maximum value at $x=(\alpha+\beta)/2$.
\end{proof}

Our main application of Theorem~\ref{thm:FXbyGabriel} is to the mollifier $f_{X}(s)$ of the Riemann zeta-function. Let $X\geq 1$ be some parameter, independent of a complex variable $s=\sigma+\ie t$, and define
\begin{equation}
\label{eq:fXMX}
f_{X}(s) \de \zeta(s) M_{X}(s)-1, \quad M_{X}(s)\de \sum_{n\leq X} \frac{\mu(n)}{n^{s}},
\end{equation}
where $\mu(n)$ is the M\"{o}bius function. Clearly, $f_X(s)$ is a holomorphic function on $\C\setminus\{1\}$ with a simple pole at $s=1$, $f_{X}(\bar{s})=\overline{f_X(s)}$, and also that $f_{X}(s)\ll X|t|$ for $\sigma\geq 0$ and large $|t|$. For $0\leq A\leq B$ and $c>0$ define 
\[
F_{X}(\sigma,A,B,c) \de \int_{A}^{B}\left|f_{X}(\sigma+\ie t)\right|^{\frac{1}{c}}\dif{t}.
\]
This function is well-defined for all $\sigma\in\R$ if $A>0$, and for all $\sigma\in\R\setminus\{1\}$ if $A=0$. The next corollary is essentially a slightly simplified inequality~\eqref{eq:FXbyGabriel} for the mollifier $f_{X}(s)$. This result is crucial in the proof of Theorem~\ref{thm:mainZD}.

\begin{corollary}
\label{cor:FXbyGabriel}
Let $1/2\leq \alpha<1<\beta\leq 3/2$, $T\geq H>1$, $\kappa>0$, $A\geq 0$, and also let $\lambda$ and $\mu$ be arbitrary positive numbers. Then
\begin{multline}
\label{eq:FXbyGabriel2}
\left(F_{X}\left(\sigma,H,T,w_{1}\lambda+w_{2}\mu\right)\right)^{w_{1}\lambda+w_{2}\mu} \leq \left(1-\frac{1}{H}\right)^{-1}\exp{\left(\kappa\left(1+\left(\frac{\beta-\alpha}{2T}\right)^{2}\right)\right)} \times \\
\left(F_{X}(\alpha,0,A,\lambda)+\int_{\frac{\kappa}{\lambda}\left(\frac{A}{T}\right)^2}^{\infty}e^{-u}F_{X}\left(\alpha,A,T\sqrt{\frac{\lambda u}{\kappa}},\lambda\right)\dif{u}\right)^{w_1\lambda}\left(\int_{0}^{\infty}e^{-u}F_{X}\left(\beta,0,T\sqrt{\frac{\mu u}{\kappa}},\mu\right)\dif{u}\right)^{w_2\mu}
\end{multline}
for $\sigma\in[\alpha,\beta]$, where $w_1=w_1(\sigma,\alpha,\beta)$ and $w_2=w_2(\sigma,\alpha,\beta)$ are given by~\eqref{eq:w1w2}.
\end{corollary}

\begin{proof}
Inequality~\eqref{eq:FXbyGabriel2} is obviously true in the case $H=T$, so we can assume that $T>H>1$. The claim then follows by Theorem~\ref{thm:FXbyGabriel} for $f=f_{X}(s)$, $m=1$, $A_1=A$ and $A_2=0$ since $f_{X}(\bar{s})=\overline{f_{X}(s)}$,
\[
\frac{(\sigma-1)^{2}+H^2}{\sigma^2+H^2}\geq 1-\frac{8}{9+(2H)^2}\geq \left(1-\frac{1}{H}\right)^{2}
\]
for $\sigma\in[1/2,3/2]$ and $H>1$. Note that all the conditions of Theorem~\ref{thm:FXbyGabriel} are satisfied by the properties of $f_{X}(s)$ mentioned before. 
\end{proof}

Theorem~\ref{thm:FXbyGabriel} can also be used to prove a convexity property for the power moments of holomorphic functions in a strip.

\begin{corollary}
\label{cor:convexity}
Let $T\geq T_0\geq2$ and $1/2\leq\alpha<\beta<1$. Assume that $f(z)$ is a holomorphic function in the strip $x=\Re\{z\}\in(\alpha,\beta)$ with continuous extension to the boundary, such that the conditions (2) and (3) from Theorem~\ref{thm:FXbyGabriel} are satisfied. Assume also that
\[
\int_{-T}^{T}\left|f(\alpha+\ie y)\right|^{b_1}\dif{y} \leq M_1 T^{a_1}\left(\log{T}\right)^{c_1}, \quad \int_{-T}^{T}\left|f(\beta+\ie y)\right|^{b_2}\dif{y} \leq M_2 T^{a_2}\left(\log{T}\right)^{c_2}
\]
for all $T\geq T_0$, where $M_i>0$, $a_i>0$, $b_i\geq1$ and $c_i\geq 0$, $i\in\{1,2\}$, are some constants that may depend on $f(z)$. Then
\[
\int_{-T}^{T}\left|f(x+\ie y)\right|^{\mu(b_1,b_2;x)}\dif{y} \leq M(x) T^{\mu(a_1,a_2;x)}\left(\log{T}\right)^{\mu(c_1,c_2;x)}
\]
for all $T\geq T_0$ and $x\in[\alpha,\beta]$, where
\[
\mu(k,l;x) \de \frac{kb_2(\beta-x)+lb_1(x-\alpha)}{b_2(\beta-x)+b_1(x-\alpha)}
\]
and, for $\kappa>0$,
\begin{gather*}
M(x)\de \exp{\left(\kappa\mu(b_1,b_2;x)\left(1+\left(\frac{\beta-\alpha}{2T_0}\right)^{2}\right)\right)}C_1^{\mu(1,0;x)}C_2^{\mu(0,1;x)}, \\
C_i\de M_i\left(1+\int_{0}^{\infty}e^{-u}\left(\frac{u}{\kappa b_i}\right)^{\frac{a_i}{2}}\left(1+\frac{\log{\left(1+\frac{u}{\kappa b_i}\right)}}{2\log{T_0}}\right)^{c_i}\dif{u}\right).
\end{gather*}
\end{corollary}

\begin{proof}
We are using Theorem~\ref{thm:FXbyGabriel} with $m=0$, $H=0$, $A_1=A_2=T_0$, $\lambda=1/b_1$ and $\mu=1/b_2$. By the assumptions we have $I_{f}\left(\alpha,-T_0,T_0,1/b_1\right) \leq M_1T^{a_1}\left(\log{T}\right)^{c_1}$ and
\begin{flalign}
\label{eq:convexity}
I_{f}\left(\alpha,-T\sqrt{\frac{u}{\kappa b_1}},-T_0,\frac{1}{b_1}\right) &+ I_{f}\left(\alpha,T_0,T\sqrt{\frac{u}{\kappa b_1}},\frac{1}{b_1}\right) \leq \int_{-T\sqrt{\frac{u}{\kappa b_1}}}^{T\sqrt{\frac{u}{\kappa b_1}}}\left|f(\alpha+\ie y)\right|^{b_i}\dif{y} \nonumber \\ 
&\leq M_1T^{a_1}\left(\log{T}\right)^{c_1}\left(\frac{u}{\kappa b_1}\right)^{\frac{a_1}{2}}\left(1+\frac{\log\left(1+\frac{u}{\kappa b_1}\right)}{2\log{T_0}}\right)^{c_1}
\end{flalign}
since $u\geq \kappa b_1\left(T_0/T\right)^{2}$. Both inequalities hold also for $\beta$, $M_2$, $a_2$, $b_2$ and $c_2$ in place of $\alpha$, $M_1$, $a_1$, $b_1$ and $c_1$, respectively. Observe that the right-hand side of~\eqref{eq:convexity} is non-negative for all $u\geq0$. The result now follows since $1/\left(w_1\lambda+w_2\mu\right)=\mu(b_1,b_2;x)$, $w_1\lambda/\left(w_1\lambda+w_2\mu\right)=\mu(1,0;x)$, $w_2\mu/\left(w_1\lambda+w_2\mu\right)=\mu(0,1;x)$, $k\mu(1,0;x)+l\mu(0,1;x)=\mu(k,l;x)$, and $\mu(1,1;x)=1$.
\end{proof}

It should be noted that Corollary~\ref{cor:convexity} also holds for $\beta\geq 1$ if $f(z)$ is holomorphic in the new extended strip. It can be applied to general $L$-functions that belong to the Selberg class or are defined as in~\cite[Chapter~5]{IKANT} since for these functions the growth conditions from Theorem~\ref{thm:FXbyGabriel} are satisfied by the convexity estimates. In the case of the Riemann zeta-function, Corollary~\ref{cor:convexity} is related to some results from~\cite[Sections~8.2 and~8.5]{Ivic} and~\cite[Section~7.8]{titchmarsh1986theory}, and could have some effective applications to the generalized divisor problem~\cite{BellottiYang}.

\section{Proof of Theorem~\ref{thm:mainZD}}
\label{sec:ProofThm1}

As is customary in the classical theory of zero density estimates, let
\[
h_{X}(s) \de 1-f_{X}^2(s) = \left(2-\zeta(s)M_{X}(s)\right)\zeta(s)M_{X}(s),
\]
where $f_X(s)$ and $M_X(s)$ are defined by~\eqref{eq:fXMX}. Note that the zeros of $\zeta(s)$ are also the zeros of $h_{X}(s)$. Let $1/2<\sigma'<1<\sigma''$ and $0<H\leq T$. Assume that $h_{X}(s)$ does not vanish on the boundary of the open rectangle $\left\{z\in\C\colon \sigma'<\Re\{z\}<\sigma'',H<\Im\{z\}<T\right\}$. Then, by Littlewood's lemma~\cite[Section~9.9]{titchmarsh1986theory},
\begin{multline}
\label{eq:Littlewood}
N(\sigma,T)-N(\sigma,H) \leq \frac{1}{2\pi\left(\sigma-\sigma'\right)}\Biggl(\int_{H}^{T}\log{\left|h_{X}\left(\sigma'+\ie t\right)\right|}\dif{t}-\int_{H}^{T}\log{\left|h_{X}\left(\sigma''+\ie t\right)\right|}\dif{t} \\
+\int_{\sigma'}^{\sigma''}\arg{h_{X}\left(u+\ie T\right)}\dif{u}-\int_{\sigma'}^{\sigma''}\arg{h_{X}\left(u+\ie H\right)}\dif{u}\Biggr)
\end{multline}
for $\sigma\in\left[\sigma',\sigma''\right]$. Before proceeding to the proof of Theorem~\ref{thm:mainZD}, we need a few technical results. Corollary~\ref{cor:FXbyGabriel} is used for the estimation of the first integral, see Lemma~\ref{lem:FX}, where two additional results on the second power moments of $M_X(1/2+\ie t)$ and $f_X(1+\delta/\log{X}+\ie t)$ are needed, see Lemmas~\ref{lem:Mxone_half} and~\ref{lem:fxone}. Estimates for the rest of the integrals in~\eqref{eq:Littlewood} are provided by Lemmas~\ref{lem:3/2} and~\ref{lem:arg}.

\begin{lemma}
\label{lem:Mxone_half}
Let $T>0$, $X\geq X_{0}\geq 10^{9}$. Then 
\begin{equation*}
\int_{0}^{T} \left|M_{X}\left(\frac{1}{2} + \ie t\right)\right|^{2} \dif{t} \leq \left(C_{1}(X_0)T + C_{2}(X_0)X\right)\log{X},
\end{equation*}
where 
\begin{equation*}
C_{1}(X_0) \de \frac{6}{\pi^{2}} + \frac{b_{2}}{\log X_{0}}, \quad C_{2}(X_0) \de \frac{\pi m_{0} b_{1}}{\log X_{0}} + \frac{6 m_{0}}{ \pi X_{0}} + \frac{\pi m_{0} b_{2}}{ X_{0} \log X_{0}},
\end{equation*}
with $b_{1}\de 0.62$, $b_{2}\de 1.048$, and $m_0\de \sqrt{1+(2/3)\sqrt{6/5}}<1.31541$.
\end{lemma}

\begin{proof}
See~\cite[Lemma 4.1]{KADIRI201822}.
\end{proof}

\begin{lemma}
\label{lem:fxone}
Let $T>0$, $X\geq X_{0}\geq 3\cdot10^{10}$, and $0<\delta\leq \frac{1}{20}\log{X}$. Then
\begin{equation*}
F_{X}\left(1+ \frac{\delta}{\log X},0,T,\frac{1}{2}\right) \leq  \left( C_{3}(X_0,\delta) + \frac{C_{4}(X_0,\delta)\left(T+\pi m_{0}\right)}{X}\right) \log {X}, 
\end{equation*}
where 
\begin{equation*} 
C_{3}(X_0,\delta) \de \frac{\pi m_{0}b_4}{2\delta}\exp{\left(\frac{2\delta\gamma}{\log{X_0}}\right)}, \quad
 C_{4}(X_0,\delta) \de \frac{b_{4}}{5\delta e^{\delta}}  + \frac{b_{3}e^{-2 \delta}}{\log{X_{0}}},
\end{equation*}
with $b_{3}\de 0.605$, $b_{4}\de 0.54$, and $m_0$ as in Lemma~\ref{lem:Mxone_half}.
\end{lemma}

\begin{proof}
For $\tau\in(1,11/10]$ and $X\geq 3\cdot10^{10}$ we have\footnote{The second inequality through a private communication with Prof.~Ramar\'{e}, since he and Zuniga Alterman had recently proved it and their result will be soon announced.}
\[
\sum_{n=1}^{\infty}\frac{\lambda_{X}^{2}(n)}{n^{\tau}} \leq \sum_{n\leq X}\frac{1}{n^\tau}\left(\sum_{\substack{d\mid n\\ d\leq X}}\mu(d)\right)^{2} + \sum_{n>X}\frac{\lambda_{X}^{2}(n)}{n^{\tau}} = \zeta(\tau)\sum_{d_1\leq X}\sum_{d_2\leq X}\frac{\mu(d_1)\mu(d_2)}{[d_1,d_2]^\tau} \leq \frac{b_{4}e^{\gamma(\tau-1)}}{\tau-1},
\]
where $\lambda_{X}(n)$ are Dirichlet coefficients of $f_{X}(s)$. The proof now closely follows that of~\cite[Lemmas~3.4 and~4.3]{KADIRI201822}, namely that
\begin{flalign*}
F_{X}\left(\sigma,0,T,\frac{1}{2}\right) &\leq \pi m_0\sum_{n=1}^{\infty}\frac{\lambda_{X}^{2}(n)}{n^{2\sigma-1}} + \left(T+\pi m_0\right)\sum_{n=1}^{\infty}\frac{\lambda_{X}^{2}(n)}{n^{2\sigma}} \\ 
&\leq \frac{\pi m_0b_4}{2}\frac{e^{2\gamma(\sigma-1)}}{\sigma-1} + \frac{T+\pi m_0}{X}\left(\frac{b_3}{X^{2(\sigma-1)}}+\frac{1}{5(5X)^{\sigma-1}}\sum_{n=1}^{\infty}\frac{\lambda_{X}^{2}(n)}{n^{\sigma}}\right) \\
&\leq \left(\frac{\pi m_0b_4}{2}e^{2\gamma(\sigma-1)}+\left(\frac{(\sigma-1)b_3}{X^{2(\sigma-1)}}+\frac{b_4}{5(5X)^{\sigma-1}}e^{\gamma(\sigma-1)}\right)\frac{T+\pi m_0}{X}\right)\frac{1}{\sigma-1}
\end{flalign*}
for $\sigma\in(1,21/20]$ and $X\geq 3\cdot 10^{10}$. The proof now follows by taking $\sigma=1+\delta/\log{X}$ in the above inequality.
\end{proof}

\begin{lemma}
\label{lem:3/2}
Let $T\geq H\geq 2\pi m_0$, $T\geq T_0>0$ and $X=hT$ for $h\geq 10^{9}/T_0$. Then
\[
-\int_{H}^{T}\log{\left|h_{X}\left(\frac{3}{2}+\ie t\right)\right|}\dif{t} \leq C_{5}(hT_0,h),
\]
where
\begin{equation}
\label{eq:C5}   
C_{5}(u,h) \de -\log{\left(1-9\left(1+\frac{2+\gamma+\frac{7}{36u}}{\log{u}}\right)^{2}\frac{\log^{2}{u}}{u}\right)}\frac{\left(\frac{3}{h}+8\pi m_0\right)\left(1+\frac{3.31}{\log{u}}\right)}{9\left(1+\frac{2+\gamma+\frac{7}{36u}}{\log{u}}\right)^{2}}\log{u}
\end{equation}
with $m_0$ as in Lemma~\ref{lem:Mxone_half}.
\end{lemma}

\begin{proof}
We closely follow the proofs of Corollary 4.11 and Lemma 4.13 from~\cite{KADIRI201822}. Note that Corollary~4.11 is valid also for $\eta=1/2$ by the continuation principle. Then $\left|f_{X}(3/2+\ie t)\right|\leq b_{6}(X_0,1/2)$ for $X\geq X_0\geq 10^9$, where the function $b_6(X,\eta)$ is from~\cite[Equation~(4.45)]{KADIRI201822}. Our conditions imply $X_0\geq 10^9$ by setting $X_0=hT_0$. Then 
\[
-\log{\left|h_{X}\left(\frac{3}{2}+\ie t\right)\right|} \leq -\frac{\log{\left(1-b_6(hT_0,1/2)^2\right)}}{b_6(hT_0,1/2)^2}\left|f_{X}\left(\frac{3}{2}+\ie t\right)\right|^{2}.
\]
Moreover, 
\[
\int_{H}^{T}\left|f_{X}\left(\frac{3}{2}+\ie t\right)\right|^{2}\dif{t} \leq \left(\frac{3}{h}+8\pi m_0\right)\left(1+\frac{3.31}{\log{(hT_0)}}\right)\frac{\log^{3}(hT_0)}{hT_0}
\]
by the proof of~\cite[Lemma 4.13]{KADIRI201822}, where we notice that $T-H+2\pi m_0\leq T$ and $b_{11}(X,3)\leq b_{11}(X,2)\leq 1+3.31/\log{X}$ for $X\geq 10^9$, where $b_{11}(X,\tau)$ is from~\cite[Equation~(4.68)]{KADIRI201822}. The statement of Lemma~\ref{lem:3/2} now easily follows.
\end{proof}

\begin{remark}
Note that, for fixed $h$, $0<C_{5}(u,h)\sim \left(3/h+8\pi m_0\right)u^{-1}\log^{3}{u}\to 0$ as $u\to\infty$.
\end{remark}

\begin{lemma}
\label{lem:arg}
Let $T\geq H\geq H_0\geq 1002$, $T\geq T_0\geq e^{2d}$ for $d>0$, $1/2+d/\log{T}\leq\sigma\leq1$ and $X=hT$ for $h\geq 10^{9}/T_0$. Then 
\begin{multline*}
\left| \int_{\sigma-\frac{d}{\log{T}}}^{\frac{3}{2}} \arg h_{X}(u+\ie T) \dif{u}-\int_{\sigma-\frac{d}{\log{T}}}^{\frac{3}{2}} \arg h_{X}(u+\ie H) \dif{u} \right| \leq \left(\frac{3}{2}- \sigma+\frac{d}{\log{T_0}}\right)\times \\
\times\left(C_{6}\left(hT_0,H_0\right)+\frac{2}{\log{2}}\log{(HT)}+\frac{3\pi}{\log{2}}\log{(hT)}\right),
\end{multline*}
where 
\begin{flalign}
\label{eq:C6}
C_{6}(u,H_0) &\de 220 + \frac{\pi}{\log{2}}\log{\left(1+\frac{1}{u^{3/2}}\right)} + \frac{4}{\log{2}}\log{\left(\frac{1}{2\pi}\sqrt{\left(\frac{5}{2H_0}\right)^{2}+\left(1+\frac{2}{H_0}\right)^{2}}\right)} \nonumber \\
&-\frac{2\pi}{\log{2}}\log{\left(1-9\left(1+\frac{2+\gamma+\frac{7}{36u}}{\log{u}}\right)^{2}\frac{\log^{2}{u}}{u}\right)}.
\end{flalign}

\end{lemma}

\begin{proof}
This follows by the proof of Lemma~4.12 from~\cite{KADIRI201822}, by using similar arguments as in the proof of Lemma~\ref{lem:3/2}, and by several numerical simplifications.  
\end{proof}

\begin{lemma}
\label{lem:FX}
Let $T\geq H\geq H_0>1$, $T\geq T_0\geq e^e$, and $\kappa>0$. Also, let $X=hT$ for $h\geq 3\cdot10^{10}/T_0$, and let $w_1=w_1\left(\sigma,1/2,1+\delta/\log{X}\right)$ and $w_2=w_2\left(\sigma,1/2,1+\delta/\log{X}\right)$ for $0<\delta\leq\frac{1}{20}\log{(hT_0)}$ and with $w_1$ and $w_2$ as in~\eqref{eq:w1w2}. Then    
\begin{flalign*}
F_{X}\left(\sigma,H,T,\frac{3}{4}w_1+\frac{1}{2}w_2\right) &\leq L_0^{\frac{2}{2-\sigma}} K_3^{\frac{2\sigma-1}{2-\sigma}}\left(L_1\cdot\left(K_1T\right)^{\frac{3(1-\sigma)}{2-\sigma}}\left(\log{T}\right)^{\frac{4(1-\sigma)}{2-\sigma}}\left(\log{(hT)}\right)^{\frac{1}{2-\sigma}}\right. \\ 
&\left.+ L_2\cdot \left(K_2T\right)^{\frac{3(1-\sigma)}{2-\sigma}}\left(\log{(hT)}\right)^{\frac{2\sigma-1}{2-\sigma}}\right) 
\end{flalign*}
for $1/2\leq\sigma\leq 1$. Functions $L_0$ and $K_3$ are defined by
\begin{equation}
\label{eq:L0}
L_0 = L_0(T_0,H_0,\kappa,h,\delta) \de \frac{\exp{\left(\kappa\left(1+\left(\frac{1}{4T_0}\left(1+\frac{2\delta}{\log{(hT_0)}}\right)\right)^{2}\right)\right)}}{1-H_{0}^{-1}},
\end{equation}
and 
\begin{equation}
\label{eq:K3}
K_3=K_3(T_0,\kappa,h,\delta) \de \int_{0}^{\infty}e^{-u}\left(C_3(hT_0,\delta)+\frac{C_4(hT_0,\delta)}{h}\left(\sqrt{\frac{u}{2\kappa}}+\frac{\pi m_0}{T_0}\right)\right)\dif{u},
\end{equation}
where $C_3(X_0,\delta)$, $C_4(X_0,\delta)$ and $m_0$ are from Lemma~\ref{lem:fxone}. Moreover,
\begin{multline}
\label{eq:K1}
K_1=K_1(T_0,A_0,\kappa,h) \de \left(2a_2(A_0)\sqrt{\frac{3}{4\kappa}}\right)^{\frac{1}{3}}\int_{0}^{\infty}e^{-u}u^{\frac{1}{6}}\left(1+\frac{\log{\left(1+\frac{3u}{4\kappa}\right)}}{2\log{T_0}}\right)^{\frac{4}{3}}\times \\
\times \left(C_1(hT_0)\sqrt{\frac{3u}{4\kappa}}+hC_2(hT_0)\right)^{\frac{2}{3}}\dif{u}
+ \frac{\left(2a_1(A_0)\right)^{\frac{1}{3}}\left(hC_2(hT_0)+\frac{A_0C_1(hT_0)}{T_0}\right)^{\frac{2}{3}}}{T_0^{\frac{1}{3}}\log^{\frac{4}{3}}{T_0}}
\end{multline}
with $C_1(X_0)$ and $C_2(X_0)$ as in Lemma~\ref{lem:Mxone_half}, and $a_1(A_0)$, $a_2(A_0)$ and $A_0$ as in Corollary~\ref{cor:4thPM}, while
\begin{equation}
\label{eq:K2}
K_2=K_2(T_0,A_0,\kappa) \de 2^{\frac{1}{3}}\sqrt{\frac{3}{4\kappa}}\int_{0}^{\infty}e^{-u}\sqrt{u}\dif{u} + 2^{\frac{1}{3}}\frac{4\kappa}{3}\left(\frac{A_0}{T_0}\right)^{3}.
\end{equation}
Finally,
\begin{multline}
\label{eq:L1}
L_1 = L_1(T_0,K_1,K_3,h,\delta)
\de \max\left\{1,K_1^{\frac{3\delta}{\log{(hT_0)}}}\right\}\max\left\{1,K_3^{-\frac{3\delta}{\log{(hT_0)}}}\right\}\times \\
\times\exp{\left(3\delta\max\left\{1,\frac{\log{T_0}}{\log{(hT_0)}}\right\}+4\delta\max\left\{\frac{\log{\log{T_0}}}{\log{(hT_0)}},\frac{\log{\log{T_0}}}{\log{T_0}}\right\}\right)}
\end{multline}
and
\begin{equation}
\label{eq:L2}
L_2 = L_2(T_0,K_2,K_3,h,\delta) \de \max\left\{1,K_2^{\frac{3\delta}{\log{(hT_0)}}}\right\}\max\left\{1,K_3^{-\frac{3\delta}{\log{(hT_0)}}}\right\}\exp{\left(3\delta\max\left\{1,\frac{\log{T_0}}{\log{(hT_0)}}\right\}\right)},
\end{equation}
where $K_1$, $K_2$ and $K_3$ are given by~\eqref{eq:K1},~\eqref{eq:K2} and~\eqref{eq:K3}, respectively.
\end{lemma}

\begin{proof}
We are using Corollary~\ref{cor:FXbyGabriel} with $A=A_0$, $\alpha=1/2$, $\beta=1+\delta/\log{X}$, $\lambda=3/4$ and $\mu=1/2$. Firstly, we estimate the first term in the second line of~\eqref{eq:FXbyGabriel2}. Note that
\begin{equation}
\label{eq:MinHol}
F_{X}\left(\alpha,U,V,\lambda\right) \leq 2^{\frac{1}{3}}\mathcal{M}_{2}(U,V)^{\frac{1}{3}}\left(\int_{0}^{V}\left|M_{X}\left(\frac{1}{2}+\ie t\right)\right|^{2}\dif{t}\right)^{\frac{2}{3}} + 2^{\frac{1}{3}}(V-U)
\end{equation}
for $V\geq U\geq 0$ by Minkowski's and H\"{o}lder's inequalities. Take $U=A_0$ and $V=T\sqrt{\lambda u/\kappa}$. Note that $u\geq(\kappa/\lambda)(A/T)^2$ due to the domain of integration in~\eqref{eq:FXbyGabriel2}, implying $V\geq U$. Then
\[
\mathcal{M}_{2}(U,V) \leq a_2(A_0)\sqrt{\frac{3u}{4\kappa}}\left(1+\frac{\log{\left(1+\frac{3u}{4\kappa}\right)}}{2\log{T_0}}\right)^{4}T\log^{4}{T}
\]
by Corollary~\ref{cor:4thPM}, and
\[
\int_{0}^{V}\left|M_{X}\left(\frac{1}{2}+\ie t\right)\right|^{2}\dif{t}\leq \left(C_1(hT_0)\sqrt{\frac{3u}{4\kappa}}+hC_2(hT_0)\right)T\log{(hT)}
\]
by Lemma~\ref{lem:Mxone_half}, where we set $X_0=hT_0$. Also,
\[
F_{X}\left(\alpha,0,A,\lambda\right) \leq \left(2a_1(A_0)\right)^{\frac{1}{3}}\left(hC_2(hT_0)+\frac{A_0C_1(hT_0)}{T_0}\right)^{\frac{2}{3}}T^{\frac{2}{3}}\log^{\frac{2}{3}}{(hT)} + 2^{\frac{1}{3}}A_0
\]
by~\eqref{eq:MinHol} for $U=0$ and $V=A_0$, and then again by Corollary~\ref{cor:4thPM} and Lemma~\ref{lem:Mxone_half}. Therefore,
\[
F_{X}\left(\alpha,0,A,\lambda\right)+\int_{\frac{\kappa}{\lambda}\left(\frac{A}{T}\right)^2}^{\infty}e^{-u}F_{X}\left(\alpha,A,T\sqrt{\frac{\lambda u}{\kappa}},\lambda\right)\dif{u} \leq K_1T\log^{\frac{4}{3}}{T}\cdot\log^{\frac{2}{3}}{(hT)} + K_2T,
\]
where $K_1$ and $K_2$ are defined by~\eqref{eq:K1} and~\eqref{eq:K2}, respectively. Secondly, for the second term in the second line of~\eqref{eq:FXbyGabriel2} we simply have
\[
\int_{\frac{\kappa}{\mu}\left(\frac{A}{T}\right)^2}^{\infty}e^{-u}F_{X}\left(\beta,0,T\sqrt{\frac{\mu u}{\kappa}},\mu\right)\dif{u} \leq K_3\log(hT)
\]
by Lemma~\ref{lem:fxone}, where $K_3$ is defined by~\eqref{eq:K3}. Remembering that $(a+b)^{x}\leq a^{x}+b^{x}$ for $x\in(0,1]$ and $a,b\geq 0$, and writing $w=\frac{3}{4}w_{1}+\frac{1}{2}w_{2}$, we obtain
\begin{multline}
\label{eq:FXraw}
F_{X}\left(\sigma,H,T,w\right) \leq \left(\frac{\exp{\left(\kappa\left(1+\left(\frac{1}{4T}\left(1+\frac{2\delta}{\log{(hT)}}\right)\right)^{2}\right)\right)}}{1-H^{-1}}\right)^{\frac{1}{w}}\times \\
\times\left(\left(K_1^{\frac{\frac{3}{4}w_1}{w}}K_3^{\frac{\frac{1}{2}w_2}{w}}\right)T^{\frac{\frac{3}{4}w_1}{w}}\left(\log{T}\right)^{\frac{w_1}{w}}\left(\log{(hT)}\right)^{\frac{\frac{1}{2}\left(w_1+w_2\right)}{w}} + \left(K_2^{\frac{\frac{3}{4}w_1}{w}}K_3^{\frac{\frac{1}{2}w_2}{w}}\right)T^{\frac{\frac{3}{4}w_1}{w}}\left(\log{(hT)}\right)^{\frac{\frac{1}{2}w_2}{w}}\right)
\end{multline}
by Corollary~\ref{cor:FXbyGabriel}. We need to estimate the resulting terms on the right-hand side of~\eqref{eq:FXraw}. In what follows, remember that $\sigma\in[1/2,1]$. Because
\[
\frac{3(1-\sigma)}{2-\sigma}\leq \frac{\frac{3}{4}w_1}{w} = \frac{3(1-\sigma)+\frac{3\delta}{\log{(hT)}}}{2-\sigma+\frac{3\delta}{\log{(hT)}}} \leq \frac{3(1-\sigma)}{2-\sigma} + \frac{3\delta}{\log{(hT_0)}}, 
\]
we have
\begin{equation}
\label{eq:K1raw}
K_1^{\frac{\frac{3}{4}w_1}{w}} \leq \max\left\{1,K_1^{\frac{3\delta}{\log{(hT_0)}}}\right\}K_1^{\frac{3(1-\sigma)}{2-\sigma}}
\end{equation}
and
\begin{equation}
\label{eq:Traw}
T^{\frac{\frac{3}{4}w_1}{w}} \leq \exp{\left(\frac{3\delta\log{T}}{\log{(hT)}}\right)}T^{\frac{3(1-\sigma)}{2-\sigma}} \leq \exp{\left(3\delta\max\left\{1,\frac{\log{T_0}}{\log{(hT_0)}}\right\}\right)}T^{\frac{3(1-\sigma)}{2-\sigma}}.
\end{equation}
Inequality~\eqref{eq:K1raw} also holds for $K_2$ in place of $K_1$. Furthermore, because
\[
\frac{2\sigma-1}{2-\sigma} - \frac{3\delta}{\log{(hT_0)}}\leq \frac{\frac{1}{2}w_2}{w} = \frac{2\sigma-1}{2-\sigma+\frac{3\delta}{\log{(hT)}}} \leq \frac{2\sigma-1}{2-\sigma},
\]
we also have
\begin{equation}
\label{eq:K3raw}
K_3^{\frac{\frac{1}{2}w_2}{w}} \leq \max\left\{1,K_3^{-\frac{3\delta}{\log{(hT_0)}}}\right\}K_3^{\frac{2\sigma-1}{2-\sigma}}, \quad \left(\log{(hT)}\right)^{\frac{\frac{1}{2}w_2}{w}} \leq \left(\log{(hT)}\right)^{\frac{2\sigma-1}{2-\sigma}}
\end{equation}
since $\log{(hT)}\geq 1$. Turning to the other log-terms, because $\log{T}\geq 1$ and
\[
\frac{w_1}{w} = \frac{4(1-\sigma)+\frac{4\delta}{\log{(hT)}}}{2-\sigma+\frac{3\delta}{\log{(hT)}}} \leq \frac{4(1-\sigma)}{2-\sigma}+\frac{4\delta}{\log{(hT)}},
\]
we have
\begin{equation}
\label{eq:logTraw}
\left(\log{T}\right)^{\frac{w_1}{w}} \leq \exp{\left(\frac{4\delta\log{\log{T}}}{\log{(hT)}}\right)}\left(\log{T}\right)^{\frac{4(1-\sigma)}{2-\sigma}} \leq \exp{\left(4\delta\max\left\{\frac{\log{\log{T_0}}}{\log{(hT_0)}},\frac{\log{\log{T_0}}}{\log{T_0}}\right\}\right)}\left(\log{T}\right)^{\frac{4(1-\sigma)}{2-\sigma}},
\end{equation}
where we also used $T_0\geq e^{e}$. Moreover, $w_1+w_2=1$ and
\begin{equation}
\label{eq:recipw}  
\frac{1}{w} = \frac{2+\frac{4\delta}{\log{(hT)}}}{2-\sigma+\frac{3\delta}{\log{(hT)}}} \leq \frac{2}{2-\sigma},
\end{equation}
which implies
\begin{equation}
\label{eq:loghTraw}
\left(\log{(hT)}\right)^{\frac{\frac{1}{2}\left(w_1+w_2\right)}{w}} \leq \left(\log{(hT)}\right)^{\frac{1}{2-\sigma}}.
\end{equation}
In addition, estimate~\eqref{eq:recipw} is used to bound the first term on the right-hand side of~\eqref{eq:FXraw}, and this contributes $L_0$ from~\eqref{eq:L0} to the final estimate. Inequalities~\eqref{eq:K1raw}--\eqref{eq:logTraw} and~\eqref{eq:loghTraw} contribute to $L_1$ from~\eqref{eq:L1}, while inequalities~\eqref{eq:K1raw} (for $K_2$ in place of $K_1$),~\eqref{eq:Traw} and~\eqref{eq:K3raw}, contribute to $L_2$ from~\eqref{eq:L2}. The proof of Lemma~\ref{lem:FX} is thus complete.
\end{proof}

\begin{proof}[Proof of Theorem~\ref{thm:mainZD}]
We are using~\eqref{eq:Littlewood} with $\sigma'=\sigma-d/\log{T}$ and $\sigma''=3/2$. Note that $\sigma'\in(1/2,1)$. Temporarily assume the non-vanishing condition on $h_{X}(s)$. The second integral in~\eqref{eq:Littlewood} is estimated by Lemma~\ref{lem:3/2}, while the last two integrals are estimated by Lemma~\ref{lem:arg}. Note that the conditions of the lemmas are satisfied by the conditions of Theorem~\ref{thm:mainZD}. Therefore, it remains to estimate the first integral in~\eqref{eq:Littlewood}. Because $\log{\left(1+u^2\right)}\leq u^{a}$ for $u\geq 0$ and $a\in[4/3,2]$, we have
\[
\log{\left|h_{X}(s)\right|} \leq \log{\left(1+\left|f_{X}(s)\right|^{2}\right)} \leq \left|f_{X}(s)\right|^{\frac{1}{w}}, \quad w = \frac{3}{4}w_1\left(\sigma',\frac{1}{2},1+\frac{\delta}{\log{X}}\right)+\frac{1}{2}w_2\left(\sigma',\frac{1}{2},1+\frac{\delta}{\log{X}}\right)
\]
for $\delta>0$, and with $w_1$ and $w_2$ as in~\eqref{eq:w1w2}. Thus,
\[
\int_{H}^{T}\log{\left|h_{X}\left(\sigma'+\ie t\right)\right|}\dif{t} \leq F_{X}\left(\sigma',H,T,w\right).
\]
Taking $X=hT$, we can see that all other conditions of Lemma~\ref{lem:FX} are satisfied by the conditions of Theorem~\ref{thm:mainZD}. Therefore,
\begin{flalign*}
F_{X}&\left(\sigma',H,T,w\right) \leq \max\left\{1,L_0^{-\frac{2d}{\log{T_0}}}\right\}\max\left\{1,K_3^{-\frac{3d}{\log{T_0}}}\right\}L_0^{\frac{2}{2-\sigma}} K_3^{\frac{2\sigma-1}{2-\sigma}}\times \\
&\times\left(L_1\max\left\{1,K_1^{\frac{3d}{\log{T_0}}}\right\}\exp{\left(d\left(3+\frac{4\log{\log{T_0}}}{\log{T_0}}\right)\right)}\left(K_1T\right)^{\frac{3(1-\sigma)}{2-\sigma}}\left(\log{T}\right)^{\frac{4(1-\sigma)}{2-\sigma}}\left(\log{(hT)}\right)^{\frac{1}{2-\sigma}}\right. \\ 
&\left.+ L_2e^{3d}\max\left\{1,K_2^{\frac{3d}{\log{T_0}}}\right\} \left(K_2T\right)^{\frac{3(1-\sigma)}{2-\sigma}}\left(\log{(hT)}\right)^{\frac{2\sigma-1}{2-\sigma}}\right),
\end{flalign*}
where $L_0$, $L_1$, $L_2$ and $K_1$, $K_2$, $K_3$ are as in Lemma~\ref{lem:FX}. It is not hard to see that the statement of Theorem~\ref{thm:mainZD} now follows after combining all estimates. Also, the non-vanishing condition can be removed by invoking the standard perturbation argument since the right-hand side of the main estimate of Theorem~\ref{thm:mainZD} is a continuous function in $\sigma$, $T$ and $H$.
\end{proof}

\section{The fourth power moment of \texorpdfstring{$\zeta(s)$}{zeta(s)} on the critical line}
\label{sec:4thPM}

In this section, we prove Theorem~\ref{thm:4thPMExplicitGeneral}. The proof, which is presented through several subsections in Section~\ref{subsec:ProofThm2}, is technically quite involved and requires estimates on various sums that include the divisor function $d(n)$ and its square. These results, which might be of interest in their own right, are gathered in Section~\ref{subsec:divisorSums}. Finally, we present the method together with the plan of the proof of Theorem~\ref{thm:4thPMExplicitGeneral} in the following section.

\subsection{The method}
\label{subsec:4thPMmethod}

Ingham's proof of~\eqref{eq:4thPMGeneral} is based on an approximate functional equation for $\zeta^2(s)$ in symmetric form and is relatively complicated. A shorter proof based on a Tauberian theorem for integrals but that produces only an asymptotic estimate (without an error term) can be found in~\cite[Section~7]{titchmarsh1986theory}. However, a very elegant proof of~\eqref{eq:4thPMGeneral} was found by Ramachandra~\cite{Ramachandra75}. His proof, which is reproduced and extended in~\cite[Section~5.3]{Ivic}, is based on a qualitative version of the following mean value theorem for Dirichlet polynomials.

\begin{lemma}
\label{lem:MV}
Let $T_2\geq T_1>0$ and let $\{a_n\}_{n=1}^{\infty}$ be a sequence of complex numbers. Then
\begin{equation}
\label{eq:MV}
\left|\int_{T_1}^{T_2}\left|\sum_{n=1}^{\infty}a_n n^{\ie t}\right|^{2}\dif{t} - \left(T_2-T_1\right)\sum_{n=1}^{\infty}|a_n|^2\right| \leq 2\pi m_0\sum_{n=1}^{\infty}\left(n+\frac{1}{2}\right)|a_n|^{2},
\end{equation}
where $m_0$ is from Lemma~\ref{lem:Mxone_half}, provided that $\sum_{n=1}^{\infty}|a_n|\left(1+n|a_n|\right)$ converges.
\end{lemma}

\begin{proof}
Let $X\geq 1$. We have
\[
\int_{T_1}^{T_2}\left|\sum_{n\leq X}a_n n^{\ie t}\right|^{2}\dif{t} = \left(T_2-T_1\right)\sum_{n\leq X}|a_n|^{2} - \ie\mathop{\sum\sum}_{1\leq n\neq m\leq X}a_n\overline{a_m}\frac{(n/m)^{\ie T_2}-(n/m)^{\ie T_1}}{\log{n}-\log{m}}.
\]
Inequality~\eqref{eq:MV} now follows after employing~\cite[Inequality (D)]{Preissmann} twice on the above double sum, noticing that $\left(\log{\frac{n+1}{n}}\right)^{-1}\leq n+\frac{1}{2}$ for all $n\in\mathbb{N}$, and then taking $X\to\infty$ since all series are convergent due to the assumption. 
\end{proof}

The functional equation for $\zeta(s)$ can be written as $\zeta(s)=\chi(s)\zeta(1-s)$, where
\begin{equation}
\label{eq:chiDef}
\chi(s) \de \frac{(2\pi)^{s}}{2\Gamma(s)\cos{\left(\frac{\pi s}{2}\right)}}.
\end{equation}
This is a meromorphic function on $\C$ with only simple poles at $s=2k-1$ for $k\in\mathbb{N}$. Also note that the functional equation implies $\chi(1/2+\ie t)\chi(1/2-\ie t)=\left|\chi(1/2\pm \ie t)\right|^{2}=1$. The starting point in Ramachandra's proof (which itself resembles the reflection principle~\cite[Section~4.4]{Ivic} or a half-way derivation of a weighted approximate functional equation~\cite[Section~4.5]{IvicZFunctionBook}) is to use the Mellin integral for $e^{-x}$ to get
\[
\sum_{n=1}^{\infty}\frac{d(n)}{n^{1/2+\ie t}}e^{-n/T} = \frac{1}{2\pi\ie}\int_{2-\ie\infty}^{2+\ie\infty}\zeta^{2}\left(\frac{1}{2}+\ie t+z\right)\Gamma(z)T^{z}\dif{z}
\]
for $T\geq 1$. By moving the line of integration to $\Re\{z\}=c_1\in(-1,-1/2)$ we get in addition to
\begin{equation}
\label{eq:intc1}
\frac{1}{2\pi\ie}\int_{c_1-\ie\infty}^{c_1+\ie\infty}\zeta^{2}\left(\frac{1}{2}+\ie t+z\right)\Gamma(z)T^{z}\dif{z}
\end{equation}
also two terms that come from the simple pole of $\Gamma(z)$ at $z=0$ and double pole of $\zeta^2(1/2+\ie t+z)$ at $z=1/2-\ie t$. Next, the integral~\eqref{eq:intc1} can be written as
\begin{equation}
\label{eq:intc12}
\frac{1}{2\pi\ie}\int_{c_1-\ie\infty}^{c_1+\ie\infty}\chi^{2}\left(\frac{1}{2}+\ie t+z\right)\Gamma(z)T^{z}\left(\left(\sum_{n\leq T}+\sum_{n>T}\right)\frac{d(n)}{n^{1/2-\ie t-z}}\right)\dif{z}
\end{equation}
due to the functional equation. By moving the line of integration to $\Re\{z\}=c_2\in(0,1/2)$ only for the first integral (implied by the first sum) in~\eqref{eq:intc12} we get, beside the new integral, also another term that comes from the simple pole of $\Gamma(z)$ at $z=0$. After multiplying the final equation by $\chi^{-1}(1/2+\ie t)$, and then truncating some sums and integrals, it is not hard to see that we obtain 
\begin{equation}
\label{eq:zeta2}
\chi^{-1}(1/2+\ie t)\zeta^{2}(1/2+\ie t)=\left|\zeta(1/2+\ie t)\right|^{2}=\sum_{k=0}^{6}J_k(t),
\end{equation}
where
\begin{gather*}
J_1(t) \de \chi{\left(\frac{1}{2}+\ie t\right)}\sum_{n\leq T}\frac{d(n)}{n^{1/2-\ie t}}, \quad J_2(t)\de \overline{J_1(t)}, \quad J_3(t)\de \chi{\left(\frac{1}{2}-\ie t\right)}\sum_{n>T}\frac{d(n)}{n^{1/2+\ie t}}e^{-n/T}, \\
J_4(t) \de \chi{\left(\frac{1}{2}-\ie t\right)}\sum_{n\leq T}\frac{d(n)}{n^{1/2+\ie t}}\left(e^{-n/T}-1\right), \quad 
J_5(t)\de -\frac{1}{2\pi}\chi{\left(\frac{1}{2}-\ie t\right)}\int_{-\mathfrak{a}_1\left(\log{T}\right)^{\mathfrak{a}_2}}^{\mathfrak{a}_1\left(\log{T}\right)^{\mathfrak{a}_2}}P(0,1,c_1;u)\dif{u}, \\
J_6(t)\de -\frac{1}{2\pi}\chi{\left(\frac{1}{2}-\ie t\right)}\int_{-\mathfrak{b}_1\left(\log{T}\right)^{\mathfrak{b}_2}}^{\mathfrak{b}_1\left(\log{T}\right)^{\mathfrak{b}_2}}P(1,0,c_2;u)\dif{u}
\end{gather*}
and
\begin{multline*}
J_0(t) \de -\frac{1}{2\pi}\chi{\left(\frac{1}{2}-\ie t\right)}\Biggl(\left(\int_{-\infty}^{-\mathfrak{a}_1\left(\log{T}\right)^{\mathfrak{a}_2}}+\int_{\mathfrak{a}_1\left(\log{T}\right)^{\mathfrak{a}_2}}^{\infty}\right)P(0,1,c_1;u)\dif{u} \\
\left.+\left(\int_{-\infty}^{-\mathfrak{b}_1\left(\log{T}\right)^{\mathfrak{b}_2}}+\int_{\mathfrak{b}_1\left(\log{T}\right)^{\mathfrak{b}_2}}^{\infty}\right)P(1,0,c_2;u)\dif{u}+2\pi \Res_{z=\frac{1}{2}-\ie t}\left(\zeta^{2}{\left(\frac{1}{2}+\ie t+z\right)}\Gamma(z)T^{z}\right)\right)
\end{multline*}
for $\mathfrak{a}_1>0$, $\mathfrak{a}_2>1$, $\mathfrak{b}_1>0$, $\mathfrak{b}_2>1$ and
\[
P(a,b,c;u)\de \chi^{2}\left(\frac{1}{2}+c+\ie(t+u)\right)\Gamma\left(c+\ie u\right)T^{c+\ie u}\left(\left(a\sum_{n\leq T}+b\sum_{n>T}\right)\frac{d(n)}{n^{\frac{1}{2}-c-\ie(t+u)}}\right).
\]
Squaring~\eqref{eq:zeta2}, and then integrating the result from $T$ to $2T$ gives
\begin{equation}
\label{eq:4thPMmainEst}   
\mathcal{M}_{2}(T,2T) = \underbrace{2\int_{T}^{2T}\left|J_1(t)\right|^{2}\dif{t}}_{\text{I. part}} + \underbrace{\int_{T}^{2T}\left(J_1^2(t) + J_2^2(t)\right)\dif{t}}_{\text{II. part}} + \mathcal{R},
\end{equation}
where
\begin{multline}
\label{eq:4thPMtermR}
|\mathcal{R}| \leq \underbrace{2\sum_{k=3}^{6}\left|\int_{T}^{2T}\left(J_1(t)+J_2(t)\right)J_{k}(t)\dif{t}\right|}_{\text{III. part}} + \underbrace{\sum_{k=3}^{6}\int_{T}^{2T}\left|J_k(t)\right|^{2}\dif{t}}_{\text{IV. part}} \\ 
+ \underbrace{2\mathop{\sum\sum}_{3\leq n<m\leq 6}\int_{T}^{2T}\left|J_{n}(t)J_{m}(t)\right|\dif{t}}_{\text{V. part}} 
+ \underbrace{2\int_{T}^{2T}\left|J_0(t)\right|\cdot\left|\zeta\left(\frac{1}{2}+\ie t\right)\right|^{2}\dif{t} + \int_{T}^{2T}\left|J_0(t)\right|^{2}\dif{t}}_{\text{VI. part}}.
\end{multline}
The curly brackets in~\eqref{eq:4thPMmainEst} and~\eqref{eq:4thPMtermR} indicate the plan of the proof of Theorem~\ref{thm:4thPMExplicitGeneral}. Precisely, the seemingly unnatural appearance of the second and the third part allows us to obtain an asymptotically correct estimate~\eqref{eq:4thPMExplicitGeneral} since we can show that all terms, except the first one, are $\ll T\log^{3}{T}$. However, if one were interested in obtaining just an upper bound like~\eqref{eq:4thPMExplicitGeneral2}, then ~\eqref{eq:4thPMmainEst} and~\eqref{eq:4thPMtermR} can be simplified to
\begin{multline}
\label{eq:4thPMmainEst2}
\mathcal{M}_{2}(T,2T) \leq 4\int_{T}^{2T}\left|J_1(t)\right|^{2}\dif{t} + \sum_{k=3}^{6}\int_{T}^{2T}\left|J_k(t)\right|^{2}\dif{t} + 4\sum_{k=3}^{6}\int_{T}^{2T}\left|J_1(t)J_{k}(t)\right|\dif{t} \\
+ 2\mathop{\sum\sum}_{3\leq n<m\leq 6}\int_{T}^{2T}\left|J_{n}(t)J_{m}(t)\right|\dif{t}
+ 2\int_{T}^{2T}\left|J_0(t)\right|\cdot\left|\zeta\left(\frac{1}{2}+\ie t\right)\right|^{2}\dif{t} + \int_{T}^{2T}\left|J_0(t)\right|^{2}\dif{t}.
\end{multline}
Before proceeding to the proof, we need first some auxiliary results for various sums that contain $d(n)$ and $d^{2}(n)$, together with some estimates for the moduli of $\Gamma(z)$, $\chi(z)$ and $\zeta(1/2+\ie t)$.

\subsection{Some results on various sums that contain \texorpdfstring{$d(n)$}{d(n)} and \texorpdfstring{$d^{2}(n)$}{d2(n)}}
\label{subsec:divisorSums}

In this section, we are going to prove several estimates on sums that involve $d(n)$ and $d^{2}(n)$, and are crucial for the estimation of terms in~\eqref{eq:4thPMmainEst} and~\eqref{eq:4thPMtermR}.

\begin{lemma}
\label{lem:d}
Let $x\geq x_0\geq 2$. Then
\[
\sum_{n\leq x}\frac{d(n)}{n^\sigma} \leq \frac{\mathcal{D}_1(x_0,\sigma_0)}{1-\sigma}x^{1-\sigma}\log{x}, \quad \sigma_0\leq\sigma\leq 1/2,\; \sigma_0\leq 0,
\]
and
\[
\sum_{n>x}\frac{d(n)}{n^\sigma} \leq \frac{\mathcal{D}_{2}(x_0,\sigma_0)}{(\sigma_0-1)^{2}}x^{1-\sigma}\log{x}, \quad \sigma\geq\sigma_0>1,
\]
where
\begin{gather}
\mathcal{D}_1(x_0,\sigma_0) \de 1 + \frac{2\gamma-1-\frac{\sigma_0}{1-\sigma_0}}{\log{x_0}} + \frac{(1-\sigma_0)\left(\frac{1}{2}+|\sigma_0|-\sigma_0\right)}{\sqrt{x_0}} + \frac{1-2\gamma\sigma_0}{\sqrt{x_0}\log{x_0}}, \label{eq:D1} \\ \mathcal{D}_{2}(x_0,\sigma_0)\de \sigma_0-1+\frac{1+2\gamma(\sigma_0-1)}{\log{x_0}}+\frac{4\sigma_0-1}{(2\sigma_0-1)\sqrt{x_0}\log{x_0}} \label{eq:D2}
\end{gather}
and $\gamma$ is the Euler--Mascheroni constant.
\end{lemma}

\begin{proof}
By~\cite[Theorem~1.1]{BBR} we have $\sum_{n\leq x}d(n)=x\log{x}+(2\gamma-1)x+\Delta(x)$ with $\left|\Delta(x)\right|\leq \sqrt{x}$ for $x\geq 1$. Let $1\leq X\leq Y$. By partial summation,
\begin{equation}
\label{eq:PSd} 
\sum_{X<n\leq Y}\frac{d(n)}{n^\sigma} = Y^{-\sigma}\sum_{n\leq Y}d(n) - X^{-\sigma}\sum_{n\leq X}d(n) + \sigma\int_{X}^{Y}\frac{\sum_{n\leq u}d(n)}{u^{1+\sigma}}\dif{u}.
\end{equation}
The statement of Lemma~\ref{lem:d} for $\sigma_0\leq\sigma\leq 1/2$, $\sigma_0\leq0$ now follows by taking $X=1$ and $Y=x$ in~\eqref{eq:PSd}, by straightforward integration, by noticing that
\[
\frac{\sigma\left(1-x^{1-\sigma}\right)}{(1-\sigma)^{2}} \leq 
\left\{
\begin{array}{ll}
    0, & \sigma\in[0,1), \\
    \frac{-\sigma x^{1-\sigma}}{(1-\sigma)^2}, & \sigma<0,
\end{array}
\right.
\]
and by the mean-value estimate
\[
\frac{\left(1+2\left(|\sigma|-\sigma\right)\right)x^{1/2-\sigma}-2|\sigma|}{1-2\sigma} \leq 1 + \left(\frac{1}{2}+|\sigma|-\sigma\right)x^{\frac{1}{2}-\sigma}\log{x}
\]
that is valid for all $\sigma<1/2$. The corresponding result for $\sigma\geq\sigma_0>1$ simply follows by taking $X=x$ and $Y\to\infty$ in~\eqref{eq:PSd}, while observing that $(\sigma-1)^{-2}\mathcal{D}_{2}(x_0,\sigma)$ is a decreasing function in $\sigma>1$.
\end{proof}

\begin{remark}
Although there exist better explicit estimates for $|\Delta(x)|$, e.g., of the form $\Delta(x) \ll x^{\frac{1}{3}} \log x$ as in~\cite{BBR}, the above weaker version serves our purpose.
\end{remark}

\begin{lemma}
\label{lem:dExp}
Let $x\geq x_0\geq 2$ and $\sigma\in[0,1]$. Then
\[
\sum_{n\leq x}\frac{d(n)}{n^\sigma}\left(1-e^{-\frac{n}{x}}\right) \leq \frac{\mathcal{D}_1(x_0,\sigma-1)}{2-\sigma}x^{1-\sigma}\log{x}, \quad \sum_{n>x}\frac{d(n)}{n^\sigma}e^{-\frac{n}{x}} \leq \mathcal{D}_{3}(x_0)x^{1-\sigma}\log{x},
\]
where $\mathcal{D}_1(x_0,\sigma_0)$ is defined by~\eqref{eq:D1} and
\begin{equation}
\label{eq:D3}
\mathcal{D}_{3}(x_0) \de \frac{1}{e} + \frac{1+2\gamma}{e\log{x_0}} + \frac{5}{2e\sqrt{x_0}\log{x_0}},
\end{equation}
where $\gamma$ is the Euler--Mascheroni constant.
\end{lemma}

\begin{proof}
The first inequality simply follows from Lemma~\ref{lem:d} due to the fact that $\left|1-e^{-u}\right|\leq u$ for $u\in[0,1]$ and $\sigma-1\in[-1,0]$. Concerning the second inequality, we have
\[
\sum_{n>x}\frac{d(n)}{n^\sigma}e^{-\frac{n}{x}} \leq \frac{1}{x^\sigma}\sum_{n>x}d(n)e^{-\frac{n}{x}}, \quad \sigma\geq 0,
\]
and
\[
\sum_{n>x}d(n)e^{-\frac{n}{x}} = -\frac{1}{e}\sum_{n\leq x}d(n) + \frac{1}{x}\int_{x}^{\infty}e^{-\frac{u}{x}}\sum_{n\leq u}d(n)\dif{u}
\]
by partial summation. Now,
\begin{flalign*}
\frac{1}{x}\int_{x}^{\infty}e^{-\frac{u}{x}}\sum_{n\leq u}d(n)\dif{u} &= \frac{1}{x}\int_{x}^{\infty}e^{-\frac{u}{x}}\left(u\log{u}+(2\gamma-1)u\right)\dif{u} + \frac{1}{x}\int_{x}^{\infty}e^{-\frac{u}{x}}\Delta(u)\dif{u} \\
&\leq \frac{2}{e}x\log{x} + \frac{4\gamma}{e}x + \frac{3}{2e}\sqrt{x},
\end{flalign*}
where we have used 
\begin{multline*}
\int_{x}^{\infty}e^{-\frac{u}{x}}u^{\alpha}(\log{u})^{k}\dif{u} = \frac{1+\alpha}{e}x^{1+\alpha}(\log{x})^{k} + \alpha(\alpha-1)x^{2}\int_{x}^{\infty}e^{-\frac{u}{x}}u^{\alpha-2}(\log{u})^{k}\dif{u} \\
+ \alpha k x^{2}\int_{x}^{\infty}e^{-\frac{u}{x}}u^{\alpha-2}(\log{u})^{k-1}\dif{u} + kx\int_{x}^{\infty}e^{-\frac{u}{x}}u^{\alpha-1}(\log{u})^{k-1}\dif{u},
\end{multline*}
which is valid for all $\alpha\in\R$ and $k\in\R$. The stated result now easily follows.
\end{proof}

\begin{lemma}
\label{lem:Sumd2}
Let $x\geq 2$. Then
\[
\sum_{n\leq x}d^{2}(n) = D_1x\log^{3}{x} + D_2x\log^{2}{x} + D_3x\log{x} + D_4x + D_0(x),
\]
where $D_1=1/\pi^2$, $0.744341<D_2<0.744342$, $0.823265<D_3<0.823266$, $0.460323<D_4<0.460324$ and $|D_0(x)|\leq 9.73x^{3/4}\log{x}$.
\end{lemma}

\begin{proof}
This is~\cite[Theorem~2]{cully2021two}.
\end{proof}

\begin{lemma}
\label{lem:d2}
Let $x\geq x_0\geq 2$. Then
\[
\sum_{n\leq x}\frac{d^{2}{(n)}}{n^{\sigma}} \leq 
\mathcal{D}_{4}(x_0,\sigma)x^{1-\sigma}\log^{3}{x}, \quad \sigma\in[-3,1),
\]
and
\[
\sum_{n>x}\frac{d^{2}{(n)}}{n^{\sigma}} \leq \frac{\mathcal{D}_{5}(x_0,\sigma)}{(\sigma-1)^{4}}x^{1-\sigma}\log^{3}{x}, \quad \sigma>1,
\]
where
\begin{equation}
\label{eq:D4}
\mathcal{D}_{4}(x_0,\sigma) \de 
\left\{
\begin{array}{ll}
   \mathcal{D}_{41}(x_0,\sigma), & \sigma\in[-3,0), \\
   \frac{1}{1-\sigma}\mathcal{D}_{42}(x_0,\sigma), & \sigma\in[0,1)
\end{array}
\right.
\end{equation}
with
\begin{flalign*}
\mathcal{D}_{41}(x_0,\sigma) &\de \frac{1}{(1-\sigma)\pi^2}+\left(-\frac{3\sigma }{(1-\sigma)^2\pi^2}+\frac{0.744342}{1-\sigma}\right)\frac{1}{\log{x_0}} \nonumber \\
&+\left(\frac{6\sigma}{(1-\sigma)^3\pi^2}-\frac{1.488684\sigma}{(1-\sigma)^2}+\frac{0.823266}{1-\sigma}\right)\frac{1}{\log^{2}{x_0}} \nonumber \\
&+\left(-\frac{6\sigma}{(1-\sigma)^{4}\pi^2}+\frac{1.488684\sigma}{(1-\sigma)^{3}}-\frac{0.823266\sigma}{(1-\sigma)^2}+\frac{0.460324}{1-\sigma}\right)\frac{1}{\log^{3}{x_0}} + \frac{17.52}{x_0^{1/4}\log^{2}{x_0}},
\end{flalign*}
and
\begin{equation*}
\mathcal{D}_{42}(x_0,\sigma)\de \frac{1}{\pi^2} + \frac{0.744342}{\log{x_0}} + \frac{0.823266}{\log^{2}{x_0}} + \frac{0.460324}{\log^{3}{x_0}} + \frac{9.73\sigma(1-\sigma)}{x_0^{\min\{1/4,1-\sigma\}}\log{x_0}} + \frac{9.73(1-\sigma)}{x_0^{1/4}\log^{2}{x_0}},
\end{equation*}
while
\begin{flalign}
\label{eq:D5}
\mathcal{D}_5(x_0,\sigma) &\de \frac{(\sigma-1)^3}{\pi^2} + \left(\sigma-1\right)^{2}\left(\frac{3\sigma}{\pi^2}+0.744342(\sigma-1)\right)\frac{1}{\log{x_0}} \nonumber \\ 
&+ \left(\sigma-1\right)\left(\frac{6\sigma}{\pi^2}+1.488684\sigma(\sigma-1)+0.823266(\sigma-1)^2\right)\frac{1}{\log^{2}{x_0}} \nonumber \\
&+ \left(\frac{6\sigma}{\pi^2}+1.488684\sigma(\sigma-1)+0.823266\sigma(\sigma-1)^{2}+0.460324(\sigma-1)^{3}\right)\frac{1}{\log^{3}{x_0}} \nonumber \\
&+ \left(1+\frac{\sigma}{\sigma-3/4}\left(1+\frac{1}{\left(\sigma-3/4\right)\log{x_0}}\right)\right)\frac{9.73(\sigma-1)^4}{x_0^{1/4}\log^{2}{x_0}}.
\end{flalign}
\end{lemma}

\begin{proof}
Let $2\leq X\leq Y$, $\sigma\in\R\setminus\{1\}$ and $k\geq 0$. Note that
\[
\int_{X}^{Y}u^{-\sigma}\left(\log{u}\right)^{k}\dif{u} = \frac{\log^{k}{Y}}{(1-\sigma)Y^{\sigma-1}} - \frac{\log^{k}{X}}{(1-\sigma)X^{\sigma-1}} - \frac{k}{1-\sigma}\int_{X}^{Y}u^{-\sigma}\left(\log{u}\right)^{k-1}\dif{u}.
\]
Repeated use of the above recursive relation, in combination with partial summation and Lemma~\ref{lem:Sumd2}, gives
\begin{flalign}
\label{eq:d2Aux1}
\sum_{X<n\leq Y}\frac{d^{2}{(n)}}{n^\sigma} &= Y^{-\sigma}\sum_{n\leq Y}d^{2}{(n)} - X^{-\sigma}\sum_{n\leq X}d^{2}{(n)} + \sigma\int_{X}^{Y}\frac{\sum_{n\leq u}d^{2}{(n)}}{u^{1+\sigma}}\dif{u} \nonumber \\
&= \sum_{k=1}^{4}D_kY^{1-\sigma}\left(\log{Y}\right)^{4-k} + \sigma\mathcal{Z}(Y,\sigma) - \sum_{k=1}^{4}D_kX^{1-\sigma}\left(\log{X}\right)^{4-k} - \sigma\mathcal{Z}(X,\sigma) \nonumber \\
&+ Y^{-\sigma}D_0(Y) - X^{-\sigma}D_0(X) + \sigma\int_{X}^{Y}\frac{D_0(u)}{u^{1+\sigma}}\dif{u},
\end{flalign}
where
\begin{flalign}
\label{eq:d2Aux2}
\mathcal{Z}(Z,\sigma) &\de \frac{D_1\log^{3}{Z}}{(1-\sigma)Z^{\sigma-1}} + \left(-\frac{3D_1}{(1-\sigma)^2}+\frac{D_2}{1-\sigma}\right)\frac{\log^{2}{Z}}{Z^{\sigma-1}} + \left(\frac{6D_1}{(1-\sigma)^3}-\frac{2D_2}{(1-\sigma)^2}+\frac{D_3}{1-\sigma}\right)\frac{\log{Z}}{Z^{\sigma-1}} \nonumber \\
&+ \left(-\frac{6D_1}{(1-\sigma)^4}+\frac{2D_2}{(1-\sigma)^3}-\frac{D_3}{(1-\sigma)^2}+\frac{D_4}{1-\sigma}\right)\frac{1}{Z^{\sigma-1}}.
\end{flalign}
Also, 
\begin{equation}
\label{eq:d2Aux3}
\mathcal{Z}(Y,\sigma)-\mathcal{Z}(X,\sigma) \leq \left(D_1\log^{3}{Y}+D_2\log^{2}{Y}+D_3\log{Y}+D_4\right)\frac{Y^{1-\sigma}-X^{1-\sigma}}{1-\sigma}.
\end{equation}
Let $\sigma\in[-3,0]$, $X=2$ and $Y=x$. Then~\eqref{eq:d2Aux1} and~\eqref{eq:d2Aux2} imply
\begin{flalign}
\label{eq:d2Aux}
\sum_{n\leq x}\frac{d^{2}{(n)}}{n^\sigma} &\leq \left(\frac{D_1}{1-\sigma}+\left(-\frac{3\sigma D_1}{(1-\sigma)^2}+\frac{D_2}{1-\sigma}\right)\frac{1}{\log{x}}+\left(\frac{6\sigma D_1}{(1-\sigma)^3}-\frac{2\sigma D_2}{(1-\sigma)^2}+\frac{D_3}{1-\sigma}\right)\frac{1}{\log^{2}{x}}\right. \nonumber \\
&\left.+\left(-\frac{6\sigma D_1}{(1-\sigma)^{4}}+\frac{2\sigma D_2}{(1-\sigma)^{3}}-\frac{\sigma D_3}{(1-\sigma)^2}+\frac{D_4}{1-\sigma}\right)\frac{1}{\log^{3}{x}}\right)x^{1-\sigma}\log^{3}{x} \nonumber \\
&+ 1 - 2^{-\sigma} - \sigma\mathcal{Z}(2,\sigma) + 9.73x^{\frac{3}{4}-\sigma}\log{x} + 9.73|\sigma|\int_{2}^{x}\frac{\log{u}}{u^{1/4+\sigma}}\dif{u}.
\end{flalign}
Note that 
\begin{equation}
\label{eq:integral}
\int_{2}^{x}\frac{\log{u}}{u^{1/4+\sigma}}\dif{u} \leq \frac{x^{3/4-\sigma}-2^{3/4-\sigma}}{3/4-\sigma}\log{x} \leq x^{\max\{3/4-\sigma,0\}}\log^{2}{x} - 2^{-1/4}\log{2}\log{x}
\end{equation}
for $\sigma\in(-\infty,1]$ by the mean-value theorem. The values of the expressions in the inner brackets of~\eqref{eq:d2Aux} are positive, while the last line in~\eqref{eq:d2Aux} is $\leq 17.52x^{3/4-\sigma}\log{x}$ by the first inequality in~\eqref{eq:integral}. The corresponding result of Lemma~\ref{lem:d2} now easily follows. Next, consider $\sigma\in[0,1)$, $X=2$ and $Y=x$. It follows from~\eqref{eq:d2Aux1},~\eqref{eq:d2Aux3} and~\eqref{eq:integral} that
\begin{multline*}
\sum_{n\leq x}\frac{d^{2}(n)}{n^\sigma} \leq \frac{x^{1-\sigma}}{1-\sigma}\sum_{k=1}^{4}D_{k}\left(\log{x}\right)^{4-k} + 1 - 2^{-\sigma} - \frac{\sigma}{1-\sigma}\sum_{k=1}^{4}D_k\left(\log{2}\right)^{4-k} \\
+ 9.73x^{\frac{3}{4}-\sigma}\log{x} + 9.73\sigma x^{\max\{3/4-\sigma,0\}}\log^{2}{x} - 9.73\sigma2^{-1/4}\log^{2}{2}.
\end{multline*}
The sum of those terms on the right-hand side of the above inequality that do not depend on $x$ is not positive. The corresponding result of Lemma~\ref{lem:d2} now easily follows. Finally, let $\sigma>1$, $X=x$ and $Y\to\infty$. The result follows from~\eqref{eq:d2Aux1} and~\eqref{eq:d2Aux2}, together with
\[
\int_{x}^{\infty}\frac{\log{u}}{u^{1/4+\sigma}}\dif{u} = \frac{1}{(\sigma-3/4)x^{\sigma-3/4}}\left(\log{x}+\frac{1}{\sigma-3/4}\right).
\]
The proof is thus complete.
\end{proof}

\begin{lemma}
\label{lem:d2Exp1}
Let $x\geq x_0\geq 2$. Then
\begin{equation*}
\left|\sum_{n\leq x}\frac{d^{2}(n)}{n}-\frac{1}{4\pi^2}\log^{4}{x}\right| \leq\mathcal{D}_{6}(x_0)\log^{3}{x},
\end{equation*}
where
\begin{equation}
\label{eq:D6}
\mathcal{D}_{6}(x_0) \de 0.349436 + \frac{1.155975}{\log{x_0}} + \frac{1.28359}{\log^{2}{x_0}} + \frac{154}{\log^{3}{x_0}} + \frac{9.73}{x_0^{1/4}\log^{2}{x_0}}.
\end{equation}
\end{lemma}

\begin{proof}
The result follows by straightforward integration, using the first equality in~\eqref{eq:d2Aux1} for $\sigma=1$, $X=2$ and $Y=x$ in combination with Lemma~\ref{lem:Sumd2}.
\end{proof}

\begin{lemma}
\label{lem:d2Exp2}
Let $x\geq x_0\geq e^{4}$ and $\sigma\geq -1$. Then
\begin{equation*}
\sum_{n>x}\frac{d^{2}(n)}{n^\sigma}e^{-\frac{2n}{x}} \leq \mathcal{D}_{7}(x_0)x^{1-\sigma}\log^{3}{x},
\end{equation*}
where
\begin{equation}
\label{eq:D7}
\mathcal{D}_{7}(x_0)\de \frac{3}{(2e\pi)^2} + \frac{0.12441}{\log{x_0}} + \frac{0.364}{\log^{2}{x_0}} + \frac{0.3}{\log^{3}{x_0}} + \frac{4.61}{x_0^{1/4}\log^{2}{x_0}}.
\end{equation}
\end{lemma}

\begin{proof}
We are going to show that 
\begin{equation}
\label{eq:d2Exp2Aux}
\sum_{n>x}nd^{2}(n)e^{-\frac{2n}{x}} \leq \mathcal{D}_{7}(x_0)x^{2}\log^{3}{x}.
\end{equation}
Then the result will follow since
\[
\sum_{n>x}\frac{d^{2}(n)}{n^\sigma}e^{-\frac{2n}{x}} \leq \frac{1}{x^{1+\sigma}}\sum_{n>x}nd^{2}(n)e^{-\frac{2n}{x}}
\]
for $\sigma\geq -1$. We have
\begin{flalign}
\label{eq:d2Exp22}
\sum_{n>x}nd^{2}(n)e^{-\frac{2n}{x}} &= -e^{-2}x\sum_{n\leq x}d^{2}(n) + \int_{x}^{\infty}\left(\frac{2u}{x}-1\right)e^{-\frac{2u}{x}}\sum_{n\leq u}d^{2}(n)\dif{u} \nonumber \\
&= -e^{-2}x^2\sum_{k=1}^{4}D_k(\log{x})^{4-k} + \sum_{k=1}^{4}D_k\int_{x}^{\infty}\left(\frac{2}{x}u^2-u\right)e^{-\frac{2u}{x}}(\log{u})^{4-k}\dif{u} \nonumber \\ 
&- e^{-2}xD_{0}(x) + \int_{x}^{\infty}\left(\frac{2u}{x}-1\right)e^{-\frac{2u}{x}}D_{0}(u)\dif{u}
\end{flalign}
by partial summation and Lemma~\ref{lem:Sumd2}. Repeated use of integration by parts gives
\begin{flalign*}
\int_{x}^{\infty}e^{-\frac{2u}{x}}&u^{\alpha}(\log{u})^{k}\dif{u} = \frac{\alpha^2+\alpha+4}{8e^2}x^{1+\alpha}(\log{x})^{k} + \frac{k(2\alpha+1)}{8e^2}x^{1+\alpha}(\log{x})^{k-1} + \frac{k(k-1)}{8e^2}x^{1+\alpha}(\log{x})^{k-2} \\ 
&+ \frac{\alpha(\alpha-1)(\alpha-2)}{8}x^3\int_{x}^{\infty}e^{-\frac{2u}{x}}\frac{(\log{u})^{k}}{u^{3-\alpha}}\dif{u} + \frac{k\left(3\alpha^2-6\alpha+2\right)}{8}x^3\int_{x}^{\infty}e^{-\frac{2u}{x}}\frac{(\log{u})^{k-1}}{u^{3-\alpha}}\dif{u} \\
&+ \frac{3(\alpha-1)k(k-1)}{8}x^3\int_{x}^{\infty}e^{-\frac{2u}{x}}\frac{(\log{u})^{k-2}}{u^{3-\alpha}}\dif{u} + \frac{k(k-1)(k-2)}{8}x^3\int_{x}^{\infty}e^{-\frac{2u}{x}}\frac{(\log{u})^{k-3}}{u^{3-\alpha}}\dif{u}
\end{flalign*}
for any $\alpha\in\R$ and $k\in\R$. Then
\begin{flalign}
\label{eq:d2Exp23}
- e^{-2}xD_{0}(x) + \int_{x}^{\infty}\left(\frac{2u}{x}-1\right)e^{-\frac{2u}{x}}D_{0}(u)\dif{u} &\leq \frac{9.73}{e^2}x^{\frac{7}{4}}\log{x} + \frac{19.46\log{x}}{x^{5/4}}\int_{x}^{\infty}e^{-\frac{2u}{x}}u^{2}\dif{u} \nonumber \\ 
&\leq 4.61x^{\frac{7}{4}}\log{x}.
\end{flalign}
Using
\[
\int_{x}^{\infty}e^{-\frac{2u}{x}}\frac{(\log{u})^{k}}{u^{\alpha}}\dif{u} \leq \frac{1}{2e^2}x^{1-\alpha}(\log{x})^{k}, \quad k\leq 2, \alpha\geq 1,
\]
we also obtain
\begin{gather*}
\int_{x}^{\infty}\left(\frac{2}{x}u^2-u\right)e^{-\frac{2u}{x}}(\log{u})^{3}\dif{u} \leq \frac{7}{4e^2}x^2(\log{x})^{3} + \frac{57}{16e^2}x^2(\log{x})^{2} + \frac{3}{e^2}x^2\log{x} + \frac{3}{4e^2}x^2, \\
\int_{x}^{\infty}\left(\frac{2}{x}u^2-u\right)e^{-\frac{2u}{x}}(\log{u})^{2}\dif{u} \leq \frac{7}{4e^2}x^2(\log{x})^2 + \frac{19}{8e^2}x^{2}\log{x} + \frac{1}{e^2}x^2, \\
\int_{x}^{\infty}\left(\frac{2}{x}u^2-u\right)e^{-\frac{2u}{x}}\log{u}\dif{u} \leq \frac{7}{4e^2}x^2\log{x} + \frac{19}{16e^2}x^2, \\
\int_{x}^{\infty}\left(\frac{2}{x}u^2-u\right)e^{-\frac{2u}{x}}\dif{u} \leq \frac{7}{4e^2}x^2.
\end{gather*}
Therefore,
\begin{multline}
\label{eq:d2Exp24}
-e^{-2}x^2\sum_{k=1}^{4}D_k(\log{x})^{4-k} + \sum_{k=1}^{4}D_k\int_{x}^{\infty}\left(\frac{2}{x}u^2-u\right)e^{-\frac{2u}{x}}(\log{u})^{4-k}\dif{u} \leq \frac{3D_1}{4e^2}x^2(\log{x})^{3} \\ 
+ \left(\frac{57D_1}{16e^2}+\frac{3D_2}{4e^2}\right)x^2(\log{x})^2 + \left(\frac{3D_1}{e^2}+\frac{19D_2}{8e^2}+\frac{3D_3}{4e^2}\right)x^{2}\log{x} + \left(\frac{3D_1}{4e^2}+\frac{D_2}{e^2}+\frac{19D_3}{16e^2}+\frac{3D_4}{4e^2}\right)x^2.
\end{multline}
It is straightforward to verify that~\eqref{eq:d2Exp2Aux} is valid by considering~\eqref{eq:d2Exp23} and~\eqref{eq:d2Exp24} in~\eqref{eq:d2Exp22}.
\end{proof}

\begin{lemma}
\label{lem:d2Exp3}
Let $x\geq x_0\geq 2$. Then
\[
\sum_{n\leq x}\frac{d^{2}(n)}{n^\sigma}\left(e^{-\frac{n}{x}}-1\right)^{2} \leq \mathcal{D}_{4}(x_0,\sigma-2)x^{1-\sigma}\log^{3}{x}
\]
for $\sigma\in[-1,2]$, where $\mathcal{D}_{4}(x_0,\sigma)$ is defined by~\eqref{eq:D4}.
\end{lemma}

\begin{proof}
We have 
\[
\sum_{n\leq x}\frac{d^{2}(n)}{n^\sigma}\left(e^{-\frac{n}{x}}-1\right)^{2} \leq \frac{1}{x^2}\sum_{n\leq x}d^{2}(n)n^{2-\sigma}
\]
since $\left|e^{-u}-1\right|\leq u$ for $u\in[0,1]$. The result now follows by Lemma~\ref{lem:d2}.
\end{proof}

\begin{remark}
The main terms in the estimates from Lemmas~\ref{lem:dExp},~\ref{lem:d2Exp2} and~\ref{lem:d2Exp3} can be improved for certain $\sigma$'s, e.g., essentially the same method as in the proof of Lemma~\ref{lem:d2Exp2} gives
\[
\sum_{n>x}d^{2}(n)e^{-\frac{2n}{x}} = \frac{1}{2(e\pi)^2}x\log^{3}{x} + O\left(x\log^{2}{x}\right).
\]
However, there are some technical difficulties to obtain asymptotically correct estimates when $\sigma$ is not an integer. Note that the above also shows that the first estimate of~\cite[p.~138]{Ivic} is not correct, although this does not change the final result.
\end{remark}

\subsection{Proof of Theorem~\ref{thm:4thPMExplicitGeneral}}
\label{subsec:ProofThm2}

In this section, we provide the proof of Theorem~\ref{thm:4thPMExplicitGeneral} by estimating each of the six parts in~\eqref{eq:4thPMmainEst} and~\eqref{eq:4thPMtermR}. In order to simplify some expressions throughout the following sections, we define
\[
\eta(T_0) \de 1 + \frac{\pi m_0}{T_0},
\]
where $m_0$ is from Lemma~\ref{lem:Mxone_half}. In addition, we assume the notation and conditions~(1)--(6) from Theorem~\ref{thm:4thPMExplicitGeneral}. Before proceeding to the proof, we need two additional results that are related to the Gamma function.

\begin{lemma}
\label{lem:chi}
Let $z=x+\ie y$ with $x\in[-1,1]$ and $|y|\geq y_0\geq 2$. Then
\begin{equation}
\label{eq:Gamma} 
\exp{\left(-\frac{1}{2y_0^2}\right)} \leq \frac{\left|\Gamma(z)\right|}{\sqrt{2\pi}|z|^{x-\frac{1}{2}}e^{-\frac{\pi}{2}|y|}} \leq \exp{\left(\frac{1}{2y_0^2}\right)}
\end{equation}
and
\begin{equation}
\label{eq:chiBound}
\mathcal{U}^{-}(y_0) \leq \left(\frac{|z|}{2\pi}\right)^{x-\frac{1}{2}}\left|\chi(z)\right| \leq \mathcal{U}^{+}(y_0), \quad \mathcal{U}^{\pm}(y_0)\de \left(1\mp e^{-\pi y_0}\right)^{-1}\exp{\left(\pm\frac{1}{2y_0^2}\right)}.
\end{equation}
In addition, $|\chi(z)|\leq 0.4|y|^{1/2-x}+16$ for $x\in[-1/2,0]$ and $y\in\R$, while $|\chi(z)|\leq 6.31/(1-x)$ for $x\in[1/2,1)$ and $y\in\R$.
\end{lemma}

\begin{proof}
Inequalities~\eqref{eq:Gamma} follow from~\cite[Theorem~2.1 (B2)]{DonaZunigaAlterman} with $\theta=\arctan{2}$. Inequalities~\eqref{eq:chiBound} follow from~\eqref{eq:chiDef},~\eqref{eq:Gamma} and
\[
1-e^{-\pi y_0} \leq 2e^{-\frac{\pi}{2}|y|}\left|\cos{\left(\frac{\pi z}{2}\right)}\right| \leq 1+e^{-\pi y_0},
\]
with the last inequalities true by~\cite[Proposition~2.2]{DonaZunigaAlterman}. For the second part, one can use~\eqref{eq:chiBound} to obtain $|\chi(z)|\leq 0.4|y|^{1/2-x}$ for $x\in[-1/2,0]$ and $|y|\geq 10^2$, while the maximum principle and numerical verification show that $|\chi(z)|\leq 16$ holds for $x\in[-1/2,0]$ and $|y|\leq 10^2$. Finally, $|\chi(z)|\leq 1.1$ for $x\in[1/2,1]$ and $|y|\geq 2\pi$, while $|(z-1)\chi(z)|\leq 6.31$ for $x\in[1/2,1]$ and $|y|\leq 2\pi$ by the maximum modulus principle and numerical verification.
\end{proof}

\begin{lemma}
\label{lem:res}
Let $t\geq T\geq 5$. Then
\[
\left|\Res_{z=\frac{1}{2}-\ie t}\left(\zeta^{2}{\left(\frac{1}{2}+\ie t+z\right)}\Gamma(z)T^{z}\right)\right| \leq 12\sqrt{T}e^{-\frac{\pi}{2}T}\log{T}.
\]
\end{lemma}

\begin{proof}
The Laurent expansion of $\zeta(s)$ at $s=1$ gives $(s-1)^{2}\zeta^{2}(s)=1+2\gamma(s-1)+(s-1)^{2}h(s)$ for some entire function $h(s)$, where $\gamma$ is the Euler--Mascheroni constant. Note that the preceding equality is valid for all $s\in\C$. Therefore,
\begin{flalign*}
\Res_{z=\frac{1}{2}-\ie t}\left(\zeta^{2}{\left(\frac{1}{2}+\ie t+z\right)}\Gamma(z)T^{z}\right) &= \lim_{z\to \frac{1}{2}-\ie t} \frac{\dif{}}{\dif{z}}\left(\left(z-\frac{1}{2}+\ie t\right)^{2}\zeta^{2}{\left(\frac{1}{2}+\ie t+z\right)}\Gamma(z)T^{z}\right) \\
&= \left(2\gamma + \frac{\Gamma'}{\Gamma}\left(\frac{1}{2}-\ie t\right)+\log{T}\right)T^{\frac{1}{2}-\ie t}\Gamma\left(\frac{1}{2}-\ie t\right).
\end{flalign*}
By~\cite[Exercise~4.2 on p.~295]{Olver} we have 
\[
\left|\frac{\Gamma'}{\Gamma}\left(\frac{1}{2}-\ie t\right)\right|\leq \log{t} + \frac{\pi}{2} + \frac{1}{2}\log{\left(1+\frac{1}{4T^2}\right)}+\frac{1}{2T}+\frac{1}{3\sqrt{2}T^2} \leq \log{t} + 1.7.
\]
By~\eqref{eq:Gamma} we also have $\left|\Gamma\left(1/2-\ie t\right)\right| \leq 3e^{-\frac{\pi}{2}t}$. The stated inequality now follows since $2\gamma+1.7+\log{T}\leq 3\log{t}$ and $e^{-\pi t/2}\log{t}$ is a decreasing function for $t\geq 5$.
\end{proof}

\subsubsection{I. part}

By Lemmas~\ref{lem:MV},~\ref{lem:d2} and~\ref{lem:d2Exp1} we have
\begin{equation}
\label{eq:J1}
\left|2\int_{T}^{2T}\left|J_1(t)\right|^{2}\dif{t} - \frac{1}{2\pi^2}T\log^{4}{T}\right| \leq \mathcal{J}_{1}(T_0)T\log^{3}{T},
\end{equation}
where
\begin{equation}
\label{eq:cJ1}
\mathcal{J}_{1}(T_0)\de 2\eta(T_0)\mathcal{D}_{6}(T_0)+4\pi m_0\mathcal{D}_{4}(T_0,0)+\frac{m_0\log{T_0}}{2\pi T_0},
\end{equation}
and $\mathcal{D}_{6}(x_0)$ and $\mathcal{D}_{4}(x_0,\sigma)$ are defined by~\eqref{eq:D6} and~\eqref{eq:D4}, respectively. This shows that the main contribution in~\eqref{eq:4thPMExplicitGeneral} comes from this part. In the subsequent sections, we show that all other terms are $\ll T\log^{3}{T}$.

\subsubsection{II. part}
\label{sec:IIpart}

Let $f(z)\de \chi^{-1}(z)\sum_{n\leq T}n^{-z}d(n)$. It is clear that this is a holomorphic function in $\{z\in\C\colon |\Im\{z\}|>1\}$. By Cauchy's theorem,
\begin{flalign}
\ie\int_{T}^{2T}\left(J_1^2(t)+J_2^2(t)\right)\dif{t} &= \left(\int_{\frac{1}{2}+\ie T}^{\frac{1}{2}+2\ie T}+\int_{\frac{1}{2}-2\ie T}^{\frac{1}{2}-\ie T}\right)f^{2}(z)\dif{z} \nonumber \\
&= \left(\int_{\sigma_1+\ie T}^{\sigma_1+2\ie T} + \int_{\sigma_1-2\ie T}^{\sigma_1-\ie T}\right)f^{2}(z)\dif{z} \nonumber \\
&+ \left(\int_{\frac{1}{2}+\ie T}^{\sigma_1+\ie T} + \int_{\frac{1}{2}-2\ie T}^{\sigma_1-2\ie T}\right)f^{2}(z)\dif{z} + \left(\int_{\sigma_1+2\ie T}^{\frac{1}{2}+2\ie T}+\int_{\sigma_1-\ie T}^{\frac{1}{2}-\ie T}\right)f^{2}(z)\dif{z}. \label{eq:secterm}
\end{flalign}
We have
\begin{flalign*}
\left|\left(\int_{\frac{1}{2}+\ie T}^{\sigma_1+\ie T} + \int_{\frac{1}{2}-2\ie T}^{\sigma_1-2\ie T}\right)f^{2}(z)\dif{z}\right| &\leq \int_{\sigma_1}^{\frac{1}{2}}\left(\left|\chi(u+\ie T)\right|^{-2}+\left|\chi(u+2\ie T)\right|^{-2}\right)\left(\sum_{n\leq T}n^{-u}d(n)\right)^{2}\dif{u} \\
&\leq \left(\frac{\mathcal{D}_1(T_0,0)}{\mathcal{U}^{-}(T_0)}\right)^{2}\left(\int_{\sigma_1}^{\frac{1}{2}}\frac{(2\pi)^{1-2u}+\pi^{1-2u}}{(1-u)^2}\dif{u}\right)T\log^{2}{T}
\end{flalign*}
by Lemmas~\ref{lem:d} and~\ref{lem:chi}, with the same upper bound holding also for the modulus of the last pair of integrals in~\eqref{eq:secterm}. Next,
\begin{flalign*}
\left|\left(\int_{\sigma_1+\ie T}^{\sigma_1+2\ie T} + \int_{\sigma_1-2\ie T}^{\sigma_1-\ie T}\right)f^{2}(z)\dif{z}\right| &\leq 2\int_{T}^{2T}\left|\chi(\sigma_1+\ie u)\right|^{-2}\left|\sum_{n\leq T}\frac{d(n)}{n^{\sigma_1+\ie u}}\right|^{2}\dif{u} \\ 
&\leq \frac{2}{\left(\mathcal{U}^{-}(T_0)\right)^{2}}\left(\frac{T}{2\pi}\right)^{2\sigma_{1}-1}\int_{T}^{2T}\left|\sum_{n\leq T}\frac{d(n)}{n^{\sigma_1+\ie u}}\right|^{2}\dif{u} \\
&\leq \frac{2(2\pi)^{1-2\sigma_1}}{\left(\mathcal{U}^{-}(T_0)\right)^2}\left(\eta(T_0)\mathcal{D}_{4}(T_0,2\sigma_1)+2\pi m_0\mathcal{D}_{4}(T_0,2\sigma_1-1)\right)T\log^{3}{T}
\end{flalign*}
by Lemmas~\ref{lem:chi},~\ref{lem:MV} and~\ref{lem:d2}. Therefore,
\begin{equation}
\label{eq:J1J2square}
\int_{T}^{2T}\left(J_1^2(t)+J_2^2(t)\right)\dif{t} \leq \mathcal{J}_{2,1}(T_0,\sigma_1)T\log^{3}{T},
\end{equation}
where
\begin{multline}
\label{eq:cJ21}
\mathcal{J}_{2,1}(T_0,\sigma_1) \de \frac{2(2\pi)^{1-2\sigma_1}}{\left(\mathcal{U}^{-}(T_0)\right)^2}\left(\eta(T_0)\mathcal{D}_{4}(T_0,2\sigma_1)+2\pi m_0\mathcal{D}_{4}(T_0,2\sigma_1-1)\right) \\ 
+ \frac{2}{\log{T_0}}\left(\frac{\mathcal{D}_1(T_0,0)}{\mathcal{U}^{-}(T_0)}\right)^{2}\int_{\sigma_1}^{\frac{1}{2}}\frac{(2\pi)^{1-2u}+\pi^{1-2u}}{(1-u)^2}\dif{u},
\end{multline}
and $\mathcal{D}_{4}(x_0,\sigma)$, $\mathcal{D}_{1}(x_0,\sigma_0)$ and $\mathcal{U}^{-}(y_0)$ are defined by~\eqref{eq:D4},~\eqref{eq:D1} and~\eqref{eq:chiBound}, respectively.

\subsubsection{III. part}

This part is the most challenging to estimate, and we need to divide it into several new terms. First, we estimate it for $k\in\{3,4\}$. Let 
\[
g_1(z)\de \chi(z)^{-1}\sum_{n>T}\frac{d(n)}{n^{z}}e^{-n/T}, \quad g_2(z)\de \chi(z)^{-1}\sum_{n\leq T}\frac{d(n)}{n^{z}}\left(e^{-n/T}-1\right). 
\]
It is clear that $g_1(z)$ and $g_2(z)$ are holomorphic functions in $\{z\in\C\colon |\Im\{z\}|>1\}$. Define $F(z)\de F_{1}(z)+F_2(z)$, where $F_1(z)\de f(1-z)\left(g_1(z)+g_2(z)\right)$ and $F_2(z)\de f(z)\left(g_1(z)+g_2(z)\right)$ with $f(z)$ defined as in Section~\ref{sec:IIpart}. By Cauchy's theorem,
\begin{multline}
\label{eq:thirdterm1}
\ie\int_{T}^{2T}\left(J_1(t)+J_2(t)\right)\left(J_3(t)+J_4(t)\right)\dif{t} = \int_{\frac{1}{2}+\ie T}^{\frac{1}{2}+2\ie T}F(z)\dif{z} = \int_{\sigma_2+\ie T}^{\sigma_2+2\ie T}F_1(z)\dif{z} + \int_{\sigma_3+\ie T}^{\sigma_3+2\ie T}F_2(z)\dif{z} \\
+ \left(\int_{\frac{1}{2}+\ie T}^{\sigma_2+\ie T}+\int_{\sigma_2+2\ie T}^{\frac{1}{2}+2\ie T}\right)F_1(z)\dif{z} + \left(\int_{\frac{1}{2}+\ie T}^{\sigma_3+\ie T}+\int_{\sigma_3+2\ie T}^{\frac{1}{2}+2\ie T}\right)F_{2}(z)\dif{z}.
\end{multline}
We need to estimate the moduli of each of the above terms. Starting with the first two terms on the right-hand side of~\eqref{eq:thirdterm1}, we have
\begin{multline*}
\left|\int_{\sigma_2+\ie T}^{\sigma_2+2\ie T}F_1(z)\dif{z}\right| \leq \frac{1}{\left(\mathcal{U}^{-}(T_0)\right)^{2}}\left(2+\frac{1}{T_0}\right)^{\sigma_2-\frac{1}{2}}\left(\int_{T}^{2T}\left|\sum_{n\leq T}\frac{d(n)}{n^{1-\sigma_2+\ie u}}\right|^{2}\dif{u}\right)^{\frac{1}{2}} \times \\
\times\left(\left(\int_{T}^{2T}\left|\sum_{n>T}\frac{d(n)}{n^{\sigma_2+\ie u}}e^{-n/T}\right|^{2}\dif{u}\right)^{\frac{1}{2}}+\left(\int_{T}^{2T}\left|\sum_{n\leq T}\frac{d(n)}{n^{\sigma_2+\ie u}}\left(e^{-n/T}-1\right)\right|^{2}\dif{u}\right)^{\frac{1}{2}}\right)
\end{multline*}
and
\begin{multline*}
\left|\int_{\sigma_3+\ie T}^{\sigma_3+2\ie T}F_2(z)\dif{z}\right| \leq \frac{1}{\left(\mathcal{U}^{-}(T_0)\right)^{2}}\left(\frac{2\pi}{T}\right)^{1-2\sigma_3}\left(\int_{T}^{2T}\left|\sum_{n\leq T}\frac{d(n)}{n^{\sigma_3+\ie u}}\right|^{2}\dif{u}\right)^{\frac{1}{2}} \times \\
\times\left(\left(\int_{T}^{2T}\left|\sum_{n>T}\frac{d(n)}{n^{\sigma_3+\ie u}}e^{-n/T}\right|^{2}\dif{u}\right)^{\frac{1}{2}}+\left(\int_{T}^{2T}\left|\sum_{n\leq T}\frac{d(n)}{n^{\sigma_3+\ie u}}\left(e^{-n/T}-1\right)\right|^{2}\dif{u}\right)^{\frac{1}{2}}\right)
\end{multline*}
by the CBS inequality and Lemma~\ref{lem:chi}. Then
\begin{equation}
\label{eq:F1F2main}
\left|\int_{1+\ie T}^{1+2\ie T}F_1(z)\dif{z} + \int_{\ie T}^{2\ie T}F_2(z)\dif{z}\right| \leq \mathfrak{G}_{1}\left(T_0,\bm{\sigma'}\right)T\log^{3}{T}
\end{equation}
by Lemmas~\ref{lem:MV},~\ref{lem:d2},~\ref{lem:d2Exp2} and~\ref{lem:d2Exp3}, where
\begin{flalign}
\label{eq:fG1}
\mathfrak{G}_{1}\left(T_0,\bm{\sigma'}\right) &\de \frac{1}{\left(\mathcal{U}^{-}(T_0)\right)^{2}}\left(2+\frac{1}{T_0}\right)^{\sigma_2-\frac{1}{2}}\left(\eta(T_0)\mathcal{D}_{4}(T_0,2-2\sigma_2)+2\pi m_0\mathcal{D}_{4}(T_0,1-2\sigma_2)\right)^{\frac{1}{2}}\times \nonumber \\
&\times\left(\left(\eta(T_0)+2\pi m_0\right)^{\frac{1}{2}}\mathcal{D}_{7}(T_0)^{\frac{1}{2}}+\left(\eta(T_0)\mathcal{D}_{4}(T_0,2\sigma_2-2)+2\pi m_0\mathcal{D}_{4}(T_0,2\sigma_2-3)
\right)^{\frac{1}{2}}\right)\nonumber \\
&+ \frac{(2\pi)^{1-2\sigma_3}}{\left(\mathcal{U}^{-}(T_0)\right)^{2}}\left(\eta(T_0)\mathcal{D}_{4}(T_0,2\sigma_3)+2\pi m_0\mathcal{D}_{4}(T_0,2\sigma_3-1)\right)^{\frac{1}{2}} \times \nonumber \\
&\times\left(\left(\eta(T_0)+2\pi m_0\right)^{\frac{1}{2}}\mathcal{D}_{7}(T_0)^{\frac{1}{2}}+\left(\eta(T_0)\mathcal{D}_{4}(T_0,2\sigma_3-2)+2\pi m_0\mathcal{D}_{4}(T_0,2\sigma_3-3)
\right)^{\frac{1}{2}}\right).
\end{flalign}
Here, $\mathcal{D}_{4}(x_0,\sigma)$, $\mathcal{D}_{7}(x_0)$ and $\mathcal{U}^{-}(y_0)$ are defined by~\eqref{eq:D4},~\eqref{eq:D7} and~\eqref{eq:chiBound}, respectively. Turning to the last two terms in~\eqref{eq:thirdterm1}, we have
\begin{multline*}
\left|\left(\int_{\frac{1}{2}+\ie T}^{\sigma_2+\ie T}+\int_{\sigma_2+2\ie T}^{\frac{1}{2}+2\ie T}\right)F_1(z)\dif{z}\right| \leq \frac{1}{\left(\mathcal{U}^{-}(T_0)\right)^{2}}\int_{\frac{1}{2}}^{\sigma_2}\left(\left(1+\frac{1}{T_0}\right)^{u-\frac{1}{2}}+\left(1+\frac{1}{2T_0}\right)^{u-\frac{1}{2}}\right)\times \\
\times\left(\sum_{n>T}\frac{d(n)}{n^u}e^{-n/T}+\sum_{n\leq T}\frac{d(n)}{n^u}\left(1-e^{-n/T}\right)\right)\sum_{n\leq T}\frac{d(n)}{n^{1-u}}\dif{u}
\end{multline*}
and
\begin{multline*}
\left|\left(\int_{\frac{1}{2}+\ie T}^{\sigma_3+\ie T}+\int_{\sigma_3+2\ie T}^{\frac{1}{2}+2\ie T}\right)F_{2}(z)\dif{z}\right| \leq \frac{1}{\left(\mathcal{U}^{-}(T_0)\right)^{2}}\int_{\sigma_3}^{\frac{1}{2}}\left((2\pi)^{1-2u}+\pi^{1-2u}\right)T^{2u-1}\times \\
\times\left(\sum_{n>T}\frac{d(n)}{n^u}e^{-n/T}+\sum_{n\leq T}\frac{d(n)}{n^u}\left(1-e^{-n/T}\right)\right)\sum_{n\leq T}\frac{d(n)}{n^{u}}\dif{u}
\end{multline*}
by Lemma~\ref{lem:chi}. Then
\begin{equation}
\label{eq:F1F2res}
\left|\left(\int_{\frac{1}{2}+\ie T}^{\sigma_2+\ie T}+\int_{\sigma_2+2\ie T}^{\frac{1}{2}+2\ie T}\right)F_1(z)\dif{z} + \left(\int_{\frac{1}{2}+\ie T}^{\sigma_3+\ie T}+\int_{\sigma_3+2\ie T}^{\frac{1}{2}+2\ie T}\right)F_{2}(z)\dif{z}\right| \leq \mathfrak{G}_2\left(T_0,\bm{\sigma'}\right)T\log^{2}{T}
\end{equation}
by Lemmas~\ref{lem:d} and~\ref{lem:dExp}, where
\begin{multline}
\label{eq:fG2}
\mathfrak{G}_2\left(T_0,\bm{\sigma'}\right) \de \frac{\mathcal{D}_1(T_0,0)}{\left(\mathcal{U}^{-}(T_0)\right)^{2}}\int_{\frac{1}{2}}^{\sigma_2}\left(\left(1+\frac{1}{T_0}\right)^{u-\frac{1}{2}}+\left(1+\frac{1}{2T_0}\right)^{u-\frac{1}{2}}\right)\left(\frac{\mathcal{D}_1(T_0,u-1)}{2-u}+\mathcal{D}_{3}(T_0)\right)\frac{\dif{u}}{u} \\
+ \frac{\mathcal{D}_1(T_0,0)}{\left(\mathcal{U}^{-}(T_0)\right)^{2}}\int_{\sigma_3}^{\frac{1}{2}}\left((2\pi)^{1-2u}+\pi^{1-2u}\right)\left(\frac{\mathcal{D}_1(T_0,u-1)}{2-u}+\mathcal{D}_{3}(T_0)\right)\frac{\dif{u}}{1-u}.
\end{multline}
Here, $\mathcal{D}_{1}(x_0,\sigma_0)$, $\mathcal{D}_{3}(x_0)$ and $\mathcal{U}^{-}(y_0)$ are defined by~\eqref{eq:D1},~\eqref{eq:D3} and~\eqref{eq:chiBound}, respectively.

It remains to estimate the third part also for $k\in\{5,6\}$, and the procedure is very similar to the one in the preceding section. Define
\begin{gather*}
g_3(z)\de -\frac{1}{2\pi}\chi^{-1}(z)\int_{-\mathfrak{a}_1(\log{T})^{\mathfrak{a}_2}}^{\mathfrak{a}_1(\log{T})^{\mathfrak{a}_2}}\chi^{2}(z+c_1+\ie u)\Gamma(c_1+\ie u)T^{c_1+\ie u}\sum_{n>T}\frac{d(n)}{n^{1-z-c_1-\ie u}}\dif{u}, \\
g_4(z)\de -\frac{1}{2\pi}\chi^{-1}(z)\int_{-\mathfrak{b}_1(\log{T})^{\mathfrak{b}_2}}^{\mathfrak{b}_1(\log{T})^{\mathfrak{b}_2}}\chi^{2}(z+c_2+\ie u)\Gamma(c_2+\ie u)T^{c_2+\ie u}\sum_{n\leq T}\frac{d(n)}{n^{1-z-c_2-\ie u}}\dif{u}.
\end{gather*}
It is clear that $g_3(z)$ and $g_4(z)$ are holomorphic functions in $\{z\in\C\colon 1/2-c_2<\Re\{z\}<|c_1|,\Im\{z\}>1+\mathfrak{b_1}(\log{T})^{\mathfrak{b}_2}\}$. Define $G(z)\de G_1(z)+G_2(z)$, where $G_1(z)\de f(1-z)\left(g_3(z)+g_4(z)\right)$ and $G_2(z)\de f(z)\left(g_3(z)+g_4(z)\right)$ with $f(z)$ defined as in Section~\ref{sec:IIpart}. By Cauchy's theorem,
\begin{multline}
\label{eq:thirdterm2}
\ie\int_{T}^{2T}\left(J_1(t)+J_2(t)\right)\left(J_5(t)+J_6(t)\right)\dif{t} = \int_{\frac{1}{2}+\ie T}^{\frac{1}{2}+2\ie T}G(z)\dif{z} = \int_{\sigma_4+\ie T}^{\sigma_4+2\ie T}G_1(z)\dif{z} + \int_{\sigma_5+\ie T}^{\sigma_5+2\ie T}G_2(z)\dif{z} \\
+ \left(\int_{\frac{1}{2}+\ie T}^{\sigma_4+\ie T}+\int_{\sigma_4+2\ie T}^{\frac{1}{2}+2\ie T}\right)G_1(z)\dif{z} + \left(\int_{\frac{1}{2}+\ie T}^{\sigma_5+\ie T}+\int_{\sigma_5+2\ie T}^{\frac{1}{2}+2\ie T}\right)G_{2}(z)\dif{z}.
\end{multline}
We need to estimate the moduli of each of the above terms. Starting with the first term on the right-hand side of~\eqref{eq:thirdterm2}, note that $0<1-\sigma_4<1/2$. By the CBS inequality and Lemma~\ref{lem:chi} we then have
\begin{multline}
\label{eq:G1First}
\left|\int_{\sigma_4+\ie T}^{\sigma_4+2\ie T}G_1(z)\dif{z}\right| \leq \frac{1}{\mathcal{U}^{-}(T_0)}\left(\frac{2\pi}{T}\right)^{\sigma_4-\frac{1}{2}}\left(\int_{T}^{2T}\left|\sum_{n\leq T}\frac{d(n)}{n^{1-\sigma_4-\ie t}}\right|^{2}\dif{t}\right)^{\frac{1}{2}} \times \\
\times\left(\left(\int_{T}^{2T}\left|g_3(\sigma_4+\ie t)\right|^{2}\dif{t}\right)^{\frac{1}{2}}+\left(\int_{T}^{2T}\left|g_4(\sigma_4+\ie t)\right|^{2}\dif{t}\right)^{\frac{1}{2}}\right).
\end{multline}
We estimate the second integral on the right-hand side of~\eqref{eq:G1First}. Note that $-1/2<c_1+\sigma_4<0$, implying $|c_1+\sigma_4+\ie(t+u)|\leq 1/2+2T+\mathfrak{a}_1(\log{T})^{\mathfrak{a}_2}$ since $t\in[T,2T]$ and $|u|\leq \mathfrak{a}_1\left(\log{T}\right)^{\mathfrak{a}_2}$. By the CBS inequality and Lemma~\ref{lem:chi} we thus have
\begin{flalign*}
\int_{T}^{2T}\left|g_3(\sigma_4+\ie t)\right|^{2}\dif{t} &\leq \frac{T^{1-2\sigma_4+2|c_1|}}{\left(2\pi\cdot\mathcal{U}^{-}(T_0)\right)^{2}}\left(\frac{1}{\pi}+\frac{1}{2\pi T_0}\right)^{2\sigma_4-1}\left(\frac{1}{\pi}+\frac{\frac{1}{2}+\mathfrak{a}_1(\log{T_0})^{\mathfrak{a}_2}}{2\pi T_0}\right)^{2\left(1-2\sigma_4\right)+4|c_1|}\times \\
&\times \left(\mathcal{U}^{+}\left(T_0-\mathfrak{a}_1(\log{T_0})^{\mathfrak{a}_2}\right)\right)^{4}\int_{-\infty}^{\infty}\left|\Gamma(c_1+\ie u)\right|\dif{u}\times \\
&\times\int_{-\mathfrak{a}_1\left(\log{T}\right)^{\mathfrak{a}_2}}^{\mathfrak{a}_1\left(\log{T}\right)^{\mathfrak{a}_2}}\left|\Gamma(c_1+\ie u)\right|\left(\int_{T}^{2T}\left|\sum_{n>T}\frac{d(n)}{n^{1-\sigma_4-c_1-\ie(t+u)}}\right|^{2}\dif{t}\right)\dif{u}.
\end{flalign*}
Moreover,
\begin{multline*}
\int_{-\mathfrak{a}_1\left(\log{T}\right)^{\mathfrak{a}_2}}^{\mathfrak{a}_1\left(\log{T}\right)^{\mathfrak{a}_2}}\left|\Gamma(c_1+\ie u)\right|\left(\int_{T}^{2T}\left|\sum_{n>T}\frac{d(n)}{n^{1-\sigma_4-c_1-\ie(t+u)}}\right|^{2}\dif{t}\right)\dif{u}
\leq \int_{-\infty}^{\infty}\left|\Gamma(c_1+\ie u)\right|\dif{u} \times \\
\times\left(\frac{\eta(T_0)\mathcal{D}_{5}\left(T_0,2-2\sigma_4+2|c_1|\right)}{(1-2\sigma_4+2|c_1|)^4}+\frac{2\pi m_0\mathcal{D}_{5}\left(T_0,1-2\sigma_4+2|c_1|\right)}{\left(2|c_1|-2\sigma_4\right)^4}\right)T^{2\sigma_4-2|c_1|}\log^{3}{T}
\end{multline*}
by Lemmas~\ref{lem:MV} and~\ref{lem:d2}. Therefore,
\begin{equation}
\label{eq:intg3c1}
\int_{T}^{2T}\left|g_3(\sigma_4+\ie t)\right|^{2}\dif{t} \leq \mathfrak{g}_1\left(T_0,\sigma_4,c_1,\bm{\mathfrak{a}}\right)T\log^{3}{T},
\end{equation}
where
\begin{flalign}
\label{eq:fg1}
\mathfrak{g}_1\left(T_0,\sigma_4,c_1,\bm{\mathfrak{a}}\right) &\de \frac{1}{\left(\mathcal{U}^{-}(T_0)\right)^{2}}\left(\frac{1}{\pi}+\frac{1}{2\pi T_0}\right)^{2\sigma_4-1}\left(\frac{1}{\pi}+\frac{\frac{1}{2}+\mathfrak{a}_1(\log{T_0})^{\mathfrak{a}_2}}{2\pi T_0}\right)^{2\left(1-2\sigma_4\right)+4|c_1|}\times \nonumber \\
&\times \left(\mathcal{U}^{+}\left(T_0-\mathfrak{a}_1(\log{T_0})^{\mathfrak{a}_2}\right)\right)^{4}\left(\int_{-\infty}^{\infty}\left|\Gamma(c_1+\ie u)\right|\dif{u}\right)^{2}\times \nonumber \\
&\times\left(\frac{\eta(T_0)\mathcal{D}_{5}\left(T_0,2-2\sigma_4+2|c_1|\right)}{(2\pi)^2(1-2\sigma_4+2|c_1|)^4}+\frac{m_0\mathcal{D}_{5}\left(T_0,1-2\sigma_4+2|c_1|\right)}{2\pi\left(2|c_1|-2\sigma_4\right)^4}\right).
\end{flalign}
Here, $\mathcal{D}_{5}(x_0,\sigma)$ and $\mathcal{U}^{\pm}(y_0)$ are defined by~\eqref{eq:D5} and~\eqref{eq:chiBound}, respectively. Turning to the third integral on the right-hand side of~\eqref{eq:G1First}, note that $1/2<c_2+\sigma_4<3/2$ and $|\sigma_4+c_2+\ie(t+u)|\geq T-\mathfrak{b}_1(\log{T})^{\mathfrak{b}_2}$ since $t\in[T,2T]$ and $|u|\leq \mathfrak{b}_1\left(\log{T}\right)^{\mathfrak{b}_2}$. Also, $-1<2(1-\sigma_4-c_2)<1$. Similarly as before, the CBS inequality and Lemmas~\ref{lem:MV},~\ref{lem:d2} and~\ref{lem:chi} imply
\begin{equation}
\label{eq:intg4c1}
\int_{T}^{2T}\left|g_4(\sigma_4+\ie t)\right|^{2}\dif{t} \leq \mathfrak{g}_{2}\left(T_0,\sigma_4,c_2,\bm{\mathfrak{b}}\right)T\log^{3}{T},
\end{equation}
where
\begin{flalign}
\label{eq:fg2}
\mathfrak{g}_{2}\left(T_0,\sigma_4,c_2,\bm{\mathfrak{b}}\right) &\de \frac{1}{\left(\mathcal{U}^{-}(T_0)\right)^{2}}\left(\frac{1}{\pi}+\frac{1}{2\pi T_0}\right)^{2\sigma_4-1}\left(\frac{1}{2\pi}-\frac{\mathfrak{b}_1\left(\log{T_0}\right)^{\mathfrak{b}_2}}{2\pi T_0}\right)^{2(1-2(\sigma_4+c_2))}\times \nonumber \\
&\times \left(\mathcal{U}^{+}\left(T_0-\mathfrak{b}_1\left(\log{T_0}\right)^{\mathfrak{b}_2}\right)\right)^{4}\left(\int_{-\infty}^{\infty}\left|\Gamma(c_2+\ie u)\right|\dif{u}\right)^{2}\times \nonumber \\
&\times\left(\frac{\eta(T_0)\mathcal{D}_{4}\left(T_0,2\left(1-\sigma_4-c_2\right)\right)}{(2\pi)^2}+\frac{m_0\mathcal{D}_{4}\left(T_0,1-2\left(\sigma_4+c_2\right)\right)}{2\pi}\right).
\end{flalign}
Here, $\mathcal{D}_{4}(x_0,\sigma)$ and $\mathcal{U}^{\pm}(y_0)$ are defined by~\eqref{eq:D4} and~\eqref{eq:chiBound}, respectively. Note that~\eqref{eq:intg3c1} and~\eqref{eq:intg4c1} also hold for $\sigma_4=1/2$. This fact will be used in Section~\ref{sec:IVpart}. Using~\eqref{eq:intg3c1} and~\eqref{eq:intg4c1} in~\eqref{eq:G1First}, together with Lemmas~\ref{lem:MV} and~\ref{lem:d2} to estimate the first integral on the right-hand side of~\eqref{eq:G1First}, gives
\begin{equation}
\label{eq:G1main}
\left|\int_{\sigma_4+\ie T}^{\sigma_4+2\ie T}G_1(z)\dif{z}\right| \leq \mathfrak{G}_3\left(T_0,\sigma_4,\bm{c},\bm{\mathfrak{a}},\bm{\mathfrak{b}}\right)T\log^{3}{T},
\end{equation}
where
\begin{flalign}
\label{eq:fG3}
\mathfrak{G}_3\left(T_0,\sigma_4,\bm{c},\bm{\mathfrak{a}},\bm{\mathfrak{b}}\right) &\de \frac{(2\pi)^{\sigma_4}}{\mathcal{U}^{-}(T_0)}\left(\frac{\eta(T_0)}{2\pi}\mathcal{D}_{4}(T_0,2(1-\sigma_4))+ m_0\mathcal{D}_{4}(T_0,1-2\sigma_4)\right)^{\frac{1}{2}}\times \nonumber \\
&\times\left(\mathfrak{g}_1\left(T_0,\sigma_4,c_1,\bm{\mathfrak{a}}\right)^{\frac{1}{2}}+\mathfrak{g}_{2}\left(T_0,\sigma_4,c_2,\bm{\mathfrak{b}}\right)^{\frac{1}{2}}\right).
\end{flalign}
Here, $\mathcal{D}_{4}(x_0,\sigma)$, $\mathcal{U}^{-}(y_0)$, $\mathfrak{g}_1\left(T_0,\sigma_4,c_1,\bm{\mathfrak{a}}\right)$ and $\mathfrak{g}_{2}\left(T_0,\sigma_4,c_2,\bm{\mathfrak{b}}\right)$ are defined by~\eqref{eq:D4},~\eqref{eq:chiBound},~\eqref{eq:fg1} and~\eqref{eq:fg2}, respectively. Moving to the second term on the right-hand side of~\eqref{eq:thirdterm2}, we have
\begin{multline*}
\left|\int_{\sigma_5+\ie T}^{\sigma_5+2\ie T}G_2(z)\dif{z}\right| \leq \frac{1}{\mathcal{U}^{-}(T_0)}\left(\frac{2\pi}{T}\right)^{\frac{1}{2}-\sigma_5}\left(\int_{T}^{2T}\left|\sum_{n\leq T}\frac{d(n)}{n^{\sigma_5+\ie t}}\right|^{2}\dif{t}\right)^{\frac{1}{2}} \times \\
\times\left(\left(\int_{T}^{2T}\left|g_3(\sigma_5+\ie t)\right|^{2}\dif{t}\right)^{\frac{1}{2}}+\left(\int_{T}^{2T}\left|g_4(\sigma_5+\ie t)\right|^{2}\dif{t}\right)^{\frac{1}{2}}\right)
\end{multline*}
by the CBS inequality and Lemma~\ref{lem:chi} since $0<\sigma_5<1/2$. Estimation of the second and the third integral is similar as before. We obtain
\[
\int_{T}^{2T}\left|g_3(\sigma_5+\ie t)\right|^{2}\dif{t} \leq \mathfrak{g}_3\left(T_0,\sigma_5,c_1,\bm{\mathfrak{a}}\right)T\log^{3}{T}
\]
and
\[
\int_{T}^{2T}\left|g_4(\sigma_5+\ie t)\right|^{2}\dif{t} \leq \mathfrak{g}_4\left(T_0,\sigma_5,c_2,\bm{\mathfrak{b}}\right)T\log^{3}{T}
\]
by Lemmas~\ref{lem:MV},~\ref{lem:d2} and~\ref{lem:chi}, where
\begin{flalign}
\label{eq:fg3}
\mathfrak{g}_3\left(T_0,\sigma_5,c_1,\bm{\mathfrak{a}}\right) &\de \frac{1}{(2\pi)^{1+2\sigma_5}\left(\mathcal{U}^{-}(T_0)\right)^{2}}\left(\frac{1}{\pi}+\frac{1+\mathfrak{a}_1\left(\log{T_0}\right)^{\mathfrak{a}_2}}{2\pi T_0}\right)^{2\left(1-2\sigma_5\right)+4|c_1|} \times \nonumber \\
&\times \left(\mathcal{U}^{+}\left(T_0-\mathfrak{a}_1(\log{T_0})^{\mathfrak{a}_2}\right)\right)^{4}\left(\int_{-\infty}^{\infty}\left|\Gamma(c_1+\ie u)\right|\dif{u}\right)^{2}\times \nonumber \\
&\times\left(\frac{\eta(T_0)\mathcal{D}_5(T_0,2-2\sigma_5+2|c_1|)}{\left(1-2\sigma_5+2|c_1|\right)^{4}}+\frac{2\pi m_0\mathcal{D}_5(T_0,1-2\sigma_5+2|c_1|)}{\left(2|c_1|-2\sigma_5\right)^{4}}\right)
\end{flalign}
and
\begin{flalign}
\label{eq:fg4}
\mathfrak{g}_{4}\left(T_0,\sigma_5,c_2,\bm{\mathfrak{b}}\right) &\de \frac{1}{(2\pi)^{1+2\sigma_5}\left(\mathcal{U}^{-}(T_0)\right)^{2}}\left(\frac{1}{2\pi}-\frac{\mathfrak{b}_1\left(\log{T_0}\right)^{\mathfrak{b}_2}}{2\pi T_0}\right)^{2(1-2(\sigma_5+c_2))}\times \nonumber \\
&\times \left(\mathcal{U}^{+}\left(T_0-\mathfrak{b}_1\left(\log{T_0}\right)^{\mathfrak{b}_2}\right)\right)^{4}\left(\int_{-\infty}^{\infty}\left|\Gamma(c_2+\ie u)\right|\dif{u}\right)^{2}\times \nonumber \\
&\times\left(\eta(T_0)\mathcal{D}_{4}\left(T_0,2\left(1-\sigma_5-c_2\right)\right)+2\pi m_0\mathcal{D}_{4}\left(T_0,1-2\left(\sigma_5+c_2\right)\right)\right).
\end{flalign}
Here, $\mathcal{D}_{5}(x_0,\sigma)$, $\mathcal{D}_{4}(x_0,\sigma)$ and $\mathcal{U}^{\pm}(y_0)$ are defined by~\eqref{eq:D5},~\eqref{eq:D4} and~\eqref{eq:chiBound}, respectively. Therefore,
\begin{equation}
\label{eq:G2main}
\left|\int_{\sigma_5+\ie T}^{\sigma_5+2\ie T}G_2(z)\dif{z}\right| \leq \mathfrak{G}_4\left(T_0,\sigma_5,\bm{c},\bm{\mathfrak{a}},\bm{\mathfrak{b}}\right)T\log^{3}{T},
\end{equation}
where
\begin{flalign}
\label{eq:fG4}
\mathfrak{G}_4\left(T_0,\sigma_5,\bm{c},\bm{\mathfrak{a}},\bm{\mathfrak{b}}\right) &\de \frac{(2\pi)^{\frac{1}{2}-\sigma_5}}{\mathcal{U}^{-}(T_0)}\left(\eta(T_0)\mathcal{D}_{4}(T_0,2\sigma_5)+ 2\pi m_0\mathcal{D}_{4}(T_0,2\sigma_{5}-1)\right)^{\frac{1}{2}}\times \nonumber \\
&\times\left(\mathfrak{g}_3\left(T_0,\sigma_5,c_1,\bm{\mathfrak{a}}\right)^{\frac{1}{2}}+\mathfrak{g}_{4}\left(T_0,\sigma_5,c_2,\bm{\mathfrak{b}}\right)^{\frac{1}{2}}\right).
\end{flalign}
Here, $\mathcal{D}_{4}(x_0,\sigma)$, $\mathcal{U}^{-}(y_0)$, $\mathfrak{g}_3\left(T_0,\sigma_5,c_1,\bm{\mathfrak{a}}\right)$ and $\mathfrak{g}_{4}\left(T_0,\sigma_5,c_2,\bm{\mathfrak{b}}\right)$ are defined by~\eqref{eq:D4},~\eqref{eq:chiBound},~\eqref{eq:fg3} and~\eqref{eq:fg4}, respectively. It remains to estimate the last two terms on the right-hand side of~\eqref{eq:thirdterm2}. Beginning with the first such term, note that by writing $z=\sigma+\ie T$ in the case of the first integral and $z=\sigma+2\ie T$ in the case of the second integral, we have $\sigma\in[1/2,\sigma_4]$, implying $-1/2<\sigma+c_1<0$, $1/2<\sigma+c_2<3/2$, $|\sigma+c_1+\ie(u+kT)|\leq 1/2+kT+\mathfrak{a}_1(\log{T})^{\mathfrak{a}_2}$ for $k>0$ and $|u|\leq \mathfrak{a}_1(\log{T})^{\mathfrak{a}_2}$, and $|\sigma+c_2+\ie(u+kT)|\geq kT-\mathfrak{b}_1(\log{T})^{\mathfrak{b}_2}$ for $k>0$ and $|u|\leq\mathfrak{b}_1(\log{T})^{\mathfrak{b}_2}$. Define 
\begin{multline}
\label{eq:fg5}
\mathfrak{g}_{5}\left(T_0,\sigma_4,c_1,\bm{\mathfrak{a}}\right) \de \int_{\frac{1}{2}}^{\sigma_4}\left(\left(1+\frac{1}{T_0}\right)^{\sigma-\frac{1}{2}}\left(\frac{1}{2\pi}+\frac{\frac{1}{2}+\mathfrak{a}_{1}(\log{T_0})^{\mathfrak{a}_{2}}}{2\pi T_0}\right)^{1-2(\sigma-|c_1|)} \right. \\
\left.+\left(1+\frac{1}{2T_0}\right)^{\sigma-\frac{1}{2}}\left(\frac{1}{\pi}+\frac{\frac{1}{2}+\mathfrak{a}_{1}(\log{T_0})^{\mathfrak{a}_{2}}}{2\pi T_0}\right)^{1-2(\sigma-|c_1|)}\right) 
\frac{\mathcal{D}_{2}(T_0,1-\sigma+|c_1|)}{\sigma\left(\sigma-|c_1|\right)^{2}}\dif{\sigma},    
\end{multline}
where $\mathcal{D}_2(x_0,\sigma)$ is defined by~\eqref{eq:D2}, and
\begin{multline}
\label{eq:fg6}
\mathfrak{g}_{6}\left(T_0,\sigma_4,c_2,\bm{\mathfrak{b}}\right) \de \int_{\frac{1}{2}}^{\sigma_4}\left(\left(1+\frac{1}{T_0}\right)^{\sigma-\frac{1}{2}}\left(\frac{1}{2\pi}-\frac{\mathfrak{b}_{1}(\log{T_0})^{\mathfrak{b}_{2}}}{2\pi T_0}\right)^{1-2(\sigma+c_2)}\right.\\
\left.+\left(1+\frac{1}{2T_0}\right)^{\sigma-\frac{1}{2}}\left(\frac{1}{\pi}-\frac{\mathfrak{b}_{1}(\log{T_0})^{\mathfrak{b}_{2}}}{2\pi T_0}\right)^{1-2(\sigma+c_2)}\right)\frac{\dif{\sigma}}{\sigma(\sigma+c_2)}.  
\end{multline}
By Lemmas~\ref{lem:d} and~\ref{lem:chi} we thus have
\begin{equation}
\label{eq:G1res}
\left|\left(\int_{\frac{1}{2}+\ie T}^{\sigma_4+\ie T}+\int_{\sigma_4+2\ie T}^{\frac{1}{2}+2\ie T}\right)G_1(z)\dif{z}\right| \leq \mathfrak{G}_{5}\left(T_0,\sigma_4,\bm{c},\bm{\mathfrak{a}},\bm{\mathfrak{b}}\right)T\log^{2}{T},
\end{equation}
where
\begin{multline}
\label{eq:fG5}
\mathfrak{G}_{5}\left(T_0,\sigma_4,\bm{c},\bm{\mathfrak{a}},\bm{\mathfrak{b}}\right) \de \frac{\mathcal{D}_{1}(T_0,0)}{2\pi}\left(\frac{\mathcal{U}^{+}\left(T_0-\mathfrak{a}_{1}(\log{T_0})^{\mathfrak{a}_{2}}\right)}{\mathcal{U}^{-}(T_0)}\right)^{2}\mathfrak{g}_{5}\left(T_0,\sigma_4,c_1,\bm{\mathfrak{a}}\right)\int_{-\infty}^{\infty}\left|\Gamma(c_1+\ie u)\right|\dif{u} \\
+ \frac{\mathcal{D}_{1}(T_0,0)\mathcal{D}_{1}(T_0,-1/2)}{2\pi}\left(\frac{\mathcal{U}^{+}\left(T_0-\mathfrak{b}_{1}(\log{T_0})^{\mathfrak{b}_{2}}\right)}{\mathcal{U}^{-}(T_0)}\right)^{2}\mathfrak{g}_{6}\left(T_0,\sigma_4,c_2,\bm{\mathfrak{b}}\right)\int_{-\infty}^{\infty}\left|\Gamma(c_2+\ie u)\right|\dif{u}.
\end{multline}
Here, $\mathcal{D}_{1}(x_0,\sigma_0)$, $\mathcal{U}^{\pm}(y_0)$, $\mathfrak{g}_{5}\left(T_0,\sigma_4,c_1,\bm{\mathfrak{a}}\right)$ and $\mathfrak{g}_{6}\left(T_0,\sigma_4,c_2,\bm{\mathfrak{b}}\right)$ are defined by~\eqref{eq:D1},~\eqref{eq:chiBound},~\eqref{eq:fg5} and~\eqref{eq:fg6}, respectively. Turning to the second term on the right-hand side of~\eqref{eq:thirdterm2}, we have, similarly as before, that $\sigma\in[\sigma_5,1/2]$, implying $-1<\sigma+c_1<0$, $1/2<\sigma+c_2<1$, $|\sigma+c_1+\ie(u+kT)|\leq 1+kT+\mathfrak{a}_1(\log{T})^{\mathfrak{a}_2}$ for $k>0$ and $|u|\leq \mathfrak{a}_1(\log{T})^{\mathfrak{a}_2}$, and $|\sigma+c_2+\ie(u+kT)|\geq kT-\mathfrak{b}_1(\log{T})^{\mathfrak{b}_2}$ for $k>0$ and $|u|\leq \mathfrak{b}_1(\log{T})^{\mathfrak{b}_2}$. Define
\begin{multline}
\label{eq:fg7}
\mathfrak{g}_{7}\left(T_0,\sigma_5,c_1,\bm{\mathfrak{a}}\right) \de \int_{\sigma_5}^{\frac{1}{2}}\left((2\pi)^{1-2\sigma}\left(\frac{1}{2\pi}+\frac{1+\mathfrak{a}_{1}(\log{T_0})^{\mathfrak{a}_2}}{2\pi T_0}\right)^{1-2(\sigma-|c_1|)}\right. \\
\left.+\pi^{1-2\sigma}\left(\frac{1}{\pi}+\frac{1+\mathfrak{a}_1(\log{T_0})^{\mathfrak{a}_2}}{2\pi T_0}\right)^{1-2(\sigma-|c_1|)}\right)\frac{\mathcal{D}_{2}(T_0,1-\sigma+|c_1|)}{(1-\sigma)\left(\sigma-|c_1|\right)^{2}}\dif{\sigma},
\end{multline}
where $\mathcal{D}_2(x_0,\sigma)$ is defined by~\eqref{eq:D2}, and
\begin{multline}
\label{eq:fg8}
\mathfrak{g}_{8}\left(T_0,\sigma_5,c_2,\bm{\mathfrak{b}}\right) \de \int_{\sigma_5}^{\frac{1}{2}}\left((2\pi)^{1-2\sigma}\left(\frac{1}{2\pi}-\frac{\mathfrak{b}_1(\log{T_0})^{\mathfrak{b}_2}}{2\pi T_0}\right)^{1-2(\sigma+c_2)}\right.\\
\left.+\pi^{1-2\sigma}\left(\frac{1}{\pi}-\frac{\mathfrak{b}_1(\log{T_0})^{\mathfrak{b}_2}}{2\pi T_0}\right)^{1-2(\sigma+c_2)}\right)\frac{\dif{\sigma}}{(1-\sigma)(\sigma+c_2)}.
\end{multline}
By Lemmas~\ref{lem:d} and~\ref{lem:chi} we thus have
\begin{equation}
\label{eq:G2res}
\left|\left(\int_{\frac{1}{2}+\ie T}^{\sigma_5+\ie T}+\int_{\sigma_5+2\ie T}^{\frac{1}{2}+2\ie T}\right)G_{2}(z)\dif{z}\right| \leq \mathfrak{G}_{6}\left(T_0,\sigma_5,\bm{c},\bm{\mathfrak{a}},\bm{\mathfrak{b}}\right)T\log^{2}{T},
\end{equation}
where
\begin{multline}
\label{eq:fG6}
\mathfrak{G}_{6}\left(T_0,\sigma_5,\bm{c},\bm{\mathfrak{a}},\bm{\mathfrak{b}}\right) \de \frac{\mathcal{D}_{1}(T_0,0)}{2\pi}\left(\frac{\mathcal{U}^{+}\left(T_0-\mathfrak{a}_{1}(\log{T_0})^{\mathfrak{a}_{2}}\right)}{\mathcal{U}^{-}(T_0)}\right)^{2}\mathfrak{g}_{7}\left(T_0,\sigma_5,c_1,\bm{\mathfrak{a}}\right)\int_{-\infty}^{\infty}\left|\Gamma(c_1+\ie u)\right|\dif{u} \\
+ \frac{1}{2\pi}\left(\frac{\mathcal{D}_{1}(T_0,0)\cdot\mathcal{U}^{+}\left(T_0-\mathfrak{b}_{1}(\log{T_0})^{\mathfrak{b}_{2}}\right)}{\mathcal{U}^{-}(T_0)}\right)^{2}\mathfrak{g}_{8}\left(T_0,\sigma_5,c_2,\bm{\mathfrak{b}}\right)\int_{-\infty}^{\infty}\left|\Gamma(c_2+\ie u)\right|\dif{u}.
\end{multline}
Here, $\mathcal{D}_{1}(x_0,\sigma_0)$, $\mathcal{U}^{\pm}(y_0)$, $\mathfrak{g}_{7}\left(T_0,\sigma_5,c_1,\bm{\mathfrak{a}}\right)$ and $\mathfrak{g}_{8}\left(T_0,\sigma_5,c_2,\bm{\mathfrak{b}}\right)$ are defined by~\eqref{eq:D1},~\eqref{eq:chiBound},~\eqref{eq:fg7} and~\eqref{eq:fg8}, respectively.

Finally, it remains to collect the relevant estimates in order to bound the left-hand sides of~\eqref{eq:thirdterm1} and~\eqref{eq:thirdterm2}. For the former equality, we are using~\eqref{eq:F1F2main} and~\eqref{eq:F1F2res}, while for the latter equality we are using~\eqref{eq:G1main},~\eqref{eq:G2main},~\eqref{eq:G1res} and~\eqref{eq:G2res}. The final result is that
\begin{equation}
\label{eq:J22}   
2\sum_{k=3}^{6}\left|\int_{T}^{2T}\left(J_1(t)+J_2(t)\right)J_{k}(t)\dif{t}\right| \leq 2\mathcal{J}_{2,2}\left(T_0,\bm{\sigma}',\bm{\sigma}'',\bm{c},\bm{\mathfrak{a}},\bm{\mathfrak{b}}\right)T\log^{3}{T},
\end{equation}
where
\begin{multline}
\label{eq:cJ22} 
\mathcal{J}_{2,2}\left(T_0,\bm{\sigma}',\bm{\sigma}'',\bm{c},\bm{\mathfrak{a}},\bm{\mathfrak{b}}\right) = \mathfrak{G}_1\left(T_0,\bm{\sigma}'\right) + \mathfrak{G}_{3}\left(T_0,\sigma_4,\bm{c},\bm{\mathfrak{a}},\bm{\mathfrak{b}}\right) + \mathfrak{G}_{4}\left(T_0,\sigma_5,\bm{c},\bm{\mathfrak{a}},\bm{\mathfrak{b}}\right) \\ 
+ \frac{1}{\log{T_0}}\left(\mathfrak{G}_{2}\left(T_0,\bm{\sigma}'\right)+\mathfrak{G}_{5}\left(T_0,\sigma_4,\bm{c},\bm{\mathfrak{a}},\bm{\mathfrak{b}}\right)+\mathfrak{G}_{6}\left(T_0,\sigma_5,\bm{c},\bm{\mathfrak{a}},\bm{\mathfrak{b}}\right)\right).
\end{multline}
Here, the functions $\mathfrak{G}_{1},\ldots,\mathfrak{G}_{6}$ are defined by~\eqref{eq:fG1},~\eqref{eq:fG2},~\eqref{eq:fG3},~\eqref{eq:fG4},~\eqref{eq:fG5} and~\eqref{eq:fG6}, respectively.

\subsubsection{IV. part}
\label{sec:IVpart}

By Lemmas~\ref{lem:MV},~\ref{lem:d2Exp2} and~\ref{lem:d2Exp3} we have
\begin{equation}
\label{eq:cJ3}
\frac{1}{T\log^{3}{T}}\int_{T}^{2T}\left|J_3(t)\right|^{2}\dif{t} \leq \mathcal{J}_{3}(T_0) \de \left(\eta(T_0)+2\pi m_0\right)\mathcal{D}_{7}(T_0),
\end{equation}
where $\mathcal{D}_7(x_0)$ is defined by~\eqref{eq:D7}, and
\begin{equation}
\label{eq:cJ4}
\frac{1}{T\log^{3}{T}}\int_{T}^{2T}\left|J_4(t)\right|^{2}\dif{t} \leq \mathcal{J}_{4}(T_0) \de \eta(T_0)\mathcal{D}_{4}(T_0,-1)+2\pi m_0\mathcal{D}_{4}(T_0,-2),
\end{equation}
where $\mathcal{D}_{4}(x_0,\sigma)$ is defined by~\eqref{eq:D4}. Turning to the integral of $|J_5(t)|^2$, observe that
\begin{equation}
\label{eq:cJ5}    
\frac{1}{T\log^{3}{T}}\int_{T}^{2T}\left|J_5(t)\right|^{2}\dif{t} = \frac{1}{T\log^{3}{T}}\int_{T}^{2T}\left|g_3\left(\frac{1}{2}+\ie t\right)\right|^{2}\dif{t} \leq \mathcal{J}_{5}\left(T_0,c_1,\bm{\mathfrak{a}}\right)\de \mathfrak{g}_1\left(T_0,\frac{1}{2},c_1,\bm{\mathfrak{a}}\right)
\end{equation}
by~\eqref{eq:intg3c1}, where $\mathfrak{g}_1\left(T_0,\sigma_4,c_1,\bm{\mathfrak{a}}\right)$ is defined by~\eqref{eq:fg1}. Similarly,
\begin{equation}
\label{eq:cJ6}
\frac{1}{T\log^{3}{T}}\int_{T}^{2T}|J_6(t)|^{2}\dif{t} = \frac{1}{T\log^{3}{T}}\int_{T}^{2T}\left|g_4\left(\frac{1}{2}+\ie t\right)\right|^{2}\dif{t} \leq \mathcal{J}_{6}\left(T_0,c_2,\bm{\mathfrak{b}}\right)\de \mathfrak{g}_{2}\left(T_0,\frac{1}{2},c_2,\bm{\mathfrak{b}}\right)
\end{equation}
by~\eqref{eq:intg4c1}, where $\mathfrak{g}_{2}\left(T_0,\sigma_4,c_2,\bm{\mathfrak{b}}\right)$ is defined by~\eqref{eq:fg2}.

\subsubsection{V.~part}
\label{sec:Vpart}

Using estimates~\eqref{eq:cJ3}--\eqref{eq:cJ6}, we obtain
\begin{equation}
\label{eq:cJMixed}
2\mathop{\sum\sum}_{3\leq n<m\leq 6}\int_{T}^{2T}\left|J_{n}(t)J_{m}(t)\right|\dif{t} \leq \left(2\mathop{\sum\sum}_{3\leq n<m\leq 6}\left(\mathcal{J}_{n}\mathcal{J}_m\right)^{\frac{1}{2}}\right)T\log^{3}{T}
\end{equation}
by the CBS inequality, where we borrowed the shorter notation for $\mathcal{J}_n$ from Theorem~\ref{thm:4thPMExplicitGeneral}.

\subsubsection{VI.~part}
\label{sec:VIpart}

By Lemma~\ref{lem:chi} we have
\[
\left|\chi\left(\frac{1}{2}+c_1+\ie (t+u)\right)\right| \leq (2T)^{|c_1|}\left(0.4+\frac{16}{\left(2T_0+\mathfrak{a}_1\left(\log{T_0}\right)^{\mathfrak{a}_2}\right)^{|c_1|}}\right)\left(1+\frac{|u|}{2T_0}\right)^{|c_1|}
\]
and
\[
\left|\Gamma(c_1+\ie u)\right| \leq \sqrt{2\pi}\exp{\left(\frac{1}{2\mathfrak{a}_1^2\left(\log{T_0}\right)^{2\mathfrak{a}_2}}\right)}T^{-\mathfrak{a}_1\left(\log{T}\right)^{\mathfrak{a}_{2}-1}}\frac{e^{-|u|/2}}{|u|^{|c_1|+1/2}}
\]
since $c_1\in(-1,-1/2)$, $t\in[T,2T]$ and $|u|\geq\mathfrak{a}_1(\log{T})^{\mathfrak{a}_2}$. Therefore,
\begin{equation}
\label{eq:cJ01Est}
\frac{1}{2\pi}\left|\left(\int_{-\infty}^{-\mathfrak{a}_1\left(\log{T}\right)^{\mathfrak{a}_2}}+\int_{\mathfrak{a}_1\left(\log{T}\right)^{\mathfrak{a}_2}}^{\infty}\right)P(0,1,c_1;u)\dif{u}\right| \leq \mathcal{J}_{01}\left(T_0,c_1,\bm{\mathfrak{a}}\right)T^{\frac{1}{2}-\mathfrak{a}_1\left(\log{T}\right)^{\mathfrak{a}_{2}-1}}\log{T}
\end{equation}
by Lemma~\ref{lem:d}, where
\begin{multline}
\label{eq:cJ01}
\mathcal{J}_{01}\left(T_0,c_1,\bm{\mathfrak{a}}\right) \de 2^{2|c_1|}\sqrt{\frac{2}{\pi}}\left(0.4+\frac{16}{\left(2T_0+\mathfrak{a}_1\left(\log{T_0}\right)^{\mathfrak{a}_2}\right)^{|c_1|}}\right)^{2}\exp{\left(\frac{1}{2\mathfrak{a}_1^2\left(\log{T_0}\right)^{2\mathfrak{a}_2}}\right)} \times \\
\times \frac{\mathcal{D}_{2}\left(T_0,\frac{1}{2}+|c_1|\right)}{\left(|c_1|-\frac{1}{2}\right)^{2}} \int_{4}^{\infty}\left(1+\frac{u}{2T_0}\right)^{2|c_1|}\frac{e^{-\frac{1}{2}u}}{u^{|c_1|+\frac{1}{2}}}\dif{u}
\end{multline}
since $\mathfrak{a}_1\left(\log{T_0}\right)^{\mathfrak{a}_{2}}\geq\log{T_0}\geq 4$. Here, $\mathcal{D}_{2}(x_0,\sigma_0)$ is defined by~\eqref{eq:D2}. By Lemma~\ref{lem:chi} we also have
\[
\left|\chi\left(\frac{1}{2}+c_1+\ie (t+u)\right)\right| \leq \frac{6.31}{\frac{1}{2}-c_2}
\]
and
\[
\left|\Gamma(c_2+\ie u)\right| \leq \sqrt{2\pi}\exp{\left(\frac{1}{2\mathfrak{b}_1^2\left(\log{T_0}\right)^{2\mathfrak{b}_2}}\right)}\left(\mathfrak{b}_1\left(\log{T_0}\right)^{\mathfrak{b}_{2}}\right)^{c_2-\frac{1}{2}}T^{-\mathfrak{b}_1\left(\log{T}\right)^{\mathfrak{b}_{2}-1}}e^{-\frac{1}{2}|u|}
\]
since $c_2\in(0,1/2)$, $t\in[T,2T]$ and $|u|\geq \mathfrak{b}_1\left(\log{T}\right)^{\mathfrak{b}_{2}}$. Therefore,
\begin{equation}
\label{eq:cJ02Est}
\frac{1}{2\pi}\left|\left(\int_{-\infty}^{-\mathfrak{b}_1\left(\log{T}\right)^{\mathfrak{b}_2}}+\int_{\mathfrak{b}_1\left(\log{T}\right)^{\mathfrak{b}_2}}^{\infty}\right)P(1,0,c_2;u)\dif{u}\right| \leq \mathcal{J}_{02}\left(T_0,c_2,\bm{\mathfrak{b}}\right)T^{\frac{1}{2}+2c_2-\mathfrak{b}_1\left(\log{T}\right)^{\mathfrak{b}_{2}-1}}\log{T}
\end{equation}
by Lemma~\ref{lem:d}, where
\begin{equation}
\label{eq:cJ02}
\mathcal{J}_{02}\left(T_0,c_2,\bm{\mathfrak{b}}\right)\de \frac{2}{e^4}\sqrt{\frac{2}{\pi}}\left(\frac{6.31}{\frac{1}{2}-c_2}\right)^{2}\exp{\left(\frac{1}{2\mathfrak{b}_1^2\left(\log{T_0}\right)^{2\mathfrak{b}_2}}\right)} \left(\mathfrak{b}_1\left(\log{T_0}\right)^{\mathfrak{b}_{2}}\right)^{c_2-\frac{1}{2}}\frac{\mathcal{D}_1(T_0,0)}{\frac{1}{2}+c_2}
\end{equation}
since $\mathfrak{b}_1\left(\log{T_0}\right)^{\mathfrak{b}_{2}}\geq 2\log{T_0}\geq 8$. Here, $\mathcal{D}_{1}(x_0,\sigma_0)$ is defined by~\eqref{eq:D1}. Then we have
\begin{equation}
\label{eq:cJ0Est}    
2\int_{T}^{2T}\left|J_0(t)\right|\cdot\left|\zeta\left(\frac{1}{2}+\ie t\right)\right|^{2}\dif{t} + \int_{T}^{2T}\left|J_0(t)\right|^{2}\dif{t} \leq \mathcal{J}_0\left(T_0,\bm{c},\bm{\mathfrak{a}},\bm{\mathfrak{b}}\right)T\log^{3}{T},
\end{equation}
where
\begin{multline}
\label{eq:cJ0}
\mathcal{J}_0\left(T_0,\bm{c},\bm{\mathfrak{a}},\bm{\mathfrak{b}}\right) \de \frac{7}{2}\left(\left(1+\frac{3}{14}\mathcal{J}_{01}\left(T_0,c_1,\bm{\mathfrak{a}}\right)\right)\mathcal{J}_{01}\left(T_0,c_1,\bm{\mathfrak{a}}\right) T_0^{1-\mathfrak{a}_1(\log{T_0})^{\mathfrak{a}_{2}-1}}\right. \\
\left.+\left(1+\frac{3}{14}\mathcal{J}_{02}\left(T_0,c_2,\bm{\mathfrak{b}}\right)\right)\mathcal{J}_{02}\left(T_0,c_2,\bm{\mathfrak{b}}\right) T_0^{2-\mathfrak{b}_1(\log{T_0})^{\mathfrak{b}_{2}-1}}+12.1T_0 e^{-\frac{\pi}{2}T_0}\right),
\end{multline}
and $\mathcal{J}_{01}\left(T_0,c_1,\bm{\mathfrak{a}}\right)$ and $\mathcal{J}_{02}\left(T_0,c_2,\bm{\mathfrak{b}}\right)$ are defined by~\eqref{eq:cJ01} and~\eqref{eq:cJ02}, respectively. In the derivation of~\eqref{eq:cJ0Est}, we have used inequalities~\eqref{eq:cJ01Est} and~\eqref{eq:cJ02Est}, together with Lemma~\ref{lem:res} and $\left|\zeta(1/2+\ie t)\right|\leq  0.618 t^{\frac{1}{6}} \log t\leq (2T)^{1/6}\log{(2T)}$ for $t\in[T,2T]$, see~\cite[Theorem~1.1]{HiaryPatelYang}.

\subsubsection{Proof of estimate~\eqref{eq:4thPMExplicitGeneral}}

First, by~\eqref{eq:J1} we have
\[
\left|\mathcal{M}_{2}(T,2T)-\frac{1}{2\pi^2}T\log^{4}{T}\right| \leq 
\mathcal{J}_{1}(T_0)T\log^{3}{T} + \left|\mathcal{M}_{2}(T,2T)-2\int_{T}^{2T}\left|J_1(t)\right|^{2}\dif{t}\right|,
\]
where $\mathcal{J}_{1}(T_0)$ is defined by~\eqref{eq:cJ1}. We are using~\eqref{eq:4thPMmainEst} and~\eqref{eq:4thPMtermR} in order to estimate the last term in the last inequality. Then~\eqref{eq:4thPMExplicitGeneral} follows by~\eqref{eq:J1J2square},~\eqref{eq:J22},~\eqref{eq:cJ3},~\eqref{eq:cJ4},~\eqref{eq:cJ5},~\eqref{eq:cJ6},~\eqref{eq:cJMixed} and~\eqref{eq:cJ0Est}.

\subsubsection{Proof of estimate~\eqref{eq:4thPMExplicitGeneral2}}

We are using~\eqref{eq:4thPMmainEst2}. By~\eqref{eq:J1} we have
\[
4\int_{T}^{2T}\left|J_1(t)\right|^{2}\dif{t} \leq \frac{1}{\pi^2}T\log^{4}{T} + 2\mathcal{J}_{1}(T_0)T\log^{3}{T}.
\]
Also, the above inequality, together with the CBS inequality and the inequalities from Section~\ref{sec:IVpart}, implies
\begin{flalign*}
4\sum_{k=3}^{6}\int_{T}^{2T}\left|J_1(t)J_{k}(t)\right|\dif{t} &\leq 2\left(4\int_{T}^{2T}\left|J_1(t)\right|^{2}\dif{t}\right)^{\frac{1}{2}}\sum_{k=3}^{6}\left(\int_{T}^{2T}\left|J_k(t)\right|^{2}\dif{t}\right)^{\frac{1}{2}} \\
&\leq 2\left(T\log^{3}{T}\right)\sum_{k=3}^{6}\left(\frac{\mathcal{J}_{k}}{\pi^2}\log{T}+2\mathcal{J}_{1}\mathcal{J}_{k}\right)^{\frac{1}{2}}.
\end{flalign*}
Other terms in~\eqref{eq:4thPMmainEst2} are estimated as in the proof of~\eqref{eq:4thPMExplicitGeneral}, that is, we are using estimates from Sections~\ref{sec:IVpart},~\ref{sec:Vpart} and~\ref{sec:VIpart}. The proof is thus complete.

\section{Appendix}
\label{sec:appendix}

The main purpose of this section is to state an improved zero density estimate from~\cite{KADIRI201822}. Lemma \ref{lem:fxone} can be used instead of~\cite[Lemma~4.3]{KADIRI201822}, which then implies
\[
N(\sigma,T) \leq \frac{\mathcal{C}_{1}}{ 2 \pi d}T^{\frac{8}{3}\left(1-\sigma\right)}  \log (kT)\left(\log{T}\right)^{5-4\sigma}  + \frac{\mathcal{C}_{2}}{2 \pi d}\log^{2}{T} 
\]
for $T\geq H_0=3.0610046\cdot10^{10}$, where $10^9/H_0\leq k\leq 1$, $d>0$, and $\mathcal{C}_{1}=\mathcal{C}_{1}\left(\alpha,d,\delta,k,H_0-1,\sigma\right)$ and $\mathcal{C}_{2}=\mathcal{C}_{2}\left(d,\eta,k,H_0-1,\mu,\sigma\right)$ are functions defined by~\cite[Equations~(4.72) and~(4.73)]{KADIRI201822} and depend on several variables, see~\cite[Lemma~4.14]{KADIRI201822} and~\cite[pp.~1223--1224]{CullyJohnstonRvM} for details. We choose $k=1$, $\mu=1.2363$ and $\eta=0.25345$. The parameters $\alpha$, $\delta$, and $d$ are calculated such that the value for $\mathcal{C}_1\left(\alpha,d,\delta,k,H_0-1,\sigma_2\right)/(2\pi d)$ is the smallest possible for each of $\sigma_2$, see Table~\ref{tab:KLN}. The final result is that
\begin{equation}
\label{eq:KLN}
N(\sigma,T) \leq A\cdot T^{\frac{8}{3}\left(1-\sigma\right)}\left(\log{T}\right)^{2(3-2\sigma)}  + B\cdot\log^{2}{T}
\end{equation}
for $T\geq 3\cdot10^{12}$ and $\sigma\in[\sigma_1,\sigma_2]$, where the values for $A$ and $B$ are given by Table~\ref{tab:KLN}.

\begin{table}[h]
\centering
\footnotesize{
\begin{tabular}{llllllll} 
\toprule
$\sigma_1$ & $\sigma_2$ & $A$ & $B$ & $\alpha$ & $\delta$ & $d$ \\ 
\midrule
$0.62500$ & $0.65625$ & $2.900$ & $5.027$ & $0.242$ & $0.324$ & $0.344$ \\
$0.65625$& $0.68750$ & $3.474$ & $4.789$ & $0.224$ & $0.324$ & $0.343$ \\
$0.68750$& $0.71875$ & $4.149$ & $4.536$ & $0.206$ & $0.324$ & $0.343$ \\
$0.71875$ & $0.75000$ & $4.940$ & $4.296$ & $0.189$ & $0.324$ & $0.342$ \\
$0.75000$ & $0.78125$ & $5.860$ & $4.042$ & $0.171$ & $0.324$ & $0.342$ \\
$0.78125$ & $0.81250$ & $6.925$ & $3.800$ & $0.153$ & $0.324$ & $0.341$ \\
$0.81250$ & $0.84375$ & $8.148$ & $3.546$ & $0.136$ & $0.324$ & $0.341$ \\
$0.84375$ & $0.87500$ & $9.543$ & $3.301$ & $0.119$ & $0.324$ & $0.340$ \\
$0.87500$ & $0.90625$ & $11.12$ & $3.055$ & $0.102$ & $0.324$ & $0.339$ \\
$0.90625$ & $0.93750$ & $12.88$ & $2.807$ & $0.086$ & $0.324$ & $0.338$  \\
$0.93750$ & $0.96875$ & $14.81$ & $2.558$ & $0.070$ & $0.325$ & $0.337$ \\
$0.96875$ & $0.97500$ & $15.21$ & $2.301$ & $0.067$ & $0.325$ & $0.337$ \\
$0.97500$ & $0.98125$ & $15.62$ & $2.249$ & $0.064$ & $0.325$ & $0.337$ \\
$0.98125$ &$0.98750$ & $16.04$ & $2.204$ & $0.061$ & $0.325$ & $0.336$ \\
$0.98750$ & $0.99375$ & $16.46$ & $2.153$ & $0.058$ & $0.325$ & $0.336$ \\
$0.99375$ & $1.00000$ & $16.89$ & $2.101$ & $0.056$ & $0.325$ & $0.336$ \\
\bottomrule
\end{tabular}}
\caption{The values for $A$ and $B$ from~\eqref{eq:KLN} for the corresponding $\sigma \in [\sigma_1, \sigma_2]$.}
\label{tab:KLN}
\end{table}








\newcommand{\etalchar}[1]{$^{#1}$}
\providecommand{\bysame}{\leavevmode\hbox to3em{\hrulefill}\thinspace}
\providecommand{\MR}{\relax\ifhmode\unskip\space\fi MR }
\providecommand{\MRhref}[2]{%
  \href{http://www.ams.org/mathscinet-getitem?mr=#1}{#2}
}
\providecommand{\href}[2]{#2}

\end{document}